\documentclass[onefignum,onetabnum]{siamart171218}

\usepackage[lmargin=1.2in, rmargin=1.2in, tmargin=1.5in]{geometry}

\usepackage{lipsum}
\usepackage{amsfonts}
\usepackage{graphicx}
\usepackage{epstopdf}
\usepackage{algorithmic}
\ifpdf
  \DeclareGraphicsExtensions{.eps,.pdf,.png,.jpg}
\else
  \DeclareGraphicsExtensions{.eps}
\fi

\usepackage{booktabs} % for professional tables
\usepackage[sort, square, numbers]{natbib}
\usepackage{tabularray}

\usepackage{amssymb}
\usepackage{mathtools}
\usepackage{algorithmic}
\usepackage{enumitem}
\setlist[enumerate]{leftmargin=.5in}
\setlist[itemize]{leftmargin=.5in}

\usepackage{changes}
\definechangesauthor[name=Sharvaj, color=purple]{svk}

\newsiamremark{remark}{Remark}
\newsiamremark{hypothesis}{Hypothesis}
\crefname{hypothesis}{Hypothesis}{Hypotheses}
\newsiamthm{claim}{Claim}

\headers{Log-concave Independent Components}{S. Kubal, C. Campbell, and E. Robeva}

\title{Log-concave Density Estimation with Independent Components\thanks{Submitted \today.
\funding{Elina Robeva was supported by an NSERC Discovery Grant DGECR-2020-00338.}}}

\author{Sharvaj Kubal\thanks{Department of Mathematics, University of British Columbia, Vancouver, Canada (\email{erobeva@math.ubc.ca})}
\and Christian Campbell\footnotemark[2]
\and Elina Robeva\footnotemark[2]}

\usepackage{amsopn}

\newcommand{\bR}{\mathbb{R}}

\newcommand{\bE}{\mathbb{E}}

\newcommand{\norm}[1]{\|#1\|}
\newcommand{\normbig}[1]{\left\|#1\right\|}
\newcommand{\op}[1]{\|#1\|_{\mathrm{op}}}
\newcommand{\fro}[1]{\|#1\|_{F}}

\newcommand{\eigmat}{\mathbf{W}}
\newcommand{\Rpl}{\mathbf{R}}

\newcommand{\VarZ}{\mathbf{\Sigma}_Z}
\newcommand{\VarX}{\mathbf{\Sigma}}

\newcommand{\argmax}{\mathop{\arg \max}}

\ifpdf
\hypersetup{
  pdftitle={Log-concave Density Estimation with Independent Components},
  pdfauthor={S. Kubal, C. Campbell, and E. Robeva}
}
\fi

\begin{document}

\maketitle

% REQUIRED
\begin{abstract}
We propose a method for estimating a log-concave density on $\mathbb R^d$ from samples, under the assumption that there exists an orthogonal transformation that makes the components of the random vector independent. While log-concave density estimation is hard both computationally and statistically, the independent components assumption alleviates both issues, while still maintaining a large non-parametric class. We prove that under mild conditions, at most $\tilde{\mathcal{O}}_{d}(\epsilon^{-4})$ samples (suppressing constants and log factors) suffice for our proposed estimator to be within $\epsilon$ of the original density in squared Hellinger distance. On the computational front, while the usual log-concave maximum likelihood estimate can be obtained via a finite-dimensional convex program, it is slow to compute -- especially in higher dimensions. We demonstrate through numerical experiments that our estimator can be computed efficiently, making it more practical to use.
\end{abstract}

% REQUIRED

% REQUIRED
% \begin{MSCcodes}
% 62G07, 62H12, 62H25
% \end{MSCcodes}

\section{Introduction}
\label{sec: intro}
A log-concave probability density $p$ on $\bR^d$ is of the form $p(x) = e^{-\varphi(x)}$ for a closed, proper, convex function $\varphi: \bR^d \to (-\infty, +\infty]$.  
Log-concave density estimation is a type of non-parametric density estimation dating back to the work of \citet{walther2002detecting}.  
The class of log-concave distributions includes many known parametric families, has attractive statistical properties, and admits maximum likelihood estimates (MLE) that can be computed with established optimization algorithms \cite{samworth2018recent}. Moreover, the MLE is `automatic' in that no tuning parameters (such as kernel bandwidths) need to be chosen -- a feature that can be especially useful in higher dimensions. 
However, log-concave density estimation suffers from the curse of dimensionality -- one requires $n \gtrsim _d\epsilon^{-\frac{d+1}2}$ samples for the log-concave MLE in $\mathbb R^d$ to be $\epsilon$ close to the true density in squared Hellinger distance~\added{\cite{kim2016global}}. Furthermore, the log-concave MLE is too slow to compute for large values of $d$ (e.g., $d\geq 6$ and $n\geq 5000$).

We reduce the family of log-concave densities by assuming that the components of the random vector at hand become independent after a suitable orthogonal transformation is applied {-- an \emph{independent components} assumption}. While the family that we study is somewhat restrictive, it includes all multivariate normal distributions, and is still non-parametric. {Furthermore, the independent components assumption provides a link with independent component analysis (ICA) models, see \citet{samworth2012independent}}.
We propose an easy-to-compute estimator for this family of densities and show that the sample complexity {rate} is no longer exponential in the ambient dimension $d$: with {at most} {$n \gtrsim_d \epsilon^{-4}$} samples (up to constants and log factors), the estimate is within $\epsilon$ of the true density in squared Hellinger distance. 

\subsection{Prior work}
%-- General log-concave overview
Shape-constrained density estimation has a long history. Classes of shape-constrained densities that have previously been studied include non-increasing~\cite{grenander_theory_1956}, convex~\cite{groeneboom2001estimation}, $k$-monotone~\cite{balabdaoui2007estimation}, $s$-concave~\cite{doss2016global} and quasi-concave densities~\cite{koenker2010quasi}. The log-concave class of densities was first studied in dimension 1 by \citet{walther2002detecting}, and a concise description of the log-concave MLE in higher dimensions was given by \citet{cule2010maximum}. 
{Various statistical properties of the procedure were established by \citet{schuhmacher2010consistency}, \citet{carpenter2018near},  \citet{kur2019optimality}, and \citet{barberLocalContinuityLogconcave2021} amongst others.}
Log-concavity has also gained popularity in applications~\cite{bagnoli2005log, samworth2018recent}. 

Subclasses of log-concave densities have also been studied in the recent years for the purposes of reducing the computational and statistical complexities of estimation, and also because they are important on their own. In the 1-dimensional setting, \citet{Kim2018Adaptation} study the set of log-concave densities that are exponentials of concave functions defined by at most $k$ linear hyperplanes. In this case, they show that the statistical rate becomes parametric. {\citet{xuHighdimensionalNonparametricDensity2021} consider log-concave densities whose super-level sets are scalar multiples of a fixed convex body in $\mathbb{R}^d$, and establish univariate rates for estimation.}
%-- Log-concave MTP2
\citet{robeva2021maximum} consider maximum likelihood estimation for totally positive and log-concave densities, which correspond to random vectors whose components are highly positively correlated. They prove existence and uniqueness of the MLE and propose an algorithm for computing it, although the statistical rates for this estimator are still unknown.
Next, \citet{logconcave} study the family of log-concave densities that factorize according to an undirected graphical model, also providing existence and uniqueness results for the MLE (when the graph is given and chordal) together with an algorithm. Sample complexity rates for this model have not been studied either.

Here, we study the subfamily of log-concave densities on $\bR^d$ obtained by transforming product densities through orthogonal maps. These correspond to random vectors whose coordinates become independent post some orthogonal transformation. The most similar family to ours was considered by \citet{samworth2012independent}, who study the class of log-concave densities of random vectors whose coordinates become independent after a linear transformation (rather than just an orthogonal transformation). They show that the projection (in a Kullback-Leibler (KL) sense) of a (not necessarily log-concave) distribution satisfying the above independent components property, to the set of log-concave distributions, still satisfies the independent components property with the same linear transformation. They also provide a method to perform density estimation in this setting, although their algorithm does not guarantee convergence to a global optimum. On the statistical front, they prove consistency; to the best of our knowledge however, finite sample rates have not yet been established in this setting. By simplifying the model here, we are able to provide a very fast algorithm together with sample complexity upper bounds {displaying univariate rates}.
% that do not grow exponentially with the dimension. 

Towards applicability to real data for tasks such as clustering, we show that our estimator can be used in conjunction with the Expectation-Maximization (EM) algorithm \cite{dempster1977maximum} to estimate mixtures of densities from our subfamily -- this is a {semiparametric} generalization of Gaussian mixtures. A similar application was proposed in \citet{cule2010maximum}, where they show that mixtures of general log-concave densities can work better than Gaussian mixtures. 
% A particular example is the Wisconsin breast cancer dataset~\cite{misc_breast_cancer_wisconsin_(diagnostic)_17}. 
However, estimating a mixture of general log-concave densities can be very costly, particularly in higher dimensions. Plugging our proposed estimator instead, into the EM algorithm, yields a much faster computation.

%#########################

\subsection{Problem}
\label{subsec: problem}

{
Suppose $X$ is a zero-mean random vector in $\bR^d$, with a (probability) density $p:\bR^d \to [0, \infty)$. We say that $X$ (or $p$) satisfies the \emph{orthogonal independent components} property if there exists an orthogonal matrix $\mathbf{W} \in \bR^{d \times d}$ such that the components of $Z = \mathbf{W}X$ are independent.
In this case, the joint density of $Z$ can be expressed as the product
\begin{align}
f(z)=\prod_{i=1}^{d} f_i(z_i), \label{eq: density of Z}
\end{align}
where each $f_i:\bR \to [0, \infty)$ is a univariate density. We focus mainly on the case where $p$ is log-concave, which is equivalent in this setting to requiring log-concavity of $f_1, \dots, f_d$. Our goal is to efficiently estimate $p$ from independent and identically distributed (iid) samples, given that it satisfies the orthogonal independent components property with some unknown $\eigmat$.

By writing $X = \eigmat^{-1} Z$, we can interpret this setting as an independent component analysis (ICA) model, where $Z$ is the vector of independent sources and $\eigmat$ is the so-called \emph{unmixing matrix}. 
Although mixing matrices in ICA do not need to be orthogonal to begin with, a pre-whitening procedure \cite{comon2010handbook} can transform the problem to such an orthogonal form.

% Our goal is to estimate the``$d$-dimensional" density $p$ of $X$ from independent and identically distributed (iid) samples.

Instead of estimating the ``$d$-dimensional" density $p$ directly, we would like to leverage the underlying structure provided by~(\ref{eq: density of Z}), which allows us to simplify $p$ (using a change-of-variables formula) as 
\begin{align}
    p(x) = \prod_{i=1}^{d} f_i\left( [\eigmat x]_i \right) = \prod_{i=1}^{d} f_i(w_i \cdot x). \label{eq: density of X} 
\end{align}
Here, $w_1, \dots, w_d \in \mathbb{S}^{d-1}$ are the rows of $\eigmat$, and we refer to them as \emph{independent directions} of $p$. Note that the model $Z=\eigmat X$ and the factorization in~(\ref{eq: density of X}) are invariant under permuting the rows of $\eigmat$ and the corresponding components of $Z$, as well as under flipping their respective signs.
}

If $\eigmat$ is known beforehand -- if it is provided by an \textit{oracle} -- the estimation problem is simplified, and one can use the following procedure:
\begin{enumerate}
    \item {For each $i=1, \dots, d$, estimate the univariate density $f_i$ (say via log-concave maximum likelihood estimation~\cite{samworth2018recent}) using the samples $\{Z_i^{(\ell)} = w_i \cdot X^{(\ell)}: \ell \in [N] \}$, where $X^{(1)}, \dots, X^{(N)} \overset{\mathrm{iid}}{\sim} p$. Denote the estimated density by $\hat{f_i}$.}
    \item Plug-in the estimated densities into~(\ref{eq: density of X}) to obtain the oracle-informed estimate of $p$:
    \begin{align}
        \hat{p}_{\text{oracle}}(x) = \prod_{i=1}^{d} \hat{f_i}(w_i \cdot x).
    \label{eq: oracle estimator def}
    \end{align}
\end{enumerate}

In reality, the unmixing matrix $\eigmat$ is unknown, and one needs to estimate it.
If the covariance matrix of $X$ has distinct, well-separated eigenvalues, then this can be done by Principal Component Analysis (PCA), which involves diagonalizing the sample covariance matrix of $X$. Otherwise, {certain ICA algorithms may be used to estimate $\eigmat$ provided some non-Gaussianity requirements are satisfied~\cite{auddyLargeDimensionalIndependent2023, fourier_pca}.} 
% This is important for pre-whitened ICA in particular, since the covariance matrix resulting from a pre-whitening procedure has all eigenvalues equal to 1.
We refrain from using the same samples for the unmixing matrix estimation and the density estimation steps. Instead, we split the samples into two subsets, $\{ X^{(1)}, \dots, X^{(N)} \}$ and $\{ Y^{(1)}, \dots, Y^{(M)} \}$, and use different subsets for the two purposes in order to prevent potentially problematic cross-interactions. 
% Further justification for this split and its role in our proof can be found in \Cref{sec: set-up and main result} and \ref{prf: estimating estimated marginals}.  
% {\color{red} LINK!}

Denote the (PCA or {ICA}) estimate of $\eigmat$ by $\hat{\eigmat}$. {We always require $\hat{\eigmat}$ to be an orthogonal matrix.} Given enough samples, the rows $\hat{w}_1, \dots, \hat{w}_d \in \mathbb{S}^{d-1}$ of $\hat{\eigmat}$ can be shown to be close to those of $\eigmat$ up to permutations and sign-flips. More precisely, there exists a permutation $\sigma:[d] \to [d]$ and signs $\theta_i \in \{-1,+1\}$ such that for some threshold $\epsilon>0$
\begin{align*}
    \norm{\hat{w_i} - \theta_i w_{\sigma(i)}}_2 \leq \epsilon, \quad \quad i \in [d].
\end{align*}
As the factorization in (\ref{eq: density of X}) is invariant under such permutations and sign-flips, one can simply relabel $w_i \leftarrow \theta_i w_{\sigma(i)}$ and $Z_i \leftarrow \theta_i Z_{\sigma(i)}$ to proceed with the analysis. 

From $\hat{\eigmat}$, we can follow a procedure analogous to that used to obtain the oracle-informed estimate of $p$. Namely, we use the estimated independent directions $\hat{w}_i \in \bR^d$ and the projected samples onto each such direction to produce {univariate} estimates $\hat p_{\hat w_i}$; then we take the product of these estimates over $i$ (see \Cref{def: proposed estimator}). 

\subsection{Summary of contributions}

The main contribution of this work is an algorithm for density estimation in the setting of \Cref{subsec: problem}, for which we prove finite sample guarantees and demonstrate numerical performance in terms of estimation errors and runtimes. The proposed algorithm is recorded as \Cref{alg: main algo}. 

The theoretical analysis in \Cref{sec: analysis} yields sample complexity bounds for our algorithm {under various assumptions, which are listed in \Cref{tab: table N} and \Cref{tab: table M}}. Briefly, if the true density $p$ {is log-concave} and satisfies some moment assumptions in addition to the {orthogonal} independent components property, then it can be estimated efficiently. In particular, {$\tilde{\mathcal{O}}_{d}(\epsilon^{-4})$ samples (up to constants and log factors)} are sufficient for our proposed estimator $\hat{p}$ to be within $\epsilon$ of $p$ in squared Hellinger distance with probability 0.99. {If the density $p$ has some additional smoothness, the rate can be improved up to $\epsilon^{-2}$}.
% The value of $\varrho$ depends on the particular smoothness assumption being evoked, but is nevertheless bounded above by 4 regardless of the dimension $d$. 

Since we referred to smoothness, a quick comparison with kernel density estimators (KDE) is in order. Recall that in dimension $d$, a KDE typically requires the true density $p$ to have ``smoothness of order $d$" in order to achieve a dimension-independent statistical rate. More precisely, for $p$ in the Sobolev space $\mathcal{W}^{s,1}(\bR^d)$ with $s\geq 1$, achieving an expected total variation error of $\epsilon$ requires $n \gtrsim_d \epsilon^{-\frac{d+2s}{s}}$ samples~\cite{holmstrom1992asymptotic} so that one needs $s \gtrsim d$ to counter the curse of dimensionality. On the contrary, our smoothness assumptions (see \nameref{para: S1} and \nameref{para: S2}) are ``order one" in that we only require control on the first-order (and if available, the second-order) derivatives.

Computational experiments are presented in \Cref{sec: comp exp}, where we demonstrate improved estimation errors (in squared Hellinger distance) and improved runtimes as compared to usual log-concave MLE. A key consequence of our fast algorithm is easy scalability to higher dimensions -- we are able to estimate densities in $d=30$ within no more than a few seconds.

\section{Set-up and algorithms}
\label{sec: set-up and main result}

First, we establish some notation. For a random vector (e.g. $X$), the subscript labels the components of the vector in $\bR^d$, whereas the superscript labels the iid samples (or copies). {Consider a (probability) density $p:\bR^d \to [0,\infty)$, and suppose $X \sim p$. Given a unit vector $s\in \mathbb{S}^{d-1}$, define the $s$-marginal $p_s:\bR \to [0,\infty)$ to be the density of $s \cdot X$}. 
{Additionally, for any (measurable) map $T: \bR^d \to \bR^d$, define the pushforward $T_{\sharp} p$ to be the law of $T(X)$. For an invertible linear map $\mathbf{A}: \bR^d \to \bR^d$, the pushforward $\mathbf{A}_{\sharp}p$ has a density $|\mathrm{det}(\mathbf{A}^{-1})|\, p ( \mathbf{A}^{-1}(x))$. As described in \Cref{sec: intro}, we have the relationship $Z=\eigmat X$ for which one can write $\mathrm{Law}(Z) = \eigmat_{\sharp} p$ by identifying $\eigmat$ with the linear map it induces}. 
% \deleted{similarly, define $\hat{Z} := \hat{\eigmat} X$}.

% As described in \Cref{sec: intro}, $Z=\eigmat X$; similarly, define $\hat{Z} := \hat{\eigmat} X$. \added{By identifying matrices with the linear maps they induce, one can write $\mathrm{Law}(Z) = \eigmat_{\sharp} p$.}
For any positive integer $n$, define $[n]:= \{1,2,\dots,n\}$. We denote by $\norm{\cdot}_2$ the Euclidean norm on $\bR^d$, and by $\op{\cdot}$ and $\norm{\cdot}_F$ respectively the operator norm and the Frobenius norm on $\bR^{d \times d}$. The squared Hellinger distance between densities $p, q: \bR^d \to [0,\infty)$ is defined as $h_d^2(p,q) = (1/2)\int_{\bR^d} (\sqrt{q(x)} - \sqrt{p(x)})^2 dx = 1 - \int_{\bR^d} \sqrt{q(x)p(x)} dx$. Note that $(p,q)\mapsto h_d(p,q) = (h_d^2(p,q))^{1/2}$ is a metric on the space of densities. {In various proofs, we also use the Wasserstein-1 distance, written $W_1(\cdot, \cdot)$.}
% We will also need the 1-Wasserstein distance between $p$ and $q$, defined as $W_1(p,q) = \inf\{\bE_{(X,Y)\sim \gamma}:  \gamma \textrm{ is a coupling between }p, q\}$. 

{Finally, the symbols $c, C, C'$ etc. stand for absolute (universal) constants that do not depend on any parameters of the problem. Also, their values may be readjusted as required in different statements. We also use constants $K_d, K_d'$ etc. that depend only on the dimension $d$.}
% for two quantities $a$ and $b$, $a \lesssim b$ means that $a \leq Cb$ for some absolute constant $C>0$ having no dependence on any parameters of the problem (including the dimension $d$). 

Recall the zero-mean density $p$, satisfying the orthogonal independent components property, and consider the split-up independent samples $\{ X^{(1)}, \dots, X^{(N)} \}$ and $\{ Y^{(1)}, \dots, Y^{(M)} \}$ as introduced in \Cref{sec: intro}. Our procedure for estimating $p$ consists of two stages: (1) use the samples $\{ Y^{(1)}, \dots, Y^{(M)} \}$ to estimate {an} unmixing matrix $\eigmat$, and (2) use the samples $\{ X^{(1)}, \dots, X^{(N)} \}$ to estimate the marginal densities of $p$ in the directions given by the (estimated) rows of $\eigmat$. We expand on these below, starting with the unmixing matrix estimation.

For $X \sim p$ ({with finite second moments}), consider the covariance matrix $\VarX = \bE XX^T$. Since the components of $Z = \eigmat X$ are independent, the covariance of $Z$
\begin{align*}
    \VarZ = \bE ZZ^T = \bE \eigmat XX^T \eigmat^T = \eigmat \VarX \eigmat^T
\end{align*}
is diagonal. Hence, the rows of $\eigmat$ are the (normalized) eigenvectors of $\VarX$, and PCA can be used to estimate these eigenvectors under the following non-degeneracy condition:
\vspace{0.1in}
\paragraph{Moment assumption M1}\label{para: M1}
The density $p$ has finite second moments, and the eigenvalues of the covariance matrix $\VarX = \bE_{X\sim p} XX^T$ are separated from each other by at least $\delta>0$.
\vspace{0.1in}
{Heteroscedastic multivariate Gaussians $\mathcal{N}(0, \mathbf{\Sigma})$, where $\mathbf{\Sigma}$ has well-separated eigenvalues, satisfy M1, and so do uniform distributions supported on rectangles with distinct side lengths.}

From samples $Y^{(1)}, \dots, Y^{(M)} \overset{\mathrm{iid}}{\sim} p$, one can compute the empirical covariance matrix 
\begin{align}
    \hat{\VarX} := \frac{1}{M} \sum_{j=1}^M Y^{(j)}{Y^{(j)}}^T,
    \label{eq: sample-VarX}
\end{align}
which is symmetric and positive semi-definite. Diagonalizing it yields the PCA-based estimate of $\eigmat$.

\begin{definition}[PCA-based estimate of $\eigmat$]
Let $Y^{(1)}, \dots, Y^{(M)} \overset{\mathrm{iid}}{\sim}p$, and $\hat{\VarX}$ defined as in (\ref{eq: sample-VarX}). Denote by $\hat{w}_1, \dots, \hat{w}_d$ a set of orthonormal eigenvectors of $\hat{\VarX}$. Then, the estimate $\hat{\eigmat}$ is defined to be the orthogonal matrix whose rows are $\hat{w}_1, \dots, \hat{w}_d$. 
% (\red{Note that it is defined only up to row permutations}).
\end{definition}
\begin{remark}
    Note that $\hat{\eigmat}$ above is defined only up to permutations and sign-flips of the rows. This is consistent, nevertheless, with our earlier discussion on the invariance of the model with respect to these transformations. 
\end{remark}
By the concentration of $\hat{\VarX}$ about $\VarX$, and the Davis-Kahan theorem~{\cite{yuUsefulVariantDavis2015}}, the $\hat{w}_i$'s can be shown to be good estimates of $w_i$, provided $\delta$ is not too small (see \Cref{res: estimating W by PCA}).

Now consider the situation where the eigenvalues of $\VarX$ are poorly separated, or possibly degenerate. PCA may fail to recover the independent directions in this case. Consider, for example, the extreme case of an isotropic uniform distribution on a cube in $\bR^d$ with $\VarX = \mathbf{I}$. Here, any $w \in \mathbb{S}^{d-1}$ is an eigenvector of $\VarX$, whereas the independent directions are only the $d$ axes of the cube.
When $\VarX$ has such degenerate eigenvalues, one needs to use higher moments to infer the unmixing matrix $\eigmat$, {which is the key principle behind ICA.} 
% {\color{blue} We consider the algorithm of \citet{auddyLargeDimensionalIndependent2023} in this case.} \deleted{An effective method with a polynomial sample complexity is Fourier-PCA \cite{fourier_pca}, which requires the following assumptions on $Z=\eigmat X$.}
% \paragraph{\deleted{Moment assumption M2}}
% \deleted{For each $i\in [d]$, assume that the fourth moment $\bE Z_i^4 \leq \mu_4$ for some $\mu_4 >0$ and that there exists a $k_i \in \mathbb{N}$ (chosen as small as possible) such that the cumulant $|\mathrm{cum}_{k_i}(Z_i)| \geq \Delta > 0$.  Defining $k = \max_i k_i$, further assume that the corresponding $k$-th moments $\bE |Z_i|^k \leq \mu_k$ for some $\mu_k>0$.}

% \deleted{Note that the assumption on the cumulants requires, in particular, that $Z$ be non-Gaussian as is standard in ICA settings. Since $\eigmat$ is a square matrix, the simpler version of Fourier PCA based on diagonalizing a re-weighted covariance matrix suffices -- see \cite{fourier_pca}. The algorithm takes as input the samples $Y^{(1)}, \dots, Y^{(M)} \overset{\mathrm{iid}}{\sim}p$ and the parameters $\Delta, k$ and $\mu_k$ from \\
% \nameref{para: M2}, and outputs an estimate $\hat{\eigmat}$.}

{
Recall our model $X = \mathbf{W}^{-1}Z$, where we think of $Z$ as the latent variables. As it is common in the ICA literature to have the latent variables `standardized' to be unit-variance, we define $S := \mathbf{\Sigma}_{Z}^{-1/2}Z$. This lets us rewrite $X = \mathbf{W}^{-1}Z = \mathbf{W}^{-1}\mathbf{\Sigma}_{Z}^{1/2} S = \mathbf{A}S$, where the new mixing matrix is $\mathbf{A}=\mathbf{W}^{-1}\mathbf{\Sigma}_{Z}^{1/2}$. By orthogonality of $\mathbf{W}$ and the diagonal nature of $\mathbf{\Sigma}_{Z}$, the columns of $\mathbf{A}$ are just scaled versions of the original independent directions $w_{1}, \dots, w_{d}$.

All ICA methods require some form of non-Gaussianity; following \citet{auddyLargeDimensionalIndependent2023}, we require the following assumptions on $S$ and $\mathbf{A}$.
\vspace{0.1in}
\paragraph{Moment assumption M2}\label{para: M2}

For each $i\in [d]$, assume that $\mu_4^{-1} \leq |\bE S_i^4 - 3| \leq \mu_4$ and $\bE |S_i|^9 \leq \mu_9$ for some $\mu_4, \mu_9 >0$. Additionally, assume that the condition number $\mathrm{cond}(\mathbf{A}) \leq \kappa$ for some $\kappa \geq 1$. 
% \vspace{0.1in}
\begin{remark}
    Note that since $S = \VarZ^{-1/2}Z$, one could alternatively frame \nameref{para: M2} in terms of $Z$. In particular, we have that $\mathrm{cond}(\mathbf{A}) = \mathrm{cond}(\VarZ)^{1/2}$.
\end{remark}
}

{
\citet{auddyLargeDimensionalIndependent2023} propose a variant of the FastICA algorithm~\cite{hyvarinen2000independent} that produces an estimate $\hat{\mathbf{A}}$ of $\mathbf{A}$ achieving the optimal sample complexity for computationally tractable procedures. To recover an estimate $\hat{\mathbf{W}}$ of $\mathbf{W}$, we first rescale the columns of $\hat{\mathbf{A}}$ to unit norm, and form a matrix $\tilde{\mathbf{A}}$ from these rescaled columns. To account for the possible non-orthogonality of the columns of $\tilde{\mathbf{A}}$, we compute the singular value decomposition $\tilde{\mathbf{A}}=\mathbf{U}\mathbf{D}\mathbf{V}^T$, and set $\hat{\mathbf{W}}:= \mathbf{V}\mathbf{U}^T$.

Isotropic uniform and exponential densities, for example, satisfy \nameref{para: M2}, but Gaussians do not due to the non-zero kurtosis requirement. But recall that heteroscedastic Gaussians do satisfy \nameref{para: M1}. If the moment assumptions discussed here are too restrictive, one could use other algorithms such as Fourier PCA~\cite{fourier_pca}, which require weaker notions of non-Gaussianity. Our theory in the following sections is agnostic of the specific procedure used to estimate $\eigmat$, and depends only on the estimation error.
}

Given $\hat{\eigmat}$, and hence the estimated independent directions $\hat{w}_1, \dots, \hat{w}_d \in \mathbb{S}^{d-1}$, consider the second stage of the estimation procedure now -- estimating the marginal densities of $p$. Take the samples $X^{(1)}, \dots, X^{(N)} \overset{\mathrm{iid}}{\sim}p$, which are independent {of} $Y^{(1)}, \dots, Y^{(M)}$, and project them as
\begin{align}
    \hat{Z}_{i}^{(j)} = \hat{w}_i \cdot X^{(j)}.
\end{align}
Conditional on $\hat{w}_1, \dots, \hat{w}_d$, it holds that $\hat{Z}_{i}^{(1)}, \dots, \hat{Z}_{i}^{(N)} \overset{\mathrm{iid}}{\sim}p_{\hat{w}_i}$ (almost surely in $\hat{w}_i$), which follows from the independence between $\hat{w}_i$ and $X^{(j)}$.
{Univariate density estimation} can then be used to estimate each marginal $p_{\hat{w}_i}$ using the samples $\hat{Z}_{i}^{(1)}, \dots, \hat{Z}_{i}^{(N)}$. {When $p$ is log-concave in particular (which is the focus of this work), we use univariate log-concave MLE~\cite{walther2002detecting}.}
This procedure finally yields our proposed ``plug-in" estimator.

\begin{definition}[Proposed estimator]
\label{def: proposed estimator}
% \deleted{Let $X^{(1)}, \dots, X^{(N)} \overset{\mathrm{iid}}{\sim}p$ be independent from \\$Y^{(1)}, \dots, Y^{(M)}$.} For \added{inferred independent directions} $\hat{w}_1, \dots, \hat{w}_d$ \deleted{estimated computed as above}, 
{Given estimated orthogonal independent directions \\$\hat{w}_1, \dots, \hat{w}_d \in \mathbb{S}^{d-1}$ computed from $Y^{(1)}, \dots, Y^{(M)} \overset{\mathrm{iid}}{\sim}p$, and univariate density estimates $\hat{p}_{\hat{w}_1}, \dots, \hat{p}_{\hat{w}_d}$ computed from $X^{(1)}, \dots, X^{(N)} \overset{\mathrm{iid}}{\sim}p$,} define the estimator
\begin{align}
    \hat{p}(x) = \prod_{i=1}^d \hat{p}_{\hat{w}_i}(\hat{w}_i \cdot x). \label{eq: proposed estimate}
\end{align}
% where each  $\hat{p}_{\hat{w}_i}$ for $i=1, \dots, d$ is \deleted{the log-concave MLE} \added{a univariate density estimate} of $p_{\hat{w}_i}$ (conditional on $\hat{w}_i$).
\end{definition}

Note that the proposed estimate is just the oracle-informed estimate with $w_i$ replaced by $\hat{w}_i$ (since $f_i = \text{Law}(Z_i) = p_{w_i}$). Further note that if we did not split the samples to conduct covariance estimation and density estimation separately, the law of $\hat{Z}_i$ conditional on $\hat{w}_i$ would no longer be $p_{\hat{w}_i}$.
% ; as a result, this proposed estimator would be somewhat ill-justified.

We summarize the above discussion as \Cref{alg: main algo}. Recall that we had assumed $p$ to be zero-mean in our set-up; in the general case, one requires an additional centering step which we include {in \Cref{alg: main algo}}.

\begin{algorithm}[!h]
    \caption{Proposed estimator}
    \begin{algorithmic}[1]
        \STATE Compute the empirical mean $\hat{\mu}$ and subtract it from the samples.
        \STATE Split the (centered) samples into two sets: $\{ X^{(1)}, \dots, X^{(N)}\}$ and $\{ Y^{(1)}, \dots, Y^{(M)}\}$.
        \STATE {$\hat{\eigmat} := \mathrm{EstimateUnmixingMatrix}(Y^{(1)}, \dots, Y^{(M)})$}
        \STATE Compute $\hat{Z}^{(j)} := \hat{\eigmat}X^{(j)}$ for $j = 1, \dots, N$.
        \FOR{$i \in \{1 , \dots, d\}$}
        \STATE {$\hat{p}_{\hat{w}_i} := \mathrm{EstimateUnivariateDensity}(\hat{Z}_{i}^{(1)}, \dots, \hat{Z}_{i}^{(N)})$ \label{alglin: logconcavemle}}
        \ENDFOR
        \STATE Compute $\hat{p}(x) := \prod_{i=1}^d \hat{p}_{\hat{w}_i}(\hat{w}_i \cdot x)$, where $\hat{w}_i$ is the $i$-th row of $\hat{\eigmat}$.
        \STATE Re-introduce the mean $\hat \mu$ to output $\hat p_{\hat \mu}(x) = \hat p(x - \hat \mu)$.
    \end{algorithmic}
    \label{alg: main algo}
\end{algorithm}

Notice that we have not specified the sample splitting proportion between $M$ and $N$. This is a free parameter in our method, and can be optimized based on the statistical rates obtained in \Cref{sec: analysis}, or based on numerical tests as explored in \Cref{subsec: split-r-results}. An equal split with $M = N$ seems to generally work well. {Interestingly, we observe empirically that \textit{not} splitting the samples, i.e. reusing the same samples for the two stages, can also work very well, and even beat the splitting version.}

{Finally, we have kept the estimation procedures for the unmixing matrix and the marginals unspecified in \Cref{def: proposed estimator} and \Cref{alg: main algo} for the sake of generality. The framework of our analysis allows for flexibility in the choice of these procedures. However, we do focus mainly on PCA, ICA, and log-concave MLE in the following sections, and refer to \Cref{alg: main algo} as the \emph{log-concave independent components} estimator (LC-IC) in that case.
}
In particular, we use the \texttt{R} package \texttt{logcondens} \cite{logcondensref} for univariate log-concave MLE in Line \ref{alglin: logconcavemle} of \Cref{alg: main algo}. This package plays an important role in the computational speed of our method.

\section{Analysis and sample complexities}
\label{sec: analysis}
This section is organized as follows. {\Cref{subsec: basic error bounds} presents basic error bounds and a generic sample complexity result, where the estimators for the unmixing matrix and the marginals are left unspecified. \Cref{subsec: univariate density estimation} and \Cref{sec: bounds for estimation of eigmat} then consider specific estimators to provide concrete sample complexities. \Cref{subsec: additional assumptions} derives stronger bounds under additional assumptions, and collects the various sample complexities. Finally, \Cref{subsec: misspecification} considers the case of misspecified log-concavity.}
% presents the sample complexity of estimating the unmixing matrix, and \Cref{subsec: stability} provides a stability analysis under two different smoothness assumptions, aiming to quantify the effect of unmixing matrix estimation error on the proposed estimator. Finally,  \Cref{subsec: corr sample complexities} gives the sample complexities of the proposed estimators as corollaries of the preceding analysis.

\subsection{Error bounds and generic sample complexities}
\label{subsec: basic error bounds}
% We would like to assess the sample complexity of the proposed estimator, as well as its performance relative to the oracle-informed estimator. 
% Throughout \Cref{subsec: basic error bounds}, let $p$ be a probability density with mean zero, satisfying the \added{orthogonal} independent components property (\ref{eq: density of X}) with independent directions $w_1, \dots, w_d$.
We consider the oracle-informed estimator first. {The analysis involved is also applicable to the proposed estimator.}
\begin{lemma}[Error bound on the oracle estimator]
\label{lemma: h oracle}
{Suppose $p$ is a zero-mean probability density on $\mathbb{R}^d$, satisfying the orthogonal independent components property with independent directions $w_{1}, \dots, w_{d}$.} Let $\hat{p}_{\mathrm{oracle}}$ be the oracle-informed estimator as defined in (\ref{eq: oracle estimator def}). We have the bound
    $$h_d^2(\hat{p}_{\mathrm{oracle}}, p) = 1 - \prod_{i=1}^d (1-h_1^2(\hat{p}_{w_i}, p_{w_i})) \leq d \, \max_{i\in [d]} h_1^2(\hat{p}_{w_i}, p_{w_i}).$$
\end{lemma}

\begin{proof}
    A direct calculation yields
    \begin{align*}
        1 - h_d^2(\hat{p}_{\text{oracle}}, p) &= \int_{\bR^d} \prod_{i=1}^d \sqrt{\hat{p}_{w_i}(w_i \cdot x)\, p_{w_i}(w_i \cdot x)} dx \\
        &= \prod_{i=1}^d \int_{\bR} \sqrt{\hat{p}_{w_i}(z_i)\, p_{w_i}(z_i)} dz_i = \prod_{i=1}^d \left(1 - h_1^2(\hat{p}_{w_i}, p_{w_i})\right).
    \end{align*}
    Now define $b = \max_{i\in [d]} h_1^2(\hat{p}_{w_i}, p_{w_i})$ and note that $b\in [0,1]$. Bernoulli's inequality gives that $\prod_{i=1}^d (1 - h_1^2(\hat{p}_{w_i}, p_{w_i})) \geq (1-b)^d \geq 1-d\,b$.
\end{proof}

\Cref{lemma: h oracle} says that the problem of bounding $h_d^2(\hat{p}_{\text{oracle}}, p)$ boils down to controlling the {univariate} estimation errors $h_1^2(\hat{p}_{w_i}, p_{w_i})$. For log-concave densities, \citet{kim2016global} and \citet{ carpenter2018near} have already established results of this kind. 
{
We stay in a general setting for now, by deriving the sample complexity of the oracle estimator from the sample complexity of any univariate density estimation procedure that it invokes.

\begin{definition}[Sample complexity of a univariate density estimator]
\label{def: sample complexity density estimation}
% Let $\mathcal{F}_1$ be a subclass of univariate densities $f:\mathbb{R}\to [0, \infty)$, and let $\hat{f}$ be an estimator of $f \in \mathcal{F}_{1}$ computed from iid samples. 
% Given $\epsilon> 0$ and $0 < \gamma < 1$, define the sample complexity $N^* = N^*(\epsilon, \gamma; \mathcal{F}_{1})$ of $\hat{f}$ to be the minimal positive integer such that for any $f \in \mathcal{F}_{1}$, and $N \geq N^*$ iid samples from $f$, we get 
Let $f:\mathbb{R} \to [0,\infty)$ be a univariate probability density, and let $\hat{f}$ be an estimator of $f$ computed from iid samples. For $\epsilon> 0$ and $0 < \gamma < 1$, define the sample complexity $N^* = N^*(\epsilon, \gamma; f)$ of $\hat{f}$ be be the minimal positive integer such that $N \geq N^*$ iid samples from $f$ are sufficient to give
\begin{equation*}
\mathbb{P} \left(  h_{1}^2(\hat{f}, f) \leq \epsilon \right) \geq 1-\gamma.
\end{equation*}
\end{definition}

% Using this definition, it is straightforward to obtain the sample complexity of the oracle estimator.

% Now suppose that the marginals $p_{w_{1}}, \dots, p_{w_d}$ of $p$ belong to a subclass $\mathcal{F}_{1}$, and the univariate density estimation step of the oracle estimator has sample complexity $N^*(\epsilon, \gamma; \mathcal{F}_{1})$. 
Now, we can immediately express the sample complexity of the oracle estimator in terms of $N^*$ as a consequence of \Cref{lemma: h oracle} and a union bound.

\begin{corollary}[Sample complexity of the oracle-informed estimator] \label{thm: oracle risk}
% Suppose the marginals $p_{w_{1}}, \dots, p_{w_d}$ of $p$ belong to some subclass $\mathcal{F}_{1}$ of univariate densities.
Let $p$ be a zero-mean probability density on $\mathbb{R}^d$, satisfying the orthogonal independent components property. Let $\hat{p}_{\mathrm{oracle}}$ be the oracle-informed estimator computed from samples $X^{(1)}, \dots, X^{(N)} \overset{\mathrm{iid}}{\sim} p$.
Then, for any $\epsilon > 0$ and $0 < \gamma < 1$, we have  
\begin{align*}
    h_d^2(\hat{p}_{\mathrm{oracle}}, p) \leq \epsilon
\end{align*}
with probability at least $1-\gamma$, whenever
\begin{align*}
    N \geq  \sup_{i \in [d]} N^*(\epsilon/d, \gamma/d; p_{w_i}).
\end{align*}
Here, $N^*$ is the sample complexity of the univariate density estimator invoked by $\hat{p}_{\mathrm{oracle}}$.
\end{corollary}
}

% \begin{proof}
% Define $\tilde{\epsilon} = \epsilon/d$ and $\tilde{\gamma} = \gamma/d$, so that 
% $$N \gtrsim \frac{1}{\tilde{\epsilon}^2} \log^6 \frac{1}{\tilde{\epsilon}\tilde{\gamma}}$$
% Then, using Theorem 7 of \citet{carpenter2018near} as applicable to the 1-dimensional log-concave MLE $\hat{p}_{w_i}$ of $p_{w_i}$, computed from samples $Z_i^{(j)} = w_i \cdot X^{(j)}$ for $j = 1, \dots, N$, we get that $h_1^2(\hat{p}_{w_i}, p_{w_i}) \leq \tilde{\epsilon}$ holds with probability at least $1-\tilde{\gamma}$ for any given $i\in [d]$. By a union bound, $\max_{i\in [d]} h_1^2(\hat{p}_{w_i}, p_{w_i}) \leq \tilde{\epsilon}$ with probability at least $1-\tilde{\gamma}d = 1-\gamma$. Finally, by \Cref{lemma: h oracle},
% $$h_d^2(\hat{p}_{\mathrm{oracle}}, p) \leq \tilde{\epsilon}d = \epsilon.$$
% as desired.
% \end{proof}

Having concluded the analysis of the oracle-informed estimator, we can now turn to the general setting. Here, we only have an estimate of $\eigmat$, and hence, the estimation error needs to be taken into account. This contributes additional terms to the risk of the proposed estimator, as the theorem below shows. The proof can be found in \Cref{prf: basic error bound}.

\begin{theorem}[Error bound on the proposed estimator]
{Suppose $p$ is a zero-mean probability density on $\mathbb{R}^d$, satisfying the orthogonal independent components property with independent directions $w_{1}, \dots, w_{d}$.} Let $\hat{p}$ be the proposed estimator (\Cref{def: proposed estimator}). We have the error bound
\label{res: basic error bound}

\begin{align*}
    h_d(\hat{p}, p) &\leq \sqrt{d} \, \max_{i\in [d]} h_1(\hat{p}_{\hat{w}_i}, p_{\hat{w}_i}) 
    + {\sqrt{d}\,h_d\left(\hat{\mathbf{R}}_{\sharp}p \, , \, p \right) + h_d\left(\hat{\mathbf{R}}^T_{\sharp}p \, , \, p \right)},
\end{align*}
where $\hat{\mathbf{R}} = \eigmat^T \hat{\eigmat}$.
    
\end{theorem}

{Here, $\hat{\Rpl}$ is an orthogonal matrix, because $\eigmat$ and $\hat{\eigmat}$ are.}
The term $\max_{i\in [d]} h_1(\hat{p}_{\hat{w}_i}, p_{\hat{w}_i})$ in the bound of \Cref{res: basic error bound} warrants an analysis similar to the oracle-informed estimator. 
Since $\hat{p}_{\hat{w}_i}$ is {a univariate density estimate} of $p_{\hat{w}_i}$ conditional on $\hat{w}_i$, 
bounds on {univariate density estimation} control this term, when applied in a conditional sense.

{
The second and third terms in the bound of \Cref{res: basic error bound} call for a notion of \emph{stability}; $h_d(\hat{\mathbf{R}}_{\sharp}p, p)$ measures how far the `rotated' density $\hat{\mathbf{R}}_{\sharp}p$ gets from the original density $p$. If Hellinger distance is stable with respect to small rotations of the density, then $h_d(\hat{\mathbf{R}}_{\sharp}p, p)$ can be bounded effectively in terms of the operator norm $\norm{\mathbf{I} - \hat{\mathbf{R}}}_{\mathrm{op}}$, where $\mathbf{I}$ is the $d\times d$ identity matrix. This operator norm is then well-controlled by the unmixing matrix estimation error as
\begin{align}
\label{eq: op norm w closeness}
    \op{\mathbf{I} - \hat{\Rpl}} &= \op{\eigmat - \hat{\eigmat}} \leq \norm{\eigmat - \hat{\eigmat}}_{F} = \sqrt{\textstyle{\sum_{i=1}^d \norm{\hat{w}_i - w_i}_2^2}} \leq \sqrt{d}\,\max_{i \in [d]} \norm{\hat{w}_i - w_i}_2.
\end{align}
The next lemma demonstrates such stability for zero-mean log-concave densities. 

\begin{lemma}[Hellinger stability]
\label{res: hellinger stability general log-concave}
Let $p$ be a zero-mean log-concave density on $\mathbb{R}^d$, and let $\mathbf{A} \in\mathbb{R}^{d \times d}$ be any invertible matrix. Then,
\begin{equation*}
h_{d}^2(p, \mathbf{A}_{\sharp}p) \leq K_{d} \,\mathrm{cond}(\mathbf{\Sigma})^{1/4}\,\op{\mathbf{I}-\mathbf{A}}^{1/2},
\end{equation*}
where $\mathbf{\Sigma}$ is the covariance matrix of $p$, and the constant $K_{d}>0$ depends only on the dimension $d$.

\end{lemma}
The proof involves the log-concave projection of \citet{dumbgen_ApproximationLogConcaveDistributions_2011}, particularly making use of its local H{\"o}lder continuity~\cite{barberLocalContinuityLogconcave2021}. It is delayed to \Cref{prf: hellinger stability general log-concave}.
Note that \Cref{res: hellinger stability general log-concave} is quite general, as it only uses log-concavity of $p$ (with no additional smoothness requirements).
However, \Cref{res: hellinger stability general log-concave} inherits a constant $K_d$ that typically depends at least exponentially in $d$. Also, the exponent $1/2$ of $\op{\mathbf{I}-\mathbf{A}}$ contributes a $1/\epsilon^4$ dependence to our sample complexity bounds (see the first two rows of \Cref{tab: table M}). 
{Later in \Cref{subsec: additional assumptions}}, we show that smoothness assumptions on $p$ can give stronger bounds.
}

{
The last bit to consider is the unmixing matrix estimation error.
In the same spirit as the discussion on univariate density estimation above, we consider a generic estimator for now, with sample complexity defined as below.

\begin{definition}[Sample complexity of an unmixing matrix estimator]
\label{def: sample complexity unmixing estimation}
Let $p$ be a zero-mean probability density on $\mathbb{R}^d$ satisfying the orthogonal independent components property. Let $\hat{\mathbf{W}}$ be an estimator of an unmixing matrix $\mathbf{W}$ of $p$, computed from iid samples. For $\epsilon> 0$ and $0 < \gamma < 1$, define the sample complexity $M^* = M^*(\epsilon, \gamma; p)$ of $\hat{\mathbf{W}}$ to be the minimal positive integer such that $M \geq M^*$ iid samples from $p$ are sufficient to give
$$
\mathbb{P}\left\{ \textstyle{\left( d^{-1} \sum_{i=1}^d \|\hat{w}_{i}-w_{i}\|_{2}^{2}\right)^{1/2}} \leq \epsilon \right\}  \geq 1-\gamma,
$$
for some set $\{w_{1},\dots,w_{d}\}$ of independent directions of $p$.
% Let $\mathcal{F}_{d}$ be a subclass of zero-mean densities $p$ on $\mathbb{R}^{d}$ satisfying the orthogonal independent components property, and let $\hat{\mathbf{W}}$ be an estimator of the unmixing matrix of $p \in \mathcal{F}_{d}$, computed from iid samples. Given $\epsilon >0$ and $0 < \gamma<1$, define the sample complexity $M^*=M^*(\epsilon, \gamma; \mathcal{F}_{d})$ of $\hat{\mathbf{W}}$ to be the minimal positive integer such that for any $p \in \mathcal{F}_{d}$ with unmixing matrix $\mathbf{W}$, and $M \geq M^*$ iid samples from $p$, we get
% \begin{equation*}
% \mathbb{P}\left( \max_{i \in [d]} \|\hat{w}_{i}-w_{i}\|_{2} \leq \epsilon \right) \geq 1-\gamma.
% \end{equation*}
\end{definition}

Now, we state the main result of this section, which expresses the sample complexity of the proposed estimator (\Cref{def: proposed estimator}) in terms of the sample complexities in \Cref{def: sample complexity density estimation} and \Cref{def: sample complexity unmixing estimation}. The proof is presented in \Cref{prf: generic sample complexity proposed}.
% building upon \Cref{def: sample complexity density estimation}, \Cref{def: sample complexity unmixing estimation}, \Cref{res: basic error bound} and \Cref{res: hellinger stability general log-concave}.

\begin{theorem}[Generic sample complexity of the proposed estimator]
\label{res: generic sample complexity proposed}
Let $p$ be a zero-mean log-concave density on $\mathbb{R}^d$, satisfying the orthogonal independent components property. Let $\hat{p}$ be the proposed estimator computed from samples $X^{(1)}, \dots, X^{(N)}, Y^{(1)}, \dots, Y^{(M)} \overset{\mathrm{iid}}{\sim}p$. Then, for any $\epsilon>0$ and $0 < \gamma<1$, we have 
\begin{equation*}
h_{d}^2(\hat{p},p) \leq \epsilon
\end{equation*}
with probability at least $1-\gamma$ whenever
\begin{subequations}
\begin{equation}
N \geq \sup_{s \in \mathbb{S}^{d-1}}N^*\left( \epsilon /4d,\, \gamma /2d;\, p_{s} \right),
\label{eq: N generic}
\end{equation}
and
\begin{equation}
M \geq M^*\left( (K_{d}'\,\mathrm{cond}(\mathbf{\Sigma}))^{-1/2}\epsilon^2,\, \gamma /2;\, p \right).
\label{eq: M generic}
\end{equation}
\end{subequations}
Here, $N^*$ and $M^*$ are respectively the sample complexities of the univariate density estimator and the unmixing matrix estimator invoked by $\hat{p}$. As before, $\mathbf{\Sigma}$ is the covariance matrix of $p$ and $K_{d}' >0$ is a constant that depends only on $d$.

\end{theorem}

To obtain concrete sample complexities, one needs to choose specific estimation procedures. In \Cref{subsec: univariate density estimation}, we consider univariate log-concave MLE and its boosted variant, and provide corresponding bounds on $N^*$. In \Cref{sec: bounds for estimation of eigmat}, we present standard bounds on $M^*$ for PCA- and ICA-based estimators of $\eigmat$.
}

% \begin{lemma}[Marginals of the proposed estimator]
% Let $p$, $\hat{p}$, and $\hat{w}_i$ as before.
% For any $0 < \epsilon < 1$ and $0 < \gamma < 1$, we have
% \begin{align*}
%     \max_{i \in [d]} h_1^2(\hat{p}_{\hat{w}_i}, p_{\hat{w}_i}) \leq \epsilon,
% \end{align*}
% with probability at least $1-\gamma$, whenever 
% \begin{align*}
%     N \gtrsim \frac{1}{\epsilon^2} \log^6 \frac{d}{\epsilon \gamma}.
% \end{align*}
% \label{res: estimating estimated marginals}
% \end{lemma}
% The proof is deferred to \ref{prf: estimating estimated marginals}. To control the remaining terms from the bound of \Cref{res: basic error bound}, we need the results from \Cref{sec: bounds for estimation of eigmat} and \Cref{subsec: stability}.

\subsection{Univariate density estimation for the marginals}
\label{subsec: univariate density estimation}
{
Recall that for each $i = 1, \dots, d$, we need to estimate the densities $p_{\hat{w}_i}$ from samples $\{\hat{Z}_i^{(j)} = \hat{w}_{i} \cdot X^{(j)}: j \in [M]\}$. When $p$ is log-concave, an efficient as well as convenient option is log-concave MLE~\cite{samworth2018recent}, which is in fact minimax optimal.

\begin{proposition}[Sample complexity of univariate log-concave MLE~\cite{kim2016global}]
\label{res: sample complexity 1d log-concave MLE}
Let $f: \mathbb{R} \to [0, \infty)$ be a univariate log-concave density, and denote by $\hat{f}$ the log-concave MLE computed from $N$ iid samples from $f$. Given $\epsilon >0$ and $0 < \gamma < 1$, we have 
$$
h_{1}^{2}(\hat{f}, f) \leq \epsilon
$$
with probability at least $1-\gamma$ whenever
$$
N \geq \frac{C_{1}}{\epsilon^{5/4}} \log^{5/4}\left( \frac{C_{2}}{\gamma} \right) \vee \frac{C_{3}}{\gamma},
$$
for absolute constants $C_1, C_2, C_3 >0$. 
\end{proposition}

In terms of \Cref{def: sample complexity density estimation}, we can restate this
as $N^*(\epsilon, \gamma; f) \leq C_{1}\epsilon^{-5/4} \log^{5/4}\left(C_{2} /\gamma \right) \vee C_{3} /\gamma$. Notice that the right-hand side does not depend on $f$, and as a result, the same upper bound also holds for the term $\sup_{s \in \mathbb{S}^{d-1}} N^*(\epsilon, \gamma; p_s)$ from \Cref{res: generic sample complexity proposed}

The statement above is a ``high-probability" variant of Theorem 5 of \citet{kim2016global}, and we provide a proof in \Cref{prf: sample complexity 1d log-concave MLE}. Note that while the $\epsilon^{-5/4}$ dependence is optimal, the $1/\gamma$ dependence in the sample complexity could be improved. A standard way to do this is via \emph{boosting}. The following definition uses the boosting technique from Algorithm A of \citet{geom_logconcave_boosting}.

\begin{definition}[Boosted univariate log-concave MLE]
\label{def: boosted}
    Let $n, q \in \mathbb{N}$, and suppose we have access to $nq$ iid samples from a univariate log-concave density $f: \mathbb{R} \to [0, \infty)$. Additionally, fix $\epsilon >0$. The boosted estimator $\hat{\hat{f}}_{\epsilon}$ at error threshold $\epsilon$ is defined via the following procedure:
\vspace{0.1in}
\begin{enumerate}
\item Split the samples into $q$ batches of size $n$, and denote by $\hat{f}^{[b]}$ ($b=1,2, \dots, q$) the log-concave MLE computed from the samples in batch $b$. 
\item For each batch $b$, count how many other batches $b'$ satisfy $h_{1}(\hat{f}^{[b]}, \hat{f}^{[b']}) \leq 2\sqrt{ \epsilon } /3$.
\item  If there exists a batch $b$ for which this number is larger than $q /2$, return $\hat{\hat{f}}_{\epsilon}:=\hat{f}^{[b]}$. Otherwise, return $\hat{\hat{f}}_{\epsilon}:= \hat{f}^{[1]}$ .
\end{enumerate}
\end{definition}

As the next proposition show, the sample complexity of the boosted log-concave MLE depends only logarithmically on $1/\gamma$, which is much milder. See \Cref{prf: boosted sample complexity} for the proof.
\begin{proposition}[Sample complexity boosted univariate log-concave MLE]
\label{res: boosted sample complexity}
Let $f: \mathbb{R} \to [0, \infty)$ be a univariate log-concave density, and fix $\epsilon>0$. Denote by $\hat{\hat{f}}_\epsilon$ the boosted log-concave MLE at error threshold $\epsilon$, computed from $q$  independent batches of $n$ iid samples from $f$ each. 
Given $0 < \gamma < 1$, we have 
\begin{equation*}
h_{1}^{2}(\hat{\hat{f}}_\epsilon,f) \leq \epsilon
\end{equation*}
with probability at $1-\gamma$ provided 
\begin{equation*}
n \geq \frac{C_{1}}{\epsilon^{5/4}} \quad \text{and} \quad q \geq C_{2} \log \left( \frac{1}{\gamma} \right),
\end{equation*}
for absolute constants $C_1, C_2 \geq 0$.

\end{proposition}

In particular, by multiplying the bounds for $n$ and $q$, we get that the sample complexity \\
$N^*(\epsilon, \gamma; f) \leq C_{1}C_{2}\epsilon^{-5/4}\log(1 /\gamma)$. As before, the right-hand side does not depend on $f$.

}

\subsection{Unmixing matrix estimation}
\label{sec: bounds for estimation of eigmat}
Here, we discuss standard techniques for estimating $\eigmat$, and
show closeness of $\hat{w}_i$ and $w_i$. Recall that the methods used here for estimating $\eigmat$ -- PCA or {ICA} -- require different assumptions on the moments of $X$ (or $Z = \eigmat X$) in order to yield bounds on the estimation errors.

Consider regular PCA first. The PCA estimate of $\eigmat$ was obtained by computing the eigenvectors of the sample covariance matrix $\hat{\VarX}$. Bounding the sample complexity of this estimation procedure is standard, and can be broken down into two simple steps: (i) bounding the distance between the empirical covariance matrix $\hat{\VarX}$ and the true covariance matrix $\VarX$, and (ii) bounding the distances between their corresponding eigenvectors.  
The result is summarized in the following proposition, the proof of which is delayed to \ref{prf: estimating W by PCA}.

\begin{proposition}[Sample complexity of estimating $\eigmat$ by PCA]
\label{res: estimating W by PCA}
Let $p$ be a zero-mean, log-concave probability density on $\bR^d$. Let $Y^{(1)}, Y^{(2)}, \dots, Y^{(M)} \overset{\mathrm{iid}}{\sim} p$ be samples, and suppose \nameref{para: M1} is satisfied with eigenvalue separation $\delta>0$. Let $\epsilon >0$  and $0<\gamma \leq 1/e$, and define $\tilde M(d, \gamma):= d \log^4(1/\gamma) \log^2\left({C_0} \log^2(1/\gamma) \right)$. If
\begin{align}
    M \geq {C _1}\tilde{M}(d, \gamma) \vee \left\{\frac{{C_2} \op{\VarX}^2 d}{\delta^2 \epsilon^2}\log^4(1/\gamma)\log^2 \left(\frac{{C_3}\,\op{\VarX}^2}{\delta^2\epsilon^2}\log^2(1/\gamma)\right) \right\}
  \label{eq: PCA sample complexity}  
\end{align}
{where $C_0, C_1, C_2, C_3 > 0$ are absolute constants}, then with probability at least $1-\gamma$, PCA recovers vectors $\{\hat{w}_1, \dots, \hat{w}_d\}$ such that
$$ 
\norm{\hat{w}_i - w_i}_2 \leq \epsilon,
$$
up to permutations and sign-flips.
\end{proposition}

{The right-hand side of \eqref{eq: PCA sample complexity} is an upper bound on sample complexity $M^*(\epsilon, \gamma; p)$ of the PCA-based estimator, for $0 < \gamma \leq 1/e$.} The term $\tilde M(d, \gamma)$ is needed in order to allow the result above to hold for all positive $\epsilon$, instead of just small values of $\epsilon$.

{
As discussed earlier, assumptions on strict eigenvalue separation can be stringent, and possibly unrealistic in practical settings. In such cases, $\eigmat$ should be estimated by ICA instead, for which we have a result similar to \Cref{res: estimating W by PCA}.
}

{
Recall the definitions of $S$ and $\mathbf{A}$ from \Cref{sec: set-up and main result}, and express $X= \mathbf{A}S$. The algorithm of \citet{auddyLargeDimensionalIndependent2023} then estimates the columns of $\mathbf{A}$, which coincide with $w_1, \dots, w_d \in \mathbb{S}^{d-1}$ up to scalings and signs. 

The proposition below is a direct consequence of Theorem 3.4 of \citet{auddyLargeDimensionalIndependent2023}, and provides an upper bound on $M^*(\epsilon, \gamma; p)$ for $3d^{-3} \leq \gamma < 1$. The proof is delayed to \Cref{prf: W estimation ICA}.

\begin{proposition}[Sample complexity of estimating $\eigmat$ by ICA~\cite{auddyLargeDimensionalIndependent2023}]
\label{res: W estimation ICA}
Let $p$ be a zero-mean probability density on $\mathbb{R}^{d}$ satisfying the orthogonal independent components property. Suppose, additionally, that \nameref{para: M2} is satisfied with parameters $\mu_{4}, \mu_{9} >0$ and $\kappa \geq 1$. Given $\epsilon > 0$ and $3d^{-3} \leq \gamma < 1$, and samples $Y^{(1)}, \dots, Y^{(M)} \overset{\mathrm{iid}}{\sim}p$, if
\begin{equation}
\label{eq: M ICA}
M \geq \frac{K' d}{\epsilon^2} \log (3 /\gamma) \,\vee\, K'' d^2 \log^2(3 /\gamma) \,\vee\, (3 /\gamma)^{8/9}
\end{equation}
for constants $K', K'' >0$ depending on $\mu_{4}, \mu_{9}, \kappa$, then with probability at least $1-\gamma$, ICA recovers vectors $\{ \hat{w}_{1}, \dots, \hat{w}_{d} \}$ such that
\begin{equation*}
\left( \frac{1}{d} \sum_{i=1}^d\|\hat{w}_{i}-w_{i}\|_{2}^2 \right)^{1/2}  \leq \epsilon
\end{equation*}
for independent directions $w_{1}, \dots, w_{d}$ of $p$.

\end{proposition}
% The right-hand side of \eqref{eq: M ICA} upper bounds $M^*(\epsilon, \gamma; p)$ over the range $3d^{-3} \leq \gamma < 1$.

% \begin{remark}
% The bound in Theorem 4.2 of \cite{fourier_pca} holds with high probability. An {arbitrarily high success probability} of $1-\gamma$ is achieved by \emph{boosting} (see for example \citet[Theorem 2.8]{geom_logconcave_boosting}), which is responsible for the additional factor of $\log 1/\gamma$ in the sample complexity above.
% \end{remark}

{
The result above relies on non-zero excess kurtosis of the components of $Z$, which may be interpreted as a measure of non-Gaussianity. However, there exist distributions which differ from Gaussians only in higher cumulants. Techniques such as Fourier PCA~\cite{fourier_pca} allow for estimating $\eigmat$ in such cases. As discussed earlier, our theory is agnostic to specific choices of ICA algorithms.
}

}

\subsection{Stronger bounds under additional assumptions}
\label{subsec: additional assumptions}
{
While the statement of \Cref{res: generic sample complexity proposed} is quite general, the inequality in \eqref{eq: M generic} contains a constant $K'_d>0$ that typically depends at least exponentially in $d$. This is owed to the stability result \Cref{res: hellinger stability general log-concave}, which uses only the log-concavity of $p$. Stronger stability inequalities are possible when $p$ is smooth.
}

In what follows, we consider two possible assumptions on the smoothness of $p$. The choice of assumptions then allows us to quantify stability in either Kullback-Leibler (KL) divergence or total variation distance first, before translating this to stability in Hellinger distance.
Let $p$ be a probability density on $\bR^d$ with mean zero.

\paragraph{Smoothness assumption S1}\label{para: S1}
{The density $p$ has finite second moments, is continuously differentiable, and the integral $\int_{\mathbb{R}^d} \|x\|_{2} \|\nabla p(x)\|_{2}\,dx < \infty$.}
The function $\varphi = -\log p$ is finite-valued everywhere, and belongs to the H{\"o}lder space $\mathcal{C}^{1, \alpha}(\bR^d)$ for some $0< \alpha \leq 1$, i.e.
\begin{align}
    [\nabla \varphi]_\alpha := \sup_{x \neq y} \frac{\norm{\nabla \varphi(y) - \nabla \varphi(x)}_2}{\norm{y-x}_2^\alpha} < \infty.
    \label{eq: S1}
\end{align}

{Note that $-\nabla\varphi = \nabla p /p$ is commonly known as the \emph{score function}, whose regularity also plays a role in sampling theory. We may also interpret \eqref{eq: S1} as a modified H{\"o}lder condition on $p$.}
This assumption yields good stability in KL divergence, but it can be stringent since it requires, in particular, that $p$ is strictly positive everywhere. 
The following assumption relaxes this.

\paragraph{Smoothness assumption S2}\label{para: S2}
The density $p$ {has finite second moments}, and has a (weak) gradient $\nabla p$, with
\begin{align}
    \norm{\nabla p}_{L^1} = \int_{\bR^d}\norm{\nabla p(x)}_2\,dx < \infty.
    \label{eq: S2}
\end{align}
This assumption yields a stability estimate in total variation distance (or equivalently, $L^1$ distance). Note that $p$ is not required to be differentiable in the strong sense, and hence, can have ``corners" or ``kinks". But, jumps are still not allowed. {An interesting example is the exponential-type density $p(x)= a e^{-\norm{x}_1}$.} 
Also, to compare with assumption S1, note that if $p(x)= e^{-\varphi(x)}$ with $[\nabla \varphi]_\alpha < \infty$, then we write $\nabla p(x) = -\nabla \varphi(x) p(x)$ which can translate to a bound on $\norm{\nabla p}_{L^1}$.

Assumptions S1 and S2 can be interpreted better by considering a multivariate Gaussian distribution with density
\begin{align*}
    p(x) = a \exp\left( - \frac{1}{2} x \cdot \VarX^{-1} x \right)
\end{align*}
where $a = (2\pi)^{-d/2} (\det \VarX)^{-1/2}$ is the normalizing constant. Here, $\varphi(x) = \frac{1}{2} x \cdot \VarX^{-1} x - \log a$ is clearly $\mathcal{C}^{1, 1}(\bR^d)$, with $[\nabla \varphi]_{1} = 1/\lambda_{\min}(\VarX)$. On the other hand, $\norm{\nabla p}_{L^1} \leq \sqrt{d/\lambda_{\min}(\VarX)} < \infty$, provided $\lambda_{\min}(\VarX)>0$. We start with stability under \nameref{para: S1}.

\begin{lemma}[Stability in KL]
    \label{res: rotational stability of KL}
Let $p$ be a zero-mean probability density on $\bR^d$, satisfying \nameref{para: S1}. Let $\Rpl \in \bR^{d\times d}$ be any orthogonal matrix. Then,
    \begin{align*}
        \mathrm{KL}(p \,||\, {\Rpl_{\sharp}^Tp}) \leq [\nabla \varphi]_\alpha \, d^{(1+\alpha)/2} \, \op{\VarX}^{(1+\alpha)/2}\, \norm{\mathbf{I} - \Rpl}^{1+\alpha}_{\mathrm{op}}.
    \end{align*}
\end{lemma}

The proof involves a direct analysis of the KL integral, using the bound $\norm{\nabla \varphi(y) - \nabla\varphi(x)}_2 \leq [\nabla \varphi]_{\alpha} \norm{y-x}_2^\alpha$. It is delayed to \Cref{prf: rotational stability of KL}. This result translates to Hellinger distance via the upper bound~\cite{gibbs2002choosing}
\begin{align*}
    h_d^2(p, {\Rpl_{\sharp}^Tp}) \leq \frac{1}{2}\mathrm{KL}(p \,||\, {\Rpl_{\sharp}^Tp}).
\end{align*}

{We remark that the exponent $1+\alpha$ of $\op{\mathbf{I} - \Rpl}$ in \Cref{res: rotational stability of KL} is significantly higher than the corresponding exponent $1/2$ in \Cref{res: hellinger stability general log-concave}, and this leads to better convergence rates as $\Rpl \to \mathbf{I}$. In particular, this contributes an $\epsilon^{-2/(1+\alpha)}$ dependence to our sample complexity bounds, as compared to $\epsilon^{-4}$ (see the rows of \Cref{tab: table M} corresponding to S1).
We note that using higher derivatives of $\varphi$ (provided $\varphi \in \mathcal{C}^{s, \alpha}(\mathbb{R}^d)$ for $s > 1$) in the proof could potentially lead to even faster convergence.
Next, we move on to stability under \nameref{para: S2}.
}
% For the next stability result, we can go beyond rotations and consider more general linear maps on $\bR^d$. Given a linear map $\mathbf{A}:\bR^d \to \bR^d$, denote by $\mathbf{A}_{\sharp}p$ the pushforward of $p$ through $\mathbf{A}$, i.e. if $X \sim p$, then $\mathbf{A}X\sim \mathbf{A}_{\sharp}p$. Also, if $\mathbf{A}$ is invertible, then $\mathbf{A}_\sharp p$ has the density $|\det(\mathbf{A}^{-1})|\, \left(p \circ \mathbf{A}^{-1}\right)$, which simplifies to $p \circ \mathbf{A}^T$ when $\mathbf{A}$ is orthogonal.

\begin{lemma}[Stability in TV]
\label{res: linear stability of TV}
Let $p$ be a zero-mean probability density on $\bR^d$, satisfying \nameref{para: S2}. Let $\mathbf{A}\in \bR^{d \times d}$  be invertible with $\op{\mathbf{A}^{-1}} \leq B$. Then,
$$
\mathrm{TV}(p, \mathbf{A}_{\sharp}p) \equiv \frac{1}{2}
\norm{\mathbf{A}_{\sharp}p-p}_{L^1}\leq \frac{1}{2}\left[(1+B)\norm{\nabla p}_{L^1}\sqrt{d} + d\right]\op{\VarX}^{1/4}\,\op{\mathbf{I}-\mathbf{A}}^{1/2}.
$$
\end{lemma}
The proof is based on bounding the Wasserstein-1 distance between $p$ and $\mathbf{A}_{\sharp}p$, accompanied by an argument inspired from \citet{CHAE2020108771} (see \Cref{prf: linear stability of TV}). Again, the bound can be translated to Hellinger distance via the inequality \cite{gibbs2002choosing}
\begin{align*}
    h_d^2(p, \mathbf{A}_{\sharp}p) \leq \mathrm{TV}(p, \mathbf{A}_{\sharp}p).
\end{align*}
Finally, in the case of orthogonal matrices, the parameter $B=1$. 

{Notice that the exponent of $\op{\mathbf{I} - \mathbf{A}}$ in \Cref{res: linear stability of TV} is $1/2$ again -- the same as in \Cref{res: hellinger stability general log-concave}. This contributes an $\epsilon^{-4}$ dependence to our sample complexity bounds (see the rows of \Cref{tab: table M} corresponding to S2). A better exponent could potentially be obtained by requiring control on higher order Sobolev norms and following estimates from KDE theory~\cite{holmstrom1992asymptotic}.}

{
We can now state the analogue of \Cref{res: generic sample complexity proposed} under these stronger stability results. The proof is provided in \Cref{prf: smooth sample complexity proposed}.

\begin{corollary}[Sample complexities under smoothness assumptions]
\label{res: smooth sample complexity proposed}
Let $p$ be a zero-mean probability density on $\mathbb{R}^d$, satisfying the orthogonal independent components property. Let $\hat{p}$ be the proposed estimator computed from samples $X^{(1)}, \dots, X^{(N)}, Y^{(1)}, \dots, Y^{(M)} \overset{\mathrm{iid}}{\sim}p$. Then, for any $\epsilon>0$ and $0 < \gamma<1$, we have 
\begin{equation*}
h_{d}^2(\hat{p},p) \leq \epsilon
\end{equation*}
with probability at least $1-\gamma$ whenever
\begin{subequations}
\begin{equation}
N \geq \sup_{s \in \mathbb{S}^{d-1}}N^*\left( \epsilon /4d,\, \gamma /2d;\, p_{s} \right),
\label{eq: N smooth}
\end{equation}
and
\begin{equation}
M \geq \begin{cases}
    M^*\left(\frac{1}{4} d^{-\frac{2+\alpha}{1+\alpha}}\,[\nabla \varphi]_{\alpha}^{- \frac{1}{1+\alpha}}\,\op{\mathbf{\Sigma}}^{-1/2}\,\epsilon^{\frac{1}{1+\alpha}}\,,\, \gamma/2; \,p \right) & \text{if S1 holds} \\
    M^*\left(4^{-4} \left[ \|\nabla p\|_{L^1} d^{7/4} + d^{9/4} \right]^{-2} \op{\mathbf{\Sigma}}^{-1/2} \epsilon^{2}\,,\, \gamma/2; \,p\right) & \text{if S2 holds}. 
\end{cases}
\label{eq: M smooth}
\end{equation}
\end{subequations}
Here, $N^*$ and $M^*$ are respectively the sample complexities of the univariate density estimator and the unmixing matrix estimator invoked by $\hat{p}$.

\end{corollary}

As stated, \Cref{res: smooth sample complexity proposed} does not require log-concavity of $p$. If $p$ is smooth enough but not log-concave, the univariate density estimation stage of the proposed estimator can be done via \textit{kernel density estimation} (KDE), for example, which would have a different sample complexity $N^*(\epsilon, \gamma, f)$.
% If $p$ is smooth enough, and a suitable univariate density estimator and unmixing matrix estimator are chose

}

\begin{table}[!t]
\caption{Number of samples $N$ that suffice (along with a sufficiently large $M$) for the proposed estimator to achieve $h_d^2(\hat{p}, p) \leq \epsilon$ with probability at least $1-\gamma$. Here, $p$ is assumed to be a zero-mean log-concave density on $\mathbb{R}^d$, satisfying the orthogonal independent components property. The bottom row additionally assumes that all marginals of $p$ are log-$k$-affine.}
    \centering
    {
    \begin{tabular}{c|c}
    \hline
    \textbf{Setting} & \textbf{Number of samples} $N$\\
    \hline
         &  \\
    Log-concave MLE &   $\frac{C_{1}d^{5/4}}{\epsilon^{5/4}} \log^{5/4}\left( \frac{C_{2}d}{\gamma} \right) \vee \frac{C_{3}d}{\gamma}$ \\ 
        &   \\
    Boosted log-concave MLE &   $\frac{Cd^{5/4}}{\epsilon^{5/4}}\log\left( \frac{2d}{\gamma} \right)$ \\
        &   \\
    Adapted log-concave MLE &   $\frac{Ckd^2}{\epsilon\gamma} \quad(\text{up to log factors})$ \\
        &   \\
        \hline
    \end{tabular}
    }
    \label{tab: table N}
\end{table}

\begin{table}[h]
\caption{Number of samples $M$ that suffice (along with a sufficiently large $N$) for the proposed estimator to achieve $h_d^2(\hat{p}, p) \leq \epsilon$ with probability at least $1-\gamma$. Here, $p$ is assumed to be a zero-mean log-concave density on $\mathbb{R}^d$, satisfying the orthogonal independent components property.
Invoking PCA or ICA additionally assumes that \nameref{para: M1} or \nameref{para: M2} is (respectively) satisfied. The term $\tilde{M}(d, \gamma)$ is defined in \Cref{res: estimating W by PCA}. The constants $K', K''$ depend on parameters from \nameref{para: M2}.}
    \centering
    {
    \begin{tabular}{c|c}
    \hline
         \textbf{Setting} &  \textbf{Number of samples} $M$\\
         \hline
            &   \\
         
         PCA & \small{$C_{1}\tilde{M}(d,\gamma) \vee \left\{ \frac{K_{d}' \mathrm{cond}(\mathbf{\Sigma}) \op{\mathbf{\Sigma}}^{2} \log^4(2 /\gamma)}{\delta^2\epsilon^4}  \log^2 \left( \frac{K_{d}'' \mathrm{cond}(\mathbf{\Sigma}) \op{\mathbf{\Sigma}}^{2} \log^2(2 /\gamma)}{\delta^2\epsilon^4}  \right) \right\}$}\\

            &   \\
        ICA & \small{$\frac{K_{d} K'\mathrm{cond}(\mathbf{\Sigma})}{\epsilon^4} \log (6 /\gamma) \,\vee\, K'' d^2 \log^2(6 /\gamma) \,\vee\, (6 /\gamma)^{8/9}$} \\
            &   \\
        PCA, S1 & \small{$C_{1}\tilde{M}(d, \gamma) \vee \left\{\frac{C_{2}\op{\mathbf{\Sigma}}^3 [\nabla \varphi]_{\alpha}^{\frac{2}{1+\alpha}} d^\frac{5+3\alpha}{1+\alpha}\log^4(2/\gamma)}{\delta^2 \epsilon^\frac{2}{1+\alpha}}\log^2 \left(\frac{C_{3}\op{\mathbf{\Sigma}}^3 [\nabla \varphi]_{\alpha}^{\frac{2}{1+\alpha}} d^\frac{4+2\alpha}{1+\alpha}\log^2(2/\gamma)}{\delta^2 \epsilon^\frac{2}{1+\alpha}}\right) \right\}$} \\

            &   \\
        ICA, S1 & \small{$\frac{K' \op{\mathbf{\Sigma}} [\nabla\varphi]_{\alpha}^{\frac{2}{1+\alpha}}\,d^{\frac{5+3\alpha}{1+\alpha}}}{\epsilon^{\frac{2}{1+\alpha}}} \log (6 /\gamma) \,\vee\, K'' d^2 \log^2(6 /\gamma) \,\vee\, (6 /\gamma)^{8/9}$} \\
            &   \\
        PCA, S2 & \small{$C_{1}\tilde{M}(d, \gamma) \vee \left\{\frac{C_{2}\op{\mathbf{\Sigma}}^3 (\norm{\nabla p}_{L^1}^4 d^8 + d^{10})\log^4(2/\gamma)}{\delta^2 \epsilon^4}\log^2 \left(\frac{C_{3}\op{\mathbf{\Sigma}}^3 (\norm{\nabla p}_{L^1}^4 d^7 + d^{9})\log^2(2/\gamma)}{\delta^2 \epsilon^4}\right) \right\}$} \\
            &   \\
        ICA, S2 & \small{$\frac{K' \op{\mathbf{\Sigma}}(\|\nabla p\|_{L^1}^4 d^8 + d^{10})}{\epsilon^4} \log (6 /\gamma) \,\vee\, K'' d^2 \log^2(6 /\gamma) \,\vee\, (6 /\gamma)^{8/9}$} \\
            &   \\
            \hline
\end{tabular}
        }
    \label{tab: table M}
\end{table}

{
A different way in which we get stronger bounds under additional assumptions is \emph{adaptation}. A univariate density $f:\mathbb{R} \to [0, \infty)$ is said to be log-$k$-affine, for some $k \in \mathbb{N}$, if $\log f$ is piece-wise affine with at most $k$ pieces on the support of $f$. \citet{Kim2018Adaptation} established parametric sample complexity of the log-concave MLE when the true density is log-$k$-affine with $k$ fixed. The following an immediate consequence of their Theorem 3. 

\begin{proposition}[Adaptation of log-concave MLE~\cite{Kim2018Adaptation}]
    \label{res: adaptation}
Let $f: \mathbb{R} \to [0, \infty)$ be a log-concave density that is log-$k$-affine for some $k \in \mathbb{N}$, and denote by $\hat{f}$ the log-concave MLE computed from $N$ iid samples from $f$. Given $\epsilon>0$ and $0 < \gamma< 1$, we have
\begin{equation*}
h_{1}^{2}(\hat{f}, f) \leq \epsilon
\end{equation*}
with probability at least $1-\gamma$ whenever
\begin{equation*}
\frac{N}{\log^{5/4}(eN /k)} \geq \frac{Ck}{\epsilon\gamma}
\end{equation*}
for an absolute constant $C>0$.
\end{proposition}

As a consequence, we have that $N^*(\epsilon, \gamma; f)$ is at most $Ck /\epsilon\gamma$ up to log factors when $f$ is log-$k$-affine. The $1 /\gamma$-dependence is due to Markov's inequality, and may be improved to $\log(1 /\gamma)$ via boosting as before (see \Cref{def: boosted}).

We conclude this discussion by collecting some concrete bounds on the number of samples $N$ and $M$ that are sufficient for the proposed estimator to achieve $h_d^2(\hat{p}, p) \leq \epsilon$ with probability at least $1-\gamma$. These depend on the particular procedures invoked for marginal estimation and unmixing matrix estimation, as well as on the smoothness assumptions. The bounds on $N$ are presented in \Cref{tab: table N}, and those for $M$ are tabulated in \Cref{tab: table M}. The results in both tables take $p$ to be a zero-mean, log-concave density on $\mathbb{R}^d$, satisfying the orthogonal independent components property.

}

\subsection{Bounds under misspecification of log-concavity}
\label{subsec: misspecification}
{

Let $\mathcal{P}_{d}$ be the set of probability distributions $P$ on $\mathbb{R}^{d}$ satisfying $\mathbb{E}_{X \sim P} \|X\|_{2} < \infty$ and $\mathbb{P}_{X \sim P}(X \in H) < 1$ for every hyperplane $H \subseteq \mathbb{R}^{d}$. Identical to the setting of densities, we say that a zero-mean distribution $P \in \mathcal{P}_{d}$ satisfies the orthogonal independent components property if there exists an orthogonal matrix $\mathbf{W} \in \mathbb{R}^{d \times d}$ such that the components of $Z=\mathbf{W}X$ are independent when $X \sim P$.

To study the case of misspecification, we need the log-concave projection of \citet{dumbgen_ApproximationLogConcaveDistributions_2011}, which is a map from $\mathcal{P}_{d}$ to the set of log-concave densities $\mathcal{F}_{d}$ on $\mathbb{R}^d$, defined by
\begin{equation*}
\psi_{d}^*(P) := \argmax_{f \in\mathcal{F}_{d}}\,\,\mathbb{E}_{X \sim P}\log f(X).
\end{equation*}
If $P$ is a zero-mean distribution in $\mathcal{P}_{d}$ satisfying the orthogonal independent components property, we show below that our proposed estimator (invoking univariate log-concave MLE), computed from iid samples from $P$, gets close to $\psi_{d}^*(P)$.

\begin{theorem}[Error bound with misspecification of log-concavity]
    \label{res: basic err bound misspecification}
Let $P$ be a zero-mean distribution in $\mathcal{P}_{d}$, satisfying the orthogonal independent components property with independent directions $w_{1}, \dots, w_{d}$. Let $\hat{p}$ be the proposed estimator computed from samples $X^{(1)}, \dots, X^{(N)}, Y^{(1)}, \dots, Y^{(M)} \overset{\mathrm{iid}}{\sim}P$. We have the error bound
\begin{align*}
h_{d}(\hat{p}, \psi_{d}^*(P)) \leq \sqrt{ d } & \left( \max_{i \in [d]} h_{1}(\hat{p}_{\hat{w}_{i}}, \psi_{1}^*(P_{\hat{w}_{i}})) + \max_{i \in[d]} h_{1}(\psi_{1}^*(P_{\hat{w}_{i}}), \psi_{1}^*(P_{w_{i}})) \right)  \\
 & + h_{d}(\psi_{d}^*(\hat{\mathbf{R}}_{\sharp} P), \psi_{d}^*(P)),
\end{align*}
where $\hat{\mathbf{R}}= \hat{\mathbf{W}}^T \mathbf{W}$.

\end{theorem}
See \Cref{prf: basic err bound misspecification} for the proof.

We can now obtain sample complexities; to get concrete bounds, we focus on the ICA case and recall \nameref{para: M2} with parameters $\mu_{4}, \mu_{9}$ and $\kappa$. Given $P \in \mathcal{P}_{d}$ and $X \sim P$, define the moment quantities
\begin{equation*}
\eta_{P} := \inf_{s \in S^{d-1}}\,\mathbb{E}\left| s \cdot (X-\mathbb{E}X) \right|, \quad \text{and}\quad L_{q} := \left( \mathbb{E}\,\|X\|_{2}^q \right)^{1/q}
\end{equation*}
for $q \geq 1$.

\begin{theorem}[Samplex complexity with misspecification of log-concavity]
    \label{res: sample complexity misspecification}
Let $P$ be a zero-mean distribution in $\mathcal{P}_{d}$ satisfying the orthogonal independent components property. Suppose that \nameref{para: M2} is satisfied, $\eta_{P}>0$, and $L_{q} < \infty$ for some $q > 1$. Let $\hat{p}$ be the proposed estimator computed from samples $X^{(1)}, \dots, X^{(N)}, Y^{(1)}, \dots, Y^{(M)} \overset{\mathrm{iid}}{\sim}P$, invoking univariate log-concave MLE and ICA. Then, for any $\epsilon>0$ and $6d^{-3} < \gamma<1$, we have 
\begin{equation*}
h_{d}^2(\hat{p},\psi_{d}^*(P)) \leq \epsilon
\end{equation*}
with probability at least $1-\gamma$ whenever
\begin{equation*}
\frac{N}{\log^{\frac{3q}{q-1}}N} \geq \left( K_{q} \sqrt{ \frac{L_{q}}{\eta_{P}} }\, \frac{8d^2}{\epsilon\gamma} \right)^{\frac{2q}{q-1}} ,
\end{equation*}
and
\begin{equation*}
M \geq \frac{K' C_{d}'' L_{1}^2}{\eta_{P}^2\, \epsilon^4} \log (6 /\gamma) \,\vee\, K'' d^2 \log^2(6 /\gamma) \,\vee\, (6 /\gamma)^{8/9}.
\end{equation*}
Here, the constants $K', K'' >0$ depend on $\mu_{4}, \mu_{9}, \kappa$, the constant $K_{q}>0$ depends only on $q$, and $C_{d}''>0$ depends only on $d$.
\end{theorem}

The proof is provided in \Cref{prf: sample complexity misspecification}. The inequality for $N$ depends polynomially on $\gamma$, and this is due to the use of Markov's inequality in the proof. If the distribution $P$ is sub-exponential, better tails bounds may be used, as discussed by \citet{barberLocalContinuityLogconcave2021}.

}

\begin{figure}[!h]
    \centering
    \includegraphics[width=0.9\textwidth]{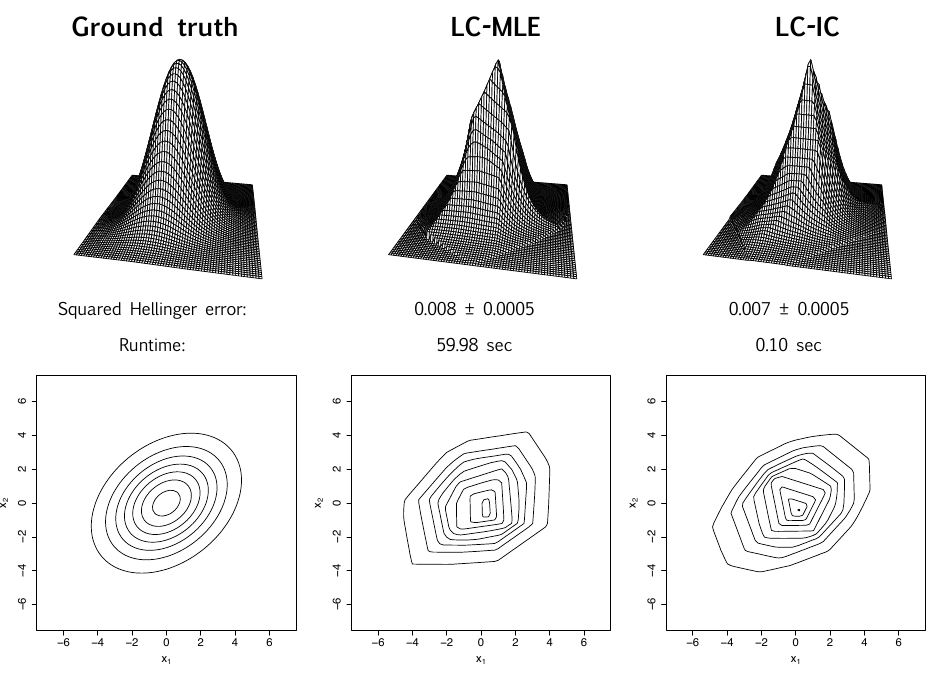}
    \caption{Comparing density estimation performance of our proposed log-concave independent components estimator (LC-IC) with the usual log-concave maximum likelihood estimator (LC-MLE), on simulated bivariate normal data. The ground truth and estimated densities are visualized in the top row, and the corresponding level curves are plotted in the bottom row.}
    \label{fig:compare-visually-gaussian}
\end{figure}

\section{Computational experiments}
\label{sec: comp exp}
In this section, we demonstrate the performance of the log-concave independent components estimator  (LC-IC, \Cref{alg: main algo}) on simulated as well as real data. \Cref{subsec: gaussian comparison tests} compares the estimation errors and runtimes of LC-IC with those of usual multivariate log-concave MLE {and the LogConICA algorithm of }\citet{samworth2012independent}, on simulated data. \Cref{subsec: split-r-results} assesses the effect of the sample splitting proportions between the two tasks: estimation of $\eigmat$ ($M$ samples) and estimation of the marginals ($N$ samples). {We also test the case of no sample splitting.} Finally, \Cref{subsec: mixtures} combines LC-IC with the Expectation-Maximization (EM) algorithm \cite{dempster1977maximum} to estimate mixtures of densities from real data -- the Wisconsin breast cancer dataset \cite{misc_breast_cancer_wisconsin_(diagnostic)_17}.
% and one slice of the mouse organogenesis spatiotemporal transcriptomic atlas (MOSTA) dataset \cite{chen2022spatiotemporal}. 

\subsection{Comparison of performance on simulated data}
\label{subsec: gaussian comparison tests}

\begin{table}[!t]
\centering
\caption{Comparing average algorithm runtimes in \textbf{seconds} of our proposed log-concave independent components estimator (LC-IC) with the usual log-concave maximum likelihood estimator (LC-MLE), on simulated multivariate normal data. Here, $n$ is the number of samples, and $d$ is the dimension.}
\begin{tblr}{
  row{1} = {c},
  cell{2}{1} = {c},
  cell{2}{3} = {r},
  cell{2}{4} = {r},
  cell{2}{5} = {r},
  cell{2}{6} = {r},
  cell{2}{7} = {r},
  cell{3}{1} = {c},
  cell{3}{3} = {r},
  cell{3}{4} = {r},
  cell{3}{5} = {r},
  cell{3}{6} = {r},
  cell{3}{7} = {r},
  cell{4}{1} = {c},
  cell{4}{3} = {r},
  cell{4}{4} = {r},
  cell{4}{5} = {r},
  cell{4}{6} = {r},
  cell{4}{7} = {r},
  cell{5}{1} = {c},
  cell{5}{3} = {r},
  cell{5}{4} = {r},
  cell{5}{5} = {r},
  cell{5}{6} = {r},
  cell{5}{7} = {r},
  cell{6}{1} = {c},
  cell{6}{3} = {r},
  cell{6}{4} = {r},
  cell{6}{5} = {r},
  cell{6}{6} = {r},
  cell{6}{7} = {r},
  cell{7}{1} = {c},
  cell{7}{3} = {r},
  cell{7}{4} = {r},
  cell{7}{5} = {r},
  cell{7}{6} = {r},
  cell{7}{7} = {r},
  hline{2,8} = {-}{},
  hline{4,6} = {1-2}{},
}
      &        & $n=100$ & $n=500$ & $n=1000$ & $n=2000$ & $n=3000$ \\
$d=2$ & LC-MLE & 0.56    & 11.93   & 60.79    & 469.65   & 1703.08  \\
      & LC-IC  & 0.02    & 0.04    & 0.10     & 0.09     & 0.12     \\
$d=3$ & LC-MLE & 1.82    & 42.00   & 141.31   & 874.58   & 3183.45  \\
      & LC-IC  & 0.02    & 0.06    & 0.08     & 0.13     & 0.17     \\
$d=4$ & LC-MLE & 8.70    & 155.03  & 735.06   & 4361.81  & 12933.45 \\
      & LC-IC  & 0.03    & 0.07    & 0.11     & 0.16     & 0.24     
    \label{tab: runtimes-gaussian}
\end{tblr}
\end{table}

\begin{figure}[!b]
    \centering
    \includegraphics[width=\textwidth]{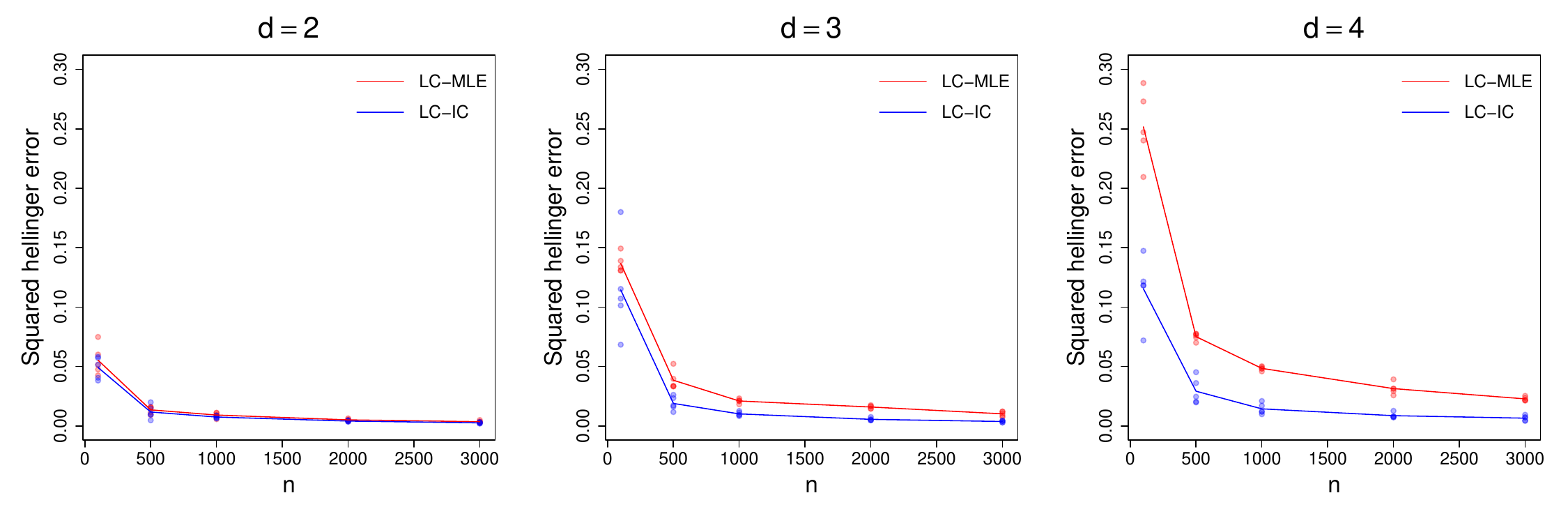}
    \caption{Comparing estimation errors in squared Hellinger distance of our proposed log-concave independent components estimator (LC-IC) with the usual log-concave maximum likelihood estimator (LC-MLE), on simulated multivariate normal data. The dots correspond to independent experiments, and the lines represent averages.
    Here, $n$ is the number of samples, and $d$ is the dimension.}
    \label{fig: error-curves-comparisons-gaussian}
\end{figure}

We first demonstrate estimation performance on bivariate normal data, and provide accompanying visualizations of the estimated density. Recalling that $X = \eigmat^T Z$, we generate $Z \sim \mathcal{N}(0, \VarZ)$ with $\VarZ = \mathrm{Diag}(6, 3)$, and choose the orthogonal matrix $\eigmat$ at random. With $n=1000$ independent samples of $X$ generated, we equally split $M=N=500$, and compute the LC-IC estimate using PCA to estimate $\eigmat$ and \texttt{logcondens} to estimate the marginals. From the same $n=1000$ samples, we also compute the usual log-concave maximum likelihood estimate (LC-MLE) using the \texttt{R} package \texttt{LogConcDEAD} \cite{logconcdead_ref}. \Cref{fig:compare-visually-gaussian} shows the ground truth and estimated densities together with their contour plots, and lists the estimation errors in squared Hellinger distance along with the algorithm runtimes. The integrals needed to obtain the squared Hellinger distances are computed via a simple Monte-Carlo procedure (outlined in \ref{subsec: app monte carlo}); the associated uncertainty is also presented in \Cref{fig:compare-visually-gaussian}. 

{

\begin{table}[!t]
\centering
\caption{Comparing average algorithm runtimes in \textbf{seconds} of our proposed log-concave independent components estimator (LC-IC) with the LogConICA method of \citet{samworth2012independent}, on simulated $\mathrm{Unif}([-1,1]^d)$ data. Here, $n$ is the number of samples, and $d$ is the dimension. For LogConICA, the reported runtimes are averages over 10 random initializations.}
{
\begin{tblr}{
  row{1} = {c},
  cell{2}{1} = {c},
  cell{2}{3} = {r},
  cell{2}{4} = {r},
  cell{2}{5} = {r},
  cell{2}{6} = {r},
  cell{2}{7} = {r},
  cell{3}{1} = {c},
  cell{3}{3} = {r},
  cell{3}{4} = {r},
  cell{3}{5} = {r},
  cell{3}{6} = {r},
  cell{3}{7} = {r},
  cell{4}{1} = {c},
  cell{4}{3} = {r},
  cell{4}{4} = {r},
  cell{4}{5} = {r},
  cell{4}{6} = {r},
  cell{4}{7} = {r},
  cell{5}{1} = {c},
  cell{5}{3} = {r},
  cell{5}{4} = {r},
  cell{5}{5} = {r},
  cell{5}{6} = {r},
  cell{5}{7} = {r},
  cell{6}{1} = {c},
  cell{6}{3} = {r},
  cell{6}{4} = {r},
  cell{6}{5} = {r},
  cell{6}{6} = {r},
  cell{6}{7} = {r},
  cell{7}{1} = {c},
  cell{7}{3} = {r},
  cell{7}{4} = {r},
  cell{7}{5} = {r},
  cell{7}{6} = {r},
  cell{7}{7} = {r},
  hline{2,8} = {-}{},
  hline{4,6} = {1-2}{},
}
      &        & $n=100$ & $n=500$ & $n=1000$ & $n=2000$ & $n=3000$ \\
$d=2$ & LogConICA & 1.27    & 2.87   & 9.17    & 24.92   & 41.64  \\
      & LC-IC  & 0.02    & 0.01    & 0.02     & 0.04     & 0.06     \\
$d=3$ & LogConICA & 0.38    & 2.99   & 4.01   & 15.99   & 31.35  \\
      & LC-IC  & 0.02    & 0.03    & 0.05     & 0.09     & 0.10     \\
$d=4$ & LogConICA & 0.44    & 1.49  & 3.22   & 7.60  & 12.33 \\
      & LC-IC  & 0.02    & 0.04    & 0.10     & 0.12     & 0.18     
    \label{tab: runtimes-unif-samworth-yuan}
\end{tblr}
}
\end{table}

\begin{figure}[!b]
    \centering
    \includegraphics[width=\textwidth]{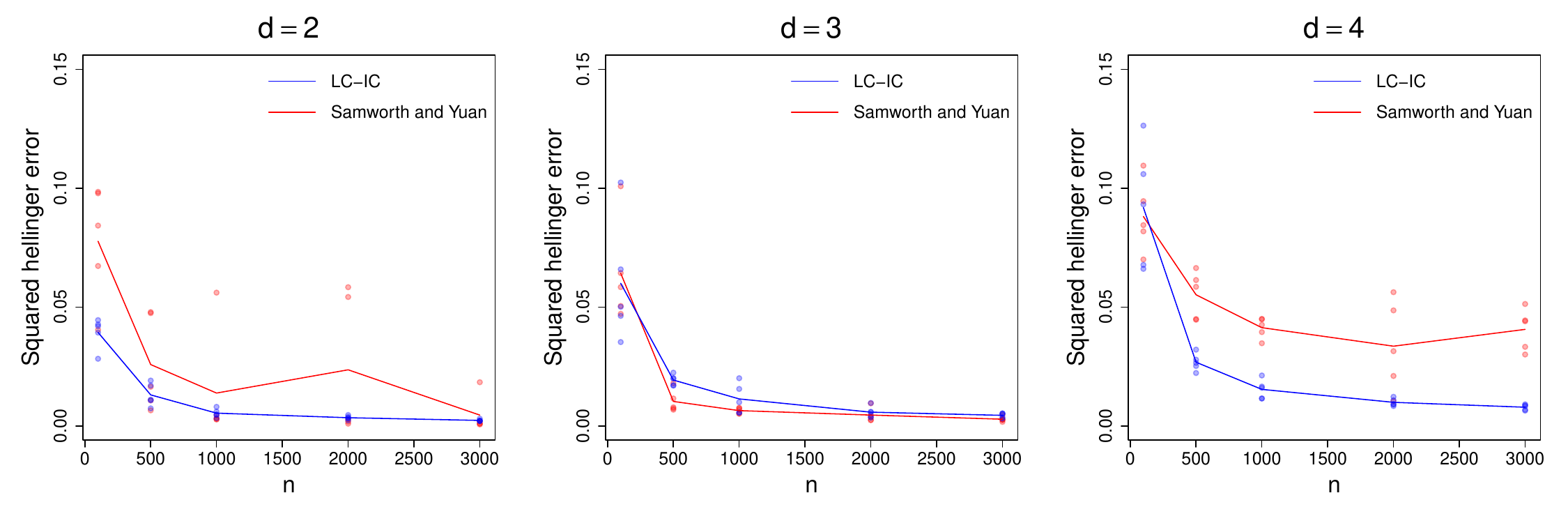}
    \caption{Comparing estimation errors in squared Hellinger distance of our proposed log-concave independent components estimator (LC-IC) with the LogConICA method of \citet{samworth2012independent}, on simulated $\mathrm{Unif}([-1,1]^d)$ data. The dots correspond to independent experiments, and the lines represent averages. Here, $n$ is the number of samples, and $d$ is the dimension.}
    \label{fig: unif ica samworth-yuan}
\end{figure}

}

Next, we conduct systematic comparison tests on simulated data from multivariate normal distributions in $d=2,3,4$, with various values of $n$. As before, set $X = \eigmat^T Z$ with $\eigmat$ chosen at random and  $Z \sim \mathcal{N}(0, \VarZ)$, where $\VarZ = \mathrm{Diag}(15,14)$ for $d=2$, $\VarZ = \mathrm{Diag}(15, 14, 13)$ for $d=3$, and $\VarZ = \mathrm{Diag}(15, 14, 13, 12)$ for $d=4$. PCA is used for estimating $\eigmat$. Each experiment is repeated 5 times and the average metrics are presented. Algorithm runtime comparisons are given in \Cref{tab: runtimes-gaussian}, whereas the estimation errors in squared Hellinger distance are plotted in \Cref{fig: error-curves-comparisons-gaussian}. Notice that LC-IC runtimes are several orders of magnitude smaller than LC-MLE runtimes, especially for large $n$. Also, the estimation errors of LC-IC are lower than those of LC-MLE, with a growing margin as $d$ increases. Similar tests on Gamma-distributed data, are provided in \ref{subsec: gamma comparison tests}.

{

These results are not surprising as the multivariate log-concave MLE is a fully non-parametric procedure that makes no assumptions on $p$ except for log-concavity. A fairer comparison is with the LogConICA algorithm of \citet{samworth2012independent}. On simulated data from the uniform distribution $\mathrm{Unif}([-1,1]^d)$ with $d=2,3,4$, we compare LC-IC (invoking the kurtosis-based FastICA implementation of \texttt{ica}~\cite{ica_package_R} to estimate $\eigmat$, and \texttt{logcondens} to estimate the marginals) with our implementation of LogConICA (using 10 random initializations and a maximum of 20 iterations per initialization).  As before, each experiment is repeated 5 times, and the average metrics are presented. Algorithm runtime comparisons are given in \Cref{tab: runtimes-unif-samworth-yuan}, where we report the average times over the random initializations for LogConICA. The estimation errors in squared Hellinger distance are plotted in \Cref{fig: unif ica samworth-yuan}.

LC-IC still has much shorter runtimes as expected, because it estimates the marginals only once, whereas LogConICA updates those for multiple iterations. Low LogConICA runtimes for larger $d$ is peculiar however; this appears to be because the alternating scheme can get stuck quickly at bad points and subsequently terminate. Next, notice that LogConICA can achieve smaller errors than LC-IC (see the plot for $d=3$), but may occasionally fail to produce good estimates due to getting stuck at bad points. This could be resolved by better-than-random initialization; one may use our LC-IC estimate as a fast and accurate initial point for LogConICA (see \Cref{fig: LogConICA init}).

\begin{figure}
    \centering
    \includegraphics[width=\textwidth]{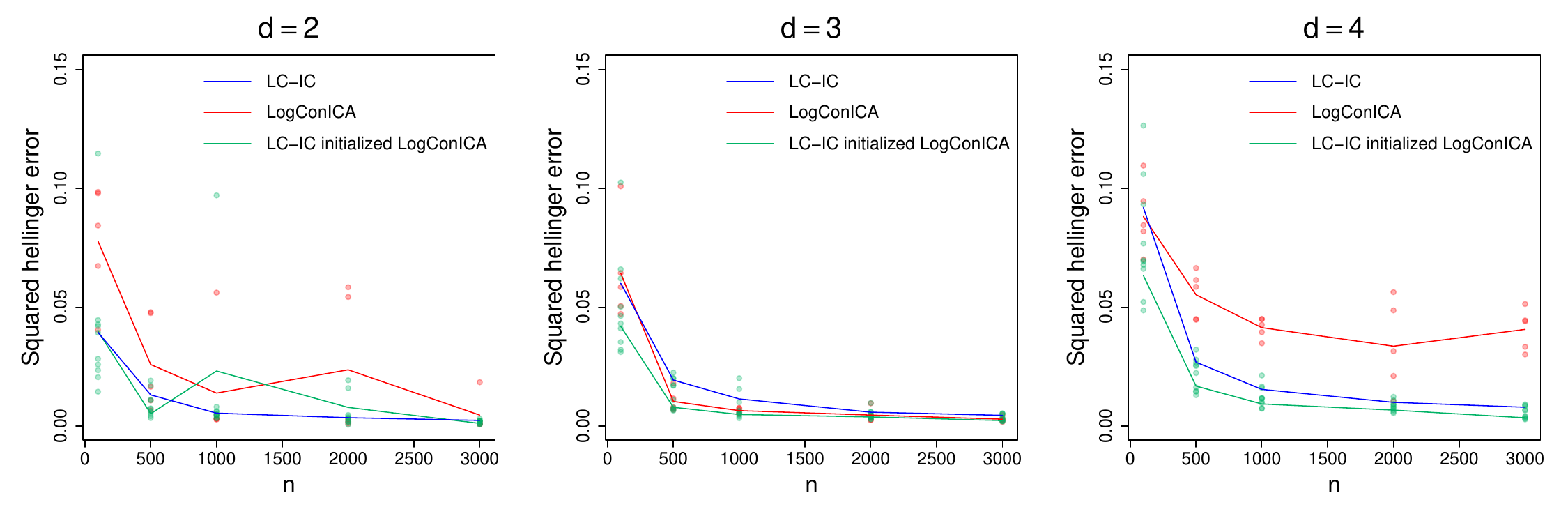}
    \caption{Comparing estimation errors in squared Hellinger distance of LC-IC, LogConICA with random initialization, and LogConICA initialized with LC-IC, on simulated $\mathrm{Unif}([-1,1]^d)$ data. The dots correspond to independent experiments, and the lines represent averages. Here, $n$ is the number of samples, and $d$ is the dimension.}
    \label{fig: LogConICA init}
\end{figure}

}

\subsection{Effect of sample splitting proportions}
\label{subsec: split-r-results}
We assess the effect of changing the sample splitting proportions between $M$ and $N$ on the estimation errors, for multivariate normal data in various dimensions. Using the total number of samples $n=M + N$, we define the parameter $r = M/n$ so that the extreme $r = 0$ corresponds to all samples being used for estimating the marginals, and $r=1$ corresponds to all samples being used for estimating $\eigmat$. 
\begin{figure}[!b]
    \centering
    \includegraphics[width=\textwidth]{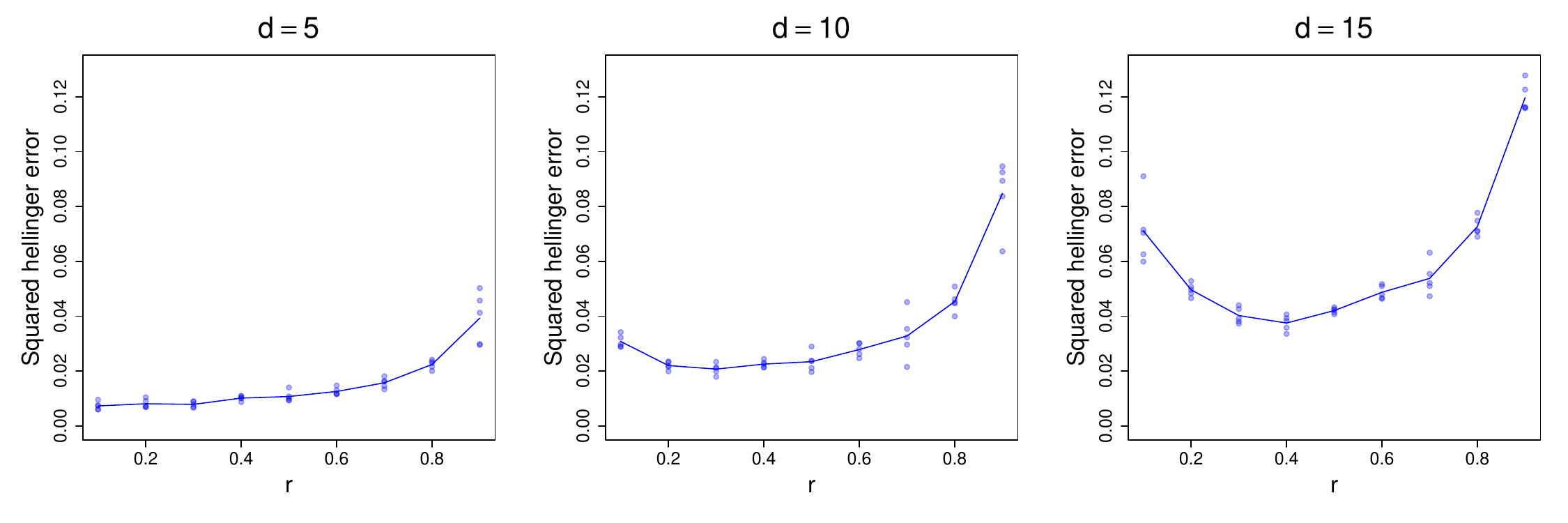}
    \caption{Demonstrating the effect of the sample splitting ratio $r=M/n$ on estimation error in squared Hellinger distance, for multivariate normal data in different dimensions $d$. The dots correspond to independent experiments, and the lines represent averages.}
    \label{fig:three_r_curves}
\end{figure}

As before, $X = \eigmat^T Z$ with the orthogonal matrix $\eigmat$ chosen at random and $Z \sim \mathcal{N}(0,\VarZ)$, where for dimension $d \in \{2, 3, \dots, 9, 10, 15\}$ we set $\VarZ = \mathrm{Diag}(15, 14, \dots, 15 - (d-1))$. With $n=2000$ total samples, the values $r \in \{0.1,0.2, \dots, 0.9\}$ are tested. PCA is used for estimating $\eigmat$. Each experiment is repeated 5 times independently, and the resulting estimation errors in squared Hellinger distance are presented in \Cref{fig:three_r_curves} for $d=5,10,15$.

Let $r^\star$ denote the minimizer of an $r$-versus-error curve for a fixed $d$, and notice that in \Cref{fig:three_r_curves}, $r^\star$~shifts to the right as $d$ increases. This trend is tabulated for all values of $d$ tested in \Cref{tab: r star}. The observed shift is expected since the bounds for $M$ (see \Cref{tab: table M}) have a stronger dependence on $d$ compared to the bounds for $N$ (see \Cref{tab: table N}). 
However, we remark that the effect of $d$ has not been completely isolated in the above experiment. The increase in $d$ is accompanied by a corresponding decrease in $\lambda_{\min}(\VarX)$, which worsens the smoothness (or conditioning) of the density. Keeping $\lambda_{\min}(\VarX)$ fixed, on the other hand, would require shrinking the eigengap $\delta$ with $d$.

{
Finally, we test the case of \emph{no sample splitting} in the setting described above. Here, all $n=2000$ samples are used for both marginal estimation and unmixing matrix estimation. We compare the resulting squared Hellinger errors with those of the usual sample splitting variant with $r=0.4$, and present the results in \Cref{fig: no split compare}. We see, somewhat surprisingly, that the variant without sample splitting achieves lower errors. This suggests that any cross-interactions between the marginal and unmixing matrix estimation stages are mild, and both stages benefit from seeing more samples. 
Our theory from \Cref{sec: analysis}
does not directly work for this variant without sample splitting, however.}
\begin{table}[t]
\centering
\caption{Dependence of the error-minimizing sample splitting ratio $r^\star$ on dimension $d$ for simulated multivariate normal data.}
\vspace{0.5em}
\begin{tabular}{l|llllllllll}
$d$       & 2   & 3& 4& 5& 6& 7& 8& 9& 10&15\\ 
\hline
$r^\star$ & 0.1 & 0.1 & 0.1& 0.1& 0.2& 0.2& 0.3& 0.3& 0.3& 
0.4\end{tabular} \label{tab: r star}
\end{table}

\begin{figure}[h]
    \centering
    \includegraphics[width=\textwidth]{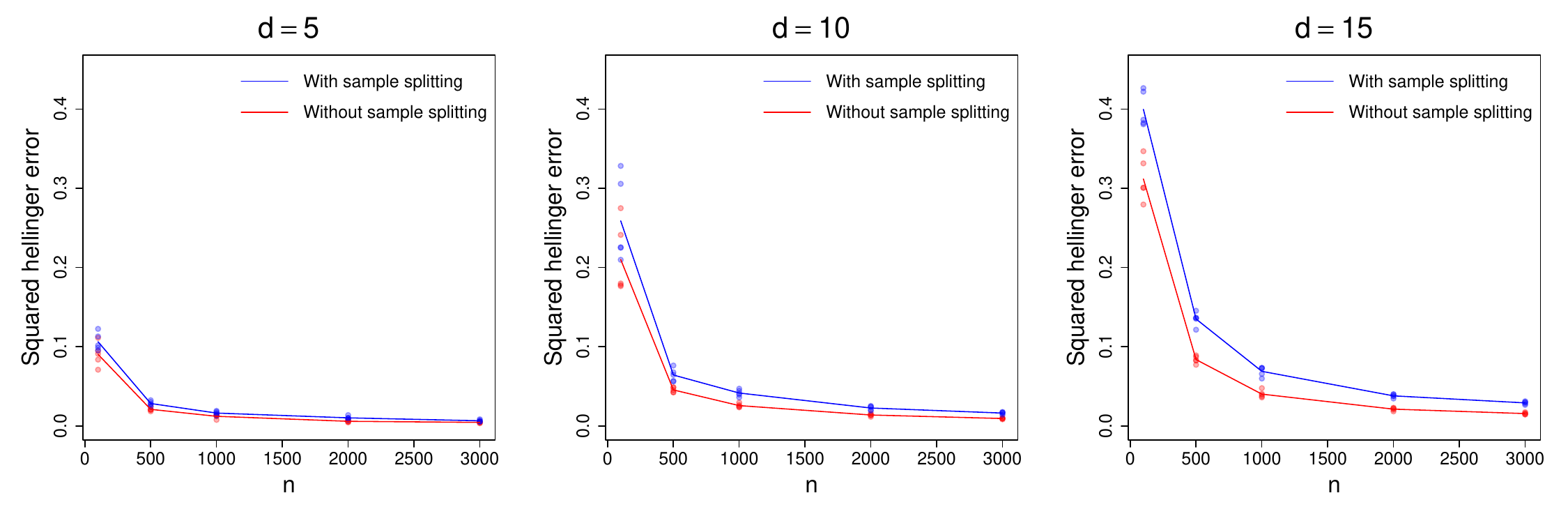}
    \caption{Comparing LC-IC variants with and without sample splitting, on multivariate normal data. The variant with sample splitting uses $r=0.4$. The dots correspond to independent experiments, and the lines represent averages.}
    \label{fig: no split compare}
\end{figure}

\subsection{Estimating mixtures of densities and clustering}
\label{subsec: mixtures}

Consider the problem of estimating finite mixture densities of the form
\begin{align}
    \label{eq: mixture}
    p(x) = \sum_{k=1}^K \pi_k p_k(x),
\end{align}
where the mixture proportions $\pi_1, \dots, \pi_k$ are nonnegative and sum to 1, and $p_1, \dots, p_k: \bR^d \to \bR$ are the component densities corresponding to different clusters. If the component densities are normal, the standard Gaussian EM (expectation maximization) algorithm can be used to estimate $p$. \citet{chang2007clustering} allowed for the component densities to be univariate log-concave by combining log-concave MLE with the EM algorithm. They also considered a normal copula for the multivariate case.
This was later extended to general multivariate log-concave component densities by \citet{cule2010maximum}.

Here, we consider mixture densities of the form \eqref{eq: mixture} where the component densities $p_1, \dots, p_K$ (when centered to mean-zero positions) satisfy the log-concave independent components property. By combining the log-concave independent components estimator (LC-IC) with the EM algorithm, we can port the aforementioned computational benefits to this mixture setting.

We follow the general strategy in \citet[Section 6.1]{cule2010maximum}, replacing their log-concave MLE step with LC-IC. Such a replacement is not completely obvious however. The log-concave MLE in this setting involves weighted samples, which are not immediately compatible with our sample splitting approach. Instead, we use a re-sampling heuristic as outlined in \ref{subsec: EM resampling}.

\begin{figure}[!b]
    \centering
    \includegraphics[width=\textwidth]{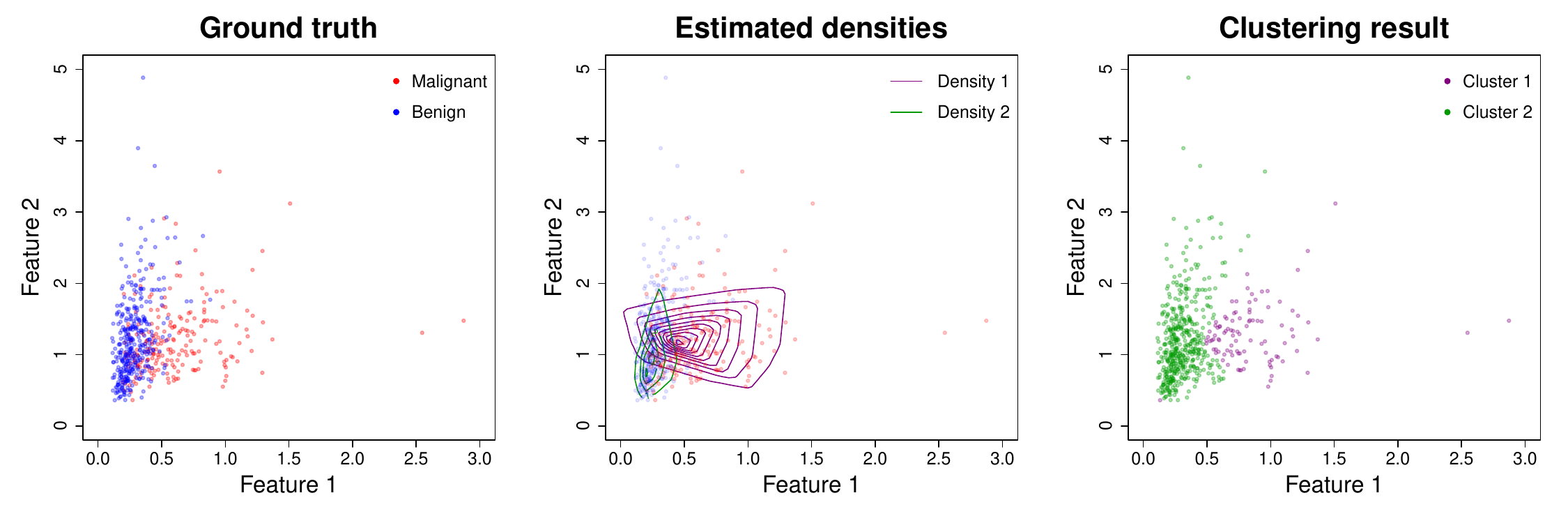}
    \caption{Estimating mixture densities and clustering using two features from the Wisconsin breast cancer dataset \cite{misc_breast_cancer_wisconsin_(diagnostic)_17}. Labeling Cluster 1 as Malignant and Cluster 2 as Benign gives a 77.7\% classification accuracy.}
    \label{fig:clustering-2d}
\end{figure}

Estimated mixture densities can be used for unsupervised clustering, and we test this on real data. To compare the performance of our EM + LC-IC against the results from \citet[Section 6.2]{cule2010maximum}, we use the Wisconsin breast cancer dataset \cite{misc_breast_cancer_wisconsin_(diagnostic)_17}.  
The dataset consists of 569 observations (individuals) and 30 features, where each observation is labeled with a diagnosis of \emph{Malignant} or \emph{Benign}. The test involves clustering the data without looking at the labels, and then comparing the learnt clusters with the ground truth labels.

First, we use the same two features as \citet{cule2010maximum}: Feature 1 -- standard error of the cell nucleus radius, and Feature 2 -- standard error of the cell nucleus texture. This gives bivariate data ($d=2$) with $n=569$ samples. 
We employ EM + LC-IC with $K=2$ target clusters, random initialization, and 20 EM iterations. Again, PCA is used for the $\eigmat$-estimation steps. \Cref{fig:clustering-2d} shows the data with ground truth labels (left-most plot), level curves of the estimated component densities (center plot), and the learnt clusters (right-most plot). Labeling Cluster 1 as Malignant and Cluster 2 as Benign attains a 77.7\% classification accuracy (i.e. 77.7\% of the instances are correctly labeled by the learnt clusters). This is close to the 78.7\% accuracy reported in \citet{cule2010maximum} (121 misclassified instances).

\begin{figure}[!t]
    \centering
    \includegraphics[width=\textwidth]{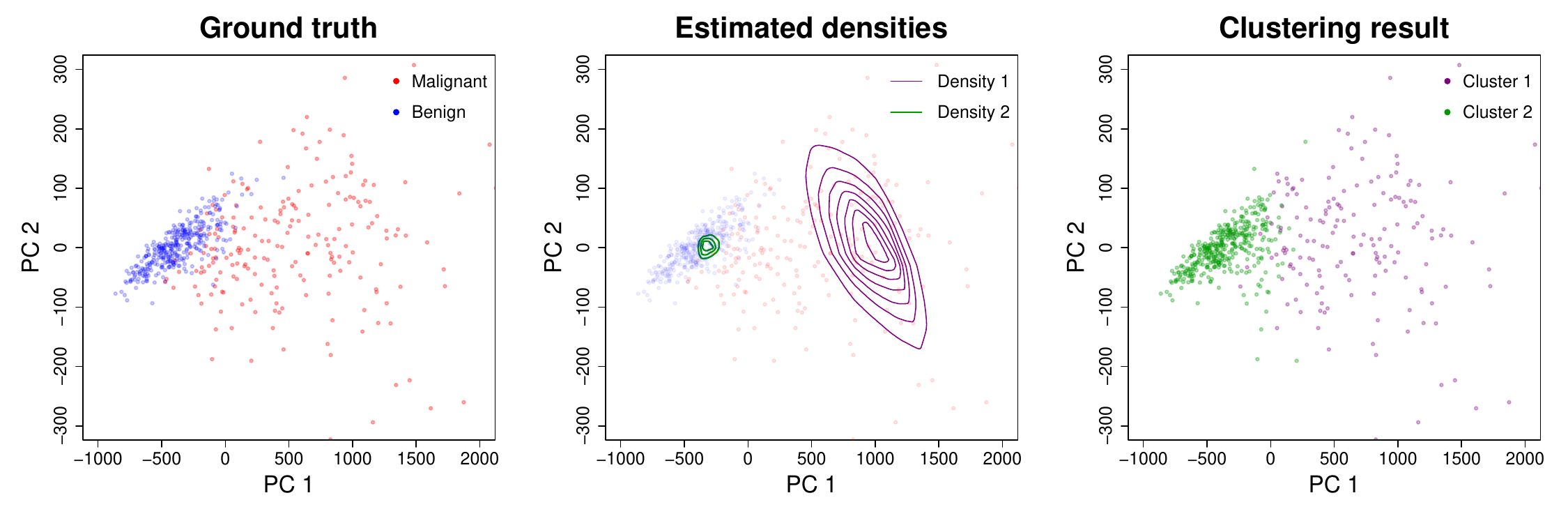}
    \caption{Estimating mixture densities and clustering using all 30 features from the Wisconsin breast cancer dataset \cite{misc_breast_cancer_wisconsin_(diagnostic)_17}. The first two principal components are plotted.
    Labeling Cluster 1 as Malignant and Cluster 2 as Benign gives an 89.6\% classification accuracy.}
    \label{fig:clustering-30d}
\end{figure}

Next, we show improved clustering performance by using more features. Due to the scalability of LC-IC to higher dimensions, one can use all 30 features available in the dataset, so that $d=30$. We employ EM + LC-IC with $K=2$ target clusters as before, choose an initialization based on agglomerative hierarchical clustering (using the \texttt{R} package \texttt{mclust}~\cite{mclust-ref}), and use 15 EM iterations. The results can be visualized in a PCA subspace of the data -- the span of the two leading PCA eigenvectors (or loadings). This particular PCA is only used for visualization, and is \textit{separate} from the PCA used in estimating $\eigmat$.

\Cref{fig:clustering-30d} shows the data in this PCA subspace with ground truth labels (left-most plot), level curves of the estimated component densities (center plot), and the learnt clusters (right-most plot). Labeling Cluster 1 as Malignant and Cluster 2 as Benign achieves an 89.6\% classification accuracy, which is a significant improvement over the results for $d=2$. In comparison, agglomerative hierarchical clustering on its own reaches an accuracy of 86.3\%.

\section{Discussion}
We have looked at statistical and computational aspects of estimating log-concave densities satisfying the orthogonal independent components property. First, we proved that up to {$\tilde{\mathcal{O}}_d(\epsilon^{-4})$} iid samples (suppressing constants and log-factors) are sufficient for our proposed estimator $\hat{p}$ to be $\epsilon$-close to $p$ with high confidence. Although squared Hellinger distance was used here, similar results could have been obtained for total variation (TV) distance since $\mathrm{TV}(\hat{p}, p) \lesssim h_d(\hat{p}, p)$. Note that we do not claim these sample complexity bounds to be tight. {The focus here was on demonstrating univariate rates, and a mild dependence on $d$ under additional assumptions.}
% polynomial dependence on $d$, and it could certainly be possible for the power of $d$ to be reduced further.

The results here could potentially be processed in the framework of \citet{ashtiani2018nearly}, who prove that any family of distributions admitting a so-called sample compression scheme can be learned with few samples. Furthermore, such compressibility extends to products and mixtures of that family.
Hence, demonstrating sample compression in our setting would allow extending the sample complexities to mixtures of log-concave independent components densities, providing some theoretical backing to the results from \Cref{subsec: mixtures}.

Next, the stability results from \Cref{subsec: additional assumptions} could be of further theoretical interest, providing geometric insights into families of densities obtained from orthogonal or linear transformations of product densities. In particular, these could allow one to port covering and bracketing entropies from $d=1$ to higher dimensions by discretizing the operator norm ball in $\bR^{d\times d}$. Also note that these stability bounds are not necessarily tight -- consider a spherically symmetric density $q$ so that $h_d^2(q, \Rpl_{\sharp}q)=0$ for all orthogonal matrices $\Rpl$. Our current analysis does not make use of such symmetry properties. 

On the computational side, we showed that our algorithm has significantly smaller runtimes compared to the usual log-concave MLE {as well as the LogConICA algorithm of} \citet{samworth2012independent}. This was achieved by breaking down the higher dimensional problem into several 1-dimensional ones. In fact these 1-dimensional marginal estimation tasks are decoupled from one another, allowing for effective parallelization to further reduce runtimes. Finally, we demonstrated the scalability of our algorithm by performing density estimation with clustering in $d=30$.

\section*{Acknowledgments}
We thank the anonymous referees for several helpful suggestions, and also thank Yaniv Plan for fruitful discussions.
Elina Robeva was supported by an NSERC Discovery Grant DGECR-2020-00338.

Parts of this paper were written using the ``Latex Exporter" plugin for Obsidian.

\bibliography{refs}

\newpage
\appendix

\section{Delayed proofs}
\label{sec: app A}
Here, we present the proofs that were deferred in the main text. We will need some additional notation. Let $\mathcal{P}_{d}$ be the set of probability distributions $P$ on $\mathbb{R}^{d}$ satisfying $\mathbb{E}_{X \sim P} \|X\|_{2} < \infty$ and $\mathbb{P}_{X \sim P}(X \in H) < 1$ for every hyperplane $H \subseteq \mathbb{R}^{d}$. The log-concave projection of \citet{dumbgen_ApproximationLogConcaveDistributions_2011} is a map from $\mathcal{P}_{d}$ to the set of log-concave densities $\mathcal{F}_{d}$ on $\mathbb{R}^d$, defined by
\begin{equation*}
\psi_{d}^*(P) := \argmax_{f \in\mathcal{F}_{d}}\,\,\mathbb{E}_{X \sim P}\log f(X).
\end{equation*}
Also define, given $X \sim P \in \mathcal{P}_d$ the minimal first moment
$$
\eta_{P} := \inf_{s \in \mathbb{S}^{d-1}}\,\mathbb{E}\left| s \cdot (X-\mathbb{E}X) \right|.
$$

\subsection{Proof of \Cref{res: basic error bound}}
\label{prf: basic error bound}
We recall the statement below.
\paragraph{Statement}
{Suppose $p$ is a zero-mean probability density on $\mathbb{R}^d$, satisfying the orthogonal independent components property with independent directions $w_{1}, \dots, w_{d}$.} Let $\hat{p}$ be the proposed estimator (\Cref{def: proposed estimator}). We have the error bound
\begin{align*}
    h_d(\hat{p}, p) &\leq \sqrt{d} \, \max_{i\in [d]} h_1(\hat{p}_{\hat{w}_i}, p_{\hat{w}_i}) 
    + {\sqrt{d}\,h_d\left(\hat{\mathbf{R}}_{\sharp}p \, , \, p \right) + h_d\left(\hat{\mathbf{R}}^T_{\sharp}p \, , \, p \right)},
\end{align*}
where $\hat{\mathbf{R}} = \eigmat^T \hat{\eigmat}$.    

\begin{proof}
For $s \in \mathbb{S}^{d-1}$, denote the projection $\pi_s (x) = s \cdot x$. This allows us to write the proposed estimator as 
\begin{align}
    \hat{p}(x) = \prod_{i=1}^d \hat{p}_{\hat{w}_i}(\hat{w}_i \cdot x) = \prod_{i=1}^d \hat{p}_{\hat{w}_i} \circ \pi_{\hat{w}_i} (x).
\end{align}
We can break $h_d(\hat{p}, p)$ down into three terms that need to be controlled:
\begin{align}
\begin{aligned}
    h_d(\hat{p}, p) &= h_d\left( \prod_{i} \hat{p}_{\hat{w}_i} \circ \pi_{\hat{w}_i} \, , \, \prod_{i} p_{w_i} \circ \pi_{w_i} \right)  \\ 
    &\leq h_d\left( \prod_{i} \hat{p}_{\hat{w}_i} \circ \pi_{\hat{w}_i} \, , \, \prod_{i} p_{\hat{w}_i} \circ \pi_{\hat{w}_i} \right)  \\ 
    &\phantom{\leq} + h_d\left( \prod_{i} p_{\hat{w}_i} \circ \pi_{\hat{w}_i} \, , \, \prod_{i} p_{w_i} \circ \pi_{\hat{w}_i} \right)
    + h_d\left(\prod_{i} p_{w_i} \circ \pi_{\hat{w}_i}  \, , \, \prod_{i} p_{w_i} \circ \pi_{w_i} \right).
\end{aligned}
\label{eq: hellinger split}
\end{align}
This is essentially \textit{removing hats} one-at-a-time. The first term is handled as
\begin{align}
h_d^2\left( \prod_{i} \hat{p}_{\hat{w}_i} \circ \pi_{\hat{w}_i} \, , \, \prod_{i} p_{\hat{w}_i} \circ \pi_{\hat{w}_i} \right) \leq d \, \max_{i\in [d]} h_1^2(\hat{p}_{\hat{w}_i}, p_{\hat{w}_i}),
\end{align}
using the simple tensorization argument from \Cref{lemma: h oracle}. Since  $\prod_{i} p_{w_i} \circ \pi_{\hat{w}_i} (x) = p(\hat{\mathbf{R}}x)$, {which is the density of $\hat{\mathbf{R}}_{\sharp}^T p$,} we get that the third term equals {$h_d(\hat{\mathbf{R}}_{\sharp}^T p, p)$}.

% Bounding the second term requires the following lemma.
% \begin{lemma}
%     \label{lem: hellinger 1d bounded by d}
%     Let $p$ and $w_1, \dots, w_d$ be defined as before, and consider any other probability density $q:\bR^d \to [0, \infty)$. Then,
%     \begin{align*}
%         h_1^2 (p_{w_k}, q_{w_k}) \leq h_d^2(p,q),
%     \end{align*}
%     for $k\in [d]$.
% \end{lemma}
% The above lemma bounds the 1-dimensional Hellinger distance between marginals by the $d$-dimensional Hellinger distance between the full densities. This is a special case of the \emph{data processing inequality} for $f$-divergences, but we nevertheless provide a proof in \ref{prf: hellinger 1d bounded by d}. 
Finally, consider the second term. Noting that $p_{\hat{w}_k} = (\hat{\mathbf{R}}_{\sharp} p )_{w_k}$ and using \Cref{lem: hellinger 1d bounded by d} gives
    \begin{align}
    \label{eq: 1d bound by d}
        h_1^2(p_{\hat{w}_i}, p_{w_i}) = h_1^2\left((\hat{\mathbf{R}}_{\sharp} p)_{w_i}, \, p_{w_i}\right) \leq h_d^2\left(\hat{\mathbf{R}}_{\sharp} p, \, p\right).
    \end{align}
Now, to bound the second term, note that
    \begin{align*}
        1 - h_d^2\left( \prod_{i} p_{\hat{w}_i} \circ \pi_{\hat{w}_i} \, , \, \prod_{i} p_{w_i} \circ \pi_{\hat{w}_i} \right)
        &= \int_{\bR^d} \sqrt{\prod_{i} p_{\hat{w}_i}(\hat{w}_i \cdot x) \, p_{w_i}(\hat{w}_i \cdot x) }\, dx \\
        &= \prod_{i} \int_{\bR} \sqrt{p_{\hat{w}_i}(z'_i) \, p_{w_i}(z'_i) }\, dz'_i,
    \end{align*}
via the change of variables $z' = \hat{\eigmat}x$. The above expression can be lower bounded using (\ref{eq: 1d bound by d}):
    \begin{align*}
        \prod_{i} \int_{\bR} \sqrt{p_{\hat{w}_i}(z'_i) \, p_{w_i}(z'_i) }\, dz'_i &= \prod_{i} \left( 1 - h_1^2(p_{\hat{w}_i}, p_{w_i}) \right)
        \geq \prod_{i} \left(1 - h_d^2\left(\hat{\mathbf{R}}_{\sharp} p, \, p\right) \right) \\
        &= \left(1 - h_d^2\left(\hat{\mathbf{R}}_{\sharp} p, \, p\right) \right)^{d}
        \geq 1 - d\cdot h_d^2\left(\hat{\mathbf{R}}_{\sharp} p, \, p\right),
    \end{align*}
    where the last bound follows from Bernoulli's inequality. Hence, 
    \begin{align*}
        h_d^2\left( \prod_{i} p_{\hat{w}_i} \circ \pi_{\hat{w}_i} \, , \, \prod_{i} p_{w_i} \circ \pi_{\hat{w}_i} \right) \leq d \cdot h_d^2\left(\hat{\mathbf{R}}_{\sharp} p \, , \, p \right),
    \end{align*}
which completes the last bit of the proof of \Cref{res: basic error bound}.
\end{proof}

\subsection{Proof of \Cref{res: hellinger stability general log-concave}}
\label{prf: hellinger stability general log-concave}
We recall the statement below.
\paragraph{Statement}
Let $p$ be a zero-mean log-concave density on $\mathbb{R}^d$, and let $\mathbf{A} \in\mathbb{R}^{d \times d}$ be any invertible matrix. Then,
\begin{equation*}
h_{d}^2(p, \mathbf{A}_{\sharp}p) \leq K_{d} \,\mathrm{cond}(\mathbf{\Sigma})^{1/4}\,\op{\mathbf{I}-\mathbf{A}}^{1/2},
\end{equation*}
where $\mathbf{\Sigma}$ is the covariance matrix of $p$, and the constant $K_{d}>0$ depends only on the dimension $d$.

\begin{proof}
We will use properties of the log-concave projection $\psi^*$. Firstly, $p$ is a log-concave density by assumption, and by the linearity and invertibility of $\mathbf{A}$, so is $\mathbf{A}_{\sharp}p$. As a result, $\psi^*(p)=p$ and $\psi^*(\mathbf{A}_{\sharp}p)=\mathbf{A}_{\sharp}p$. By Theorem 2 of \citet{barberLocalContinuityLogconcave2021}, 
\begin{equation*}
h_{d}^{2}(p, \mathbf{A}_{\sharp}p) \leq C_{d}^2 \left[ \frac{W_{1}(p, \mathbf{A}_{\sharp}p)}{\eta_{p} \vee \eta_{\mathbf{A}_{\sharp}p}} \right]^{1/2},
\end{equation*}
where $C_{d}>0$ is a constant that depends only on $d$. We will conclude by upper-bounding the Wasserstein-1 distance $W_{1}(p, \mathbf{A}_{\sharp}p)$ and lower bounding $\eta_{p}$.

By \Cref{res: stability in W1}, we have $W_{1}(p, \mathbf{A}_{\sharp}p) \leq \sqrt{ \mathrm{tr}(\mathbf{\Sigma}) } \,\op{\mathbf{I-A}} \leq \lambda_{\max}(\mathbf{\Sigma})^{1/2}d^{1/2}\,\op{\mathbf{I}-\mathbf{A}}$. On the other hand, by \Cref{res: lower bd eta}, one can lower bound $\eta_{p} \geq c \,\lambda_{\min}(\mathbf{\Sigma})^{1/2}$, where $c>0$ is an absolute constant (say $c=1 /1600$). The result then follows with $K_{d}=C_{d}^2 \,d^{1/4} c^{-1/2}$, recalling that $\mathrm{cond}(\mathbf{\Sigma}):= \lambda_{\max}(\mathbf{\Sigma}) /\lambda_{\min}(\mathbf{\Sigma})$. 
\end{proof}

\subsection{Proof of \Cref{res: generic sample complexity proposed}}
\label{prf: generic sample complexity proposed}
We recall the statement below.
\paragraph{Statement (brief)}
Let $p$ be a zero-mean log-concave density on $\mathbb{R}^d$ satisfying the orthogonal independent components property, and let $\hat{p}$ be the proposed estimator. Then, for any $\epsilon>0$ and $0 < \gamma<1$, we have 
\begin{equation*}
h_{d}^2(\hat{p},p) \leq \epsilon
\end{equation*}
with probability at least $1-\gamma$ whenever
\begin{subequations}
\begin{equation*}
N \geq \sup_{s \in \mathbb{S}^{d-1}}N^*\left( \epsilon /4d,\, \gamma /2d;\, p_{s} \right),
\end{equation*}
and
\begin{equation*}
M \geq M^*\left( (K_{d}'\,\mathrm{cond}(\mathbf{\Sigma}))^{-1/2}\epsilon^2,\, \gamma /2;\, p \right).
\end{equation*}
\end{subequations}

\begin{proof}
By \Cref{res: basic error bound}, decompose $h_{d}(\hat{p}, p)$ as
\begin{equation*}
\begin{aligned}
h_d(\hat{p}, p) &\leq \sqrt{d} \, \max_{i\in [d]} h_1(\hat{p}_{\hat{w}_i}, p_{\hat{w}_i}) 
    + \sqrt{d}\,h_d\left(\hat{\mathbf{R}}_{\sharp}p \, , \, p \right) + h_d\left(\hat{\mathbf{R}}^T_{\sharp}p \, , \, p \right),
\end{aligned}
\end{equation*}
where $\hat{\mathbf{R}} = \mathbf{W}^T \hat{\mathbf{W}}$. We show that $\sqrt{ d } \,\max_{i \in[d]} h_{1}(\hat{p}_{\hat{w}_{i}}, p_{\hat{w}_{i}}) \leq \sqrt{ \epsilon } /2$ holds with probability at least $1-\gamma /2$ when $N$ is large enough, and that $\sqrt{d}\,h_d\left(\hat{\mathbf{R}}_{\sharp}p \, , \, p \right) + h_d\left(\hat{\mathbf{R}}^T_{\sharp}p \, , \, p \right) \leq \sqrt{ \epsilon } /2$ holds with probability at least $1-\gamma /2$ when $M$ is large enough. A final union bound then gives the desired result.

We handle $\sqrt{ d } \,\max_{i \in[d]} h_{1}(\hat{p}_{\hat{w}_{i}}, p_{\hat{w}_{i}})$ first. Write $\mathbf{X} = (X^{(1)}, \dots,X^{(N)})$ and $\mathbf{Y}=(Y^{(1)},\dots,Y^{M})$, and denote the marginal error $D_{i}(\mathbf{X}, \mathbf{Y}):= h_{1}^2(\hat{p}_{\hat{w}_{i}}, p_{\hat{w}_{i}})$ for $i=1,\dots,d$ to highlight the dependence on $\mathbf{X}$ and $\mathbf{Y}$. Then, one can use the independence of $\mathbf{X}$ and $\mathbf{Y}$ to express, for any $\tilde{\epsilon}>0$, 
\begin{equation*}
\begin{aligned}
\mathbb{P}\left\{ D_{i}(\mathbf{X}, \mathbf{Y}) > \tilde{\epsilon} \right\} = \mathbb{E}\,\mathbf{1}{\{ D_{i}(\mathbf{X}, \mathbf{Y}) > \tilde{\epsilon} \}} = \mathbb{E}_{\mathbf{Y}} [\mathbb{E}_{\mathbf{X}}\,\mathbf{1}\{ D_{i}(\mathbf{X}, \mathbf{Y})>\tilde{\epsilon} \}],
\end{aligned}
\end{equation*}
where $\mathbf{1}\{\cdot\}$ denotes the indicator of an event, and $\mathbb{E}_{\mathbf{X}}$, $\mathbb{E}_{\mathbf{Y}}$ denote expectation with respect to the distributions of $\mathbf{X}$ and $\mathbf{Y}$ respectively. As a result, it suffices to bound $\mathbb{E}_{\mathbf{X}}\,\mathbf{1}\{ D_{i}(\mathbf{X}, \mathbf{y})>\tilde{\epsilon} \}$ for any deterministic $\mathbf{y} \in (\mathbb{R}^{d})^{N}$. 

Recall that the estimated independent directions $\hat{w}_{1}, \dots, \hat{w}_{d}$ are computed from the samples in $\mathbf{Y}$. For a deterministic $\mathbf{y}$, the corresponding $\hat{w}_{1}(\mathbf{y}), \dots, \hat{w}_{d}(\mathbf{y})$ are deterministic as well. In this case, the projected samples $\{\hat{w}_{i}(\mathbf{y}) \cdot X^{(j)}\}_{j=1}^{N}$ are independent and identically distributed with law $p_{\hat{w}_{i}(\mathbf{y})}$, and as a result,
\begin{equation*}
\begin{aligned}
\mathbb{E}_{\mathbf{X}}\,\mathbf{1}\{ D_{i}(\mathbf{X}, \mathbf{y})>\tilde{\epsilon} \} = \mathbb{P}_{\mathbf{X}}\{ h_{i}^2(\hat{p}_{\hat{w}_{i}(\mathbf{y})},p_{\hat{w}_{i}(\mathbf{y})}) > \tilde{\epsilon} \} \leq \tilde{\gamma}
\end{aligned}
\end{equation*}
holds for $\tilde{\gamma}\in(0,1)$ provided $N \geq \max_{s \in \mathbb{S}^{d-1}} N^*(\tilde{\epsilon}, \tilde{\gamma}; p_{s}) \geq N^*(\tilde{\epsilon}, \tilde{\gamma}; p_{\hat{w}_{i}}(\mathbf{y}))$. Now, re-introducing the randomness in $\mathbf{Y}$ and taking expectation gives $\mathbb{P}\{ D_{i}(\mathbf{X}, \mathbf{Y})>\tilde{\epsilon} \} \leq \tilde{\gamma}$. We set $\tilde{\epsilon}=\epsilon /4d$ and $\tilde{\gamma}=\gamma /2d$, and use a union bound over $i=1,\dots,d$ to finally get
\begin{equation*}
\mathbb{P}\left\{ \sqrt{ d }\,\max_{i \in [d]} h_{1}(\hat{p},p) > \sqrt{ \epsilon } /2 \right\} \leq \gamma /2.
\end{equation*}

Now, we handle the remaining terms. By \Cref{res: hellinger stability general log-concave} and \eqref{eq: op norm w closeness}, we have
\begin{equation*}
\begin{aligned}
h_d^2\left(\hat{\mathbf{R}}_{\sharp}p \, , \, p \right) &\leq K_{d} \,\mathrm{cond}(\mathbf{\Sigma})^{1/4} \op{\mathbf{I}-\hat{\mathbf{R}}}^{1/2} \\
&\leq d^{1/4}\,K_{d}\,\mathrm{cond}(\mathbf{\Sigma})^{1/4}\left( d^{-1} \textstyle{\sum_{i=1}^d \|\hat{w}_{i}-w_{i}\|_{2}^2} \right)^{1/4},
\end{aligned}
\end{equation*}
and similarly for $h_{d}^2\left( \hat{\mathbf{R}}^T_{\sharp}p, p \right)$. For any $\tilde{\epsilon}>0$, note that $M \geq M^*(\tilde{\epsilon}, \gamma /2; p)$ samples are sufficient to give 
\begin{equation*}
\begin{aligned}
\sqrt{d}\,h_d\left(\hat{\mathbf{R}}_{\sharp}p \, , \, p \right) + h_d\left(\hat{\mathbf{R}}^T_{\sharp}p \, , \, p \right) &\leq (\sqrt{ d }+1)d^{1/8}K_{d}^{1/2}\,\mathrm{cond}(\mathbf{\Sigma})^{1/8} \tilde{\epsilon}^{1/4}
\end{aligned}
\end{equation*}
with probability at least $1 - \gamma /2$. Setting $\tilde{\epsilon} = (K_{d}'\,\mathrm{cond}(\mathbf{\Sigma}))^{-1/2}\epsilon^2$ for a suitable constant $K_{d}'>0$ (depending only on $d$) then gives
\begin{equation*}
\mathbb{P}\left\{ \sqrt{d}\,h_d\left(\hat{\mathbf{R}}_{\sharp}p \, , \, p \right) + h_d\left(\hat{\mathbf{R}}^T_{\sharp}p \, , \, p \right) > \sqrt{ \epsilon } /2 \right\} \leq \gamma /2,
\end{equation*}
which concludes the proof.
\end{proof}

\subsection{Proof of \Cref{res: sample complexity 1d log-concave MLE}}
\label{prf: sample complexity 1d log-concave MLE}
We recall the statement below.
\paragraph{Statement}
Let $f: \mathbb{R} \to [0, \infty)$ be a univariate log-concave density, and denote by $\hat{f}$ the log-concave MLE computed from $N$ iid samples from $f$. Given $\epsilon >0$ and $0 < \gamma < 1$, we have 
$$
h_{1}^{2}(\hat{f}, f) \leq \epsilon
$$
with probability at least $1-\gamma$ whenever
$$
N \geq \frac{C_{1}}{\epsilon^{5/4}} \log^{5/4}\left( \frac{C_{2}}{\gamma} \right) \vee \frac{C_{3}}{\gamma},
$$
for absolute constants $C_1, C_2, C_3 >0$. 

\begin{proof}
Lemma 6 and the proof of Theorem 5 of \citet{kim2016global} imply the concentration bound
\begin{equation*}
\mathbb{P}\left( h_{1}^{2}(\hat{f}, f) \geq n^{-4/5}t \right) \leq C'e^{ -ct } + C''n^{-1}
\end{equation*}
for suitably large absolute constants $C', C'' >0$. Setting $t=n^{4/5}\epsilon$ in the concentration bound gives that
\begin{equation*}
\mathbb{P}\left( h_{1}^{2}(\hat{f}, f) \geq \epsilon \right) \leq C'e^{ -cn^{4/5}\epsilon } + C''n^{-1}.
\end{equation*}
Now, $C'e^{ -cn^{4/5}\epsilon }\leq\gamma /2$ holds provided $n \geq \left[ (c\epsilon)^{-1}\log\left( 2C' /\gamma \right) \right]^{5/4}$, and $C''n^{-1} \leq \gamma /2$ holds provided $n \geq 2C'' /\gamma$. Setting $C_{1}=c^{-5/4}$, $C_{2}=2C'$ and $C_{3}=2C''$ concludes the proof.
\end{proof}

\subsection{Proof of \Cref{res: boosted sample complexity}}
\label{prf: boosted sample complexity}
We recall the statement below.
\paragraph{Statement}
Let $f: \mathbb{R} \to [0, \infty)$ be a univariate log-concave density, and fix $\epsilon>0$. Denote by $\hat{\hat{f}}_\epsilon$ the boosted log-concave MLE at error threshold $\epsilon$, computed from $q$  independent batches of $n$ iid samples from $f$ each. 
Given $0 < \gamma < 1$, we have 
\begin{equation*}
h_{1}^{2}(\hat{\hat{f}}_\epsilon,f) \leq \epsilon
\end{equation*}
with probability at $1-\gamma$ provided 
\begin{equation*}
n \geq \frac{C_{1}}{\epsilon^{5/4}} \quad \text{and} \quad q \geq C_{2} \log \left( \frac{1}{\gamma} \right),
\end{equation*}
for absolute constants $C_1, C_2 \geq 0$.

\begin{proof}
We follow the proof of Theorem 2.8 of \citet{geom_logconcave_boosting}. Firstly, $n \geq C_{1}\epsilon^{-5/4}$ for a sufficiently large $C_{1}>0$ ensures (by Markov's inequality) that
\begin{equation*}
\mathbb{P}\left\{ h_{1}(\hat{f}^{[b]}, f) \leq \sqrt{ \epsilon } /3 \right\}  \geq 0.9
\end{equation*}
for each batch $b$. Then, by Chernoff's inequality, strictly more than ($q /2$)-many batches $b$ satisfy $h_{1}(\hat{f}^{[b]},f) \leq \sqrt{ \epsilon } /2$ simultaneously, with probability at least $1 - 0.9^q$. Under this event, suppose $\hat{f}^{[b^\star]}$ is returned. Then there exists some $\hat{f}^{[b']}$ such that $h_{1}(\hat{f}^{[b^\star]}, \hat{f}^{[b']}) \leq 2 \sqrt{ \epsilon } /3$ and $h_{1}(\hat{f}^{[b']}, f) \leq \sqrt{ \epsilon } /3$, so that by triangle inequality, $h_{1}(\hat{f}^{[b^\star]},f) \leq \sqrt{ \epsilon }$. Finally, choosing $q \geq C_{2} \log(1 /\gamma)$ makes the probability of this event $1-\gamma$.
\end{proof}

\subsection{Proof of \Cref{res: estimating W by PCA}}
\label{prf: estimating W by PCA}
We recall the statement below.
\paragraph{Statement} (Sample complexity of estimating $\eigmat$ by PCA).
Let $p$ be a zero-mean, log-concave probability density on $\bR^d$. Let $Y^{(1)}, \dots, Y^{(M)} \overset{\mathrm{iid}}{\sim} p$ be samples, and suppose \nameref{para: M1} is satisfied with eigenvalue separation $\delta>0$. Let $\epsilon >0$  and $0<\gamma \leq 1/e$, and define $\tilde M(d, \gamma):= d \log^4(1/\gamma) \log^2\left(C_0 \log^2(1/\gamma) \right)$. If
\begin{align*}
    M \geq C_1\tilde{M}(d, \gamma) \vee \left\{\frac{C_2\op{\VarX}^2 d}{\delta^2 \epsilon^2}\log^4(1/\gamma)\log^2 \left(\frac{C_3\,\op{\VarX}^2}{\delta^2\epsilon^2}\log^2(1/\gamma)\right) \right\}, 
\end{align*}
then with probability at least $1-\gamma$, PCA recovers vectors $\{\hat{w}_1, \dots, \hat{w}_d\}$ such that
$$ 
\norm{\hat{w}_i - w_i}_2 \leq \epsilon,
$$
up to permutations and sign-flips.
\begin{proof}
Since $w_i$ and $\hat{w}_i$ are eigenvectors of $\VarX$ and $\hat{\VarX}$ respectively, the Davis-Kahan theorem (see \citet{yuUsefulVariantDavis2015}, or Theorem 4.5.5 of \citet{vershynin2018high}) gives the bound
\begin{align*}
    \norm{\hat{w}_i - w_i}_2 \leq \frac{2^{3/2}}{\delta} \op{\hat{\VarX} - \VarX},
\end{align*}
for some permutation and choice of sign-flips of $w_1, \dots, w_d$. Hence, the problem boils down to covariance estimation of log-concave random vectors \cite{adamczak2010quantitative}.

To bring the problem into an isotropic setting, define random vectors $V^{(j)} := \VarX^{-1/2} Y^{(j)}$ for $j = 1, \dots, M$, such that $\bE V^{(j)} {V^{(j)}}^T = \mathbf{I}$, and note that these still have log-concave densities. Let $C,c >0$ be absolute constants.
By Theorem 4.1 of \citet{adamczak2010quantitative}, for any $\tilde{\epsilon}\in (0,1)$ and $t\geq 1$, if
    \begin{align*}
        M \geq C d \,\frac{t^4}{\tilde{\epsilon}^2} \log^2 \left(\frac{2t^2}{\tilde{\epsilon}^2}\right),
    \end{align*}
then with probability at least $1 - e^{-ct \sqrt{d}}$, it holds that 
    \begin{align}
        \label{eq: V covariance bound}
        \normbig{\frac{1}{M} \sum_{j=1}^M V^{(j)} {V^{(j)}}^T - \mathbf{I}}_{\mathrm{op}} \leq \tilde{\epsilon}.
    \end{align}
For a $\gamma \in (0,1/e]$, we have $\log(1/\gamma) \geq 1$ so that setting $t=C'\log(1/\gamma)$ for a large enough absolute constant $C'>0$ results in the failure probability being bounded above as $e^{-ct\sqrt{d}} \leq \gamma$. Hence, (\ref{eq: V covariance bound}) holds
with probability at least $1-\gamma$, provided 
    \begin{align*}
        M \geq \frac{C_1 d}{\tilde{\epsilon}^2} \log^4(1/\gamma) \log^2 \left(\frac{C_0}{\tilde{\epsilon}^2} \log^2(1/\gamma)\right).
    \end{align*}
Furthermore, letting $\tilde \epsilon \uparrow 1$ yields that
\begin{align*}
    \normbig{\frac{1}{M} \sum_{j=1}^M V^{(j)} {V^{(j)}}^T - \mathbf{I}}_{\mathrm{op}} \leq 1
\end{align*}
with probability at least $1-\gamma$ provided $M \geq C_1 \tilde M(d, \gamma)$. As a result, (\ref{eq: V covariance bound}) also extends to $\tilde\epsilon \geq 1$. Finally, letting $\tilde{\epsilon} = 2^{-3/2}\epsilon \delta\,\op{\VarX}^{-1}$ and adjusting the absolute constants gives the required result, since 
    \begin{align*}
        \norm{\hat{\VarX} - \VarX}_{\mathrm{op}} &= \normbig{\frac{1}{M} \sum_{j=1}^M Y^{(j)} {Y^{(j)}}^T - \VarX}_{\mathrm{op}} \\
        &= \normbig{\VarX^{1/2} \left( \frac{1}{M} \sum_{j=1}^M V^{(j)} {V^{(j)}}^T - \mathbf{I}\right) \VarX^{1/2}}_{\mathrm{op}} \\
        &\leq \norm{\VarX}_{\mathrm{op}} \tilde{\epsilon} \, \leq \frac{\delta}{2^{3/2}} \epsilon.
    \end{align*}

\end{proof}

\subsection{Proof of \Cref{res: W estimation ICA}}
\label{prf: W estimation ICA}
We recall the statement below.
\paragraph{Statement} 
Let $p$ be a zero-mean probability density on $\mathbb{R}^{d}$ satisfying the orthogonal independent components property. Suppose, additionally, that \nameref{para: M2} is satisfied with parameters $\mu_{4}, \mu_{9} >0$ and $\kappa \geq 1$. Given $\epsilon > 0$ and $3d^{-3} \leq \gamma < 1$, and samples $Y^{(1)}, \dots, Y^{(M)} \overset{\mathrm{iid}}{\sim}p$, if
\begin{equation*}
M \geq \frac{K' d}{\epsilon^2} \log (3 /\gamma) \,\vee\, K'' d^2 \log^2(3 /\gamma) \,\vee\, (3 /\gamma)^{8/9}
\end{equation*}
for constants $K', K'' >0$ depending on $\mu_{4}, \mu_{9}, \kappa$, then with probability at least $1-\gamma$, ICA recovers vectors $\{ \hat{w}_{1}, \dots, \hat{w}_{d} \}$ such that
\begin{equation*}
\left( \frac{1}{d} \sum_{i=1}^d\|\hat{w}_{i}-w_{i}\|_{2}^2 \right)^{1/2}  \leq \epsilon
\end{equation*}
for independent directions $w_{1}, \dots, w_{d}$ of $p$.

\begin{proof}
Given vectors $u,v \in \mathbb{S}^{d}$, \citet{auddyLargeDimensionalIndependent2023} consider discrepancy $\sin \angle(u, v)$ in their bounds. We can relate this to the distance $\|v-u\|_{2}$ up to a sign-flip. More precisely, pick a $\theta \in \{ -1,+1 \}$ such that $u \cdot \theta v \in [0,1]$. In this case,
\begin{equation*}
\begin{aligned}
\|\theta v- u\|_{2}^2 = 2(1-u\cdot\theta v) \leq 2\left( 1-(u \cdot \theta v)^2 \right) = 2 \sin^2 \angle(u,v).
\end{aligned}
\end{equation*}
Next, the algorithm of \citet{auddyLargeDimensionalIndependent2023} produces vectors $\hat{a}_{1},\dots,\hat{a}_{d} \in \mathbb{R}^d$ which estimate the columns of $\mathbf{A}=\mathbf{W}^T \mathbf{\Sigma}_{Z}^{1/2}$. Recall that the columns of $\mathbf{A}$ are simply scalings of independent directions $w_{1}, \dots, w_{d}$.
Rescaling as $\tilde{a}_{i}=\hat{a}_{i} /\|\hat{a}_{i}\|_{2}$ does not change the angles, and hence,
\begin{equation*}
\|\tilde{a}_{i}-w_{i}\|_{2} \leq \sqrt{ 2 } \sin \angle(\tilde{a}_{i}, w_{i}) = \sqrt{ 2 } \sin \angle(\hat{a}_{i}, w_{i})
\end{equation*}
(flipping the sign of $w_i$ as required).
To get orthonormal estimates $\hat{w}_{1}, \dots, \hat{w}_{d}$, we form a matrix $\tilde{\mathbf{A}}$ with columns $\tilde{a}_{1}, \dots, \tilde{a}_{d}$, compute the singular value decomposition $\tilde{\mathbf{A}}=\mathbf{U}\mathbf{D}\mathbf{V}^T$ and set $\hat{\mathbf{W}}^T=\mathbf{U}\mathbf{V}^T$, which is the closest orthogonal matrix to $\tilde{\mathbf{A}}$ in Frobenius norm. As a result,
\begin{equation*}
\begin{aligned}
\fro{\hat{\mathbf{W}}^T - \mathbf{W}^T} &\leq \fro{\mathbf{U}\mathbf{V}^T - \tilde{\mathbf{A}}} + \fro{\tilde{\mathbf{A}}- \mathbf{W}^T} \\
&\leq 2\,\fro{\tilde{\mathbf{A}}-\mathbf{W}^T} \\
&= 2\sqrt{ d }\,\textstyle{\left( d^{-1} \sum_{i=1}^d \|\tilde{a}_{i} - w_{i}\|_{2}^2\right)^{1/2} } \\
&\leq 2 \sqrt{ 2d } \,\textstyle{\left( d^{-1} \sum_{i=1}^{d} \sin^2 \angle(\hat{a}_{i}, w_{i}) \right)^{1/2} }.
\end{aligned}
\end{equation*}

Finally, we required a bound on the 9th moment of $S_{i}$ in Moment assumption M2, instead of on some $(8+\epsilon)$th moment, for the sake of simplicity. Using these facts, and setting $\delta$ in Theorem 3.4 of \citet{auddyLargeDimensionalIndependent2023} to $\gamma /3$, gives the desired form of the result.
\end{proof}

\subsection{Proof of \Cref{res: rotational stability of KL}}
\label{prf: rotational stability of KL}
We recall the statement below.
\paragraph{Statement}
Let $p$ be a zero-mean probability density on $\bR^d$, satisfying \nameref{para: S1}. Let $\Rpl \in \bR^{d\times d}$ be any orthogonal matrix. Then,
    \begin{align*}
        \mathrm{KL}(p \,||\, \Rpl^T_{\sharp}p) \leq [\nabla \varphi]_\alpha \, d^{(1+\alpha)/2} \, \op{\VarX}^{(1+\alpha)/2}\, \norm{\mathbf{I} - \Rpl}^{1+\alpha}_{\mathrm{op}}.
    \end{align*}

\begin{proof}
Recall that $\Rpl^T_{\sharp}p (x) = p(\Rpl x) = p \circ \Rpl (x)$. The KL divergence
\begin{align}
    \mathrm{KL}(p \, ||\, p \circ \Rpl) &= \int_{\bR^d} p(x) \left(\log p(x) - \log p(\Rpl x) \right) \, dx, \nonumber \\
    &= \int_{\bR^d} p(x) \left(\varphi(\Rpl x) - \varphi(x) \right) \, dx. 
\label{eq: KL varphi}
\end{align}
where $\varphi(x) = - \log p(x)$. Since KL is always non-negative, an upper bound suffices to show closeness between $p$ and $p\circ \Rpl$. 

Using smoothness assumption S1 i.e. $\varphi \in C^{1,\alpha}(\bR^d)$ with $[\nabla \varphi]_{\alpha}< \infty$,
we get a Taylor-like expansion
\begin{align}
    \varphi(y) - \varphi(x) &\leq \nabla \varphi(x) \cdot (y-x) +  [\nabla \varphi]_{\alpha} \norm{y-x}_2^{1+\alpha}.
    \label{eq: taylor-like}
\end{align}
for $x,y\in \bR^d$. Setting $y = \Rpl x$, and plugging (\ref{eq: taylor-like}) into (\ref{eq: KL varphi}) yields
\begin{align}
    &\int_{\bR^d} p(x) \left(\varphi(\Rpl x) - \varphi(x) \right) \, dx \nonumber \\
    &\leq \int_{\bR^d} \nabla \varphi(x) \cdot (\Rpl x - x) \,p(x)dx + [\nabla \varphi]_{\alpha} \int_{\bR^d} \norm{\Rpl x-x}_2^{1+\alpha}\, p(x)dx.
    \label{eq: smt bound}
\end{align}
The second term can be handled as 
\begin{align*}
    \int_{\bR^d} \norm{\Rpl x-x}_2^{1+\alpha}\, p(x)dx &\leq \norm{\mathbf{I} - \Rpl}_{\mathrm{op}}^{1+\alpha} \int_{\bR^d} \norm{x}_2^{1+\alpha} \, p(x)dx \\
    &=  \norm{\mathbf{I} - \Rpl}_{\mathrm{op}}^{1+\alpha} \, \bE_{X \sim p} \norm{X}_2^{1+\alpha} \\
    &\leq   \norm{\mathbf{I} - \Rpl}_{\mathrm{op}}^{1+\alpha} \, \left( \bE_{X \sim p} \norm{X}_2^2 \right)^{\frac{1+\alpha}{2}} \\
    &= \norm{\mathbf{I} - \Rpl}_{\mathrm{op}}^{1+\alpha} \, \left( \mathrm{tr}\,\VarX \right)^{\frac{1+\alpha}{2}},
\end{align*}
since $1 + \alpha \leq 2$.

Now consider the first term in (\ref{eq: smt bound}). Note that $\nabla p(x)=- \nabla\varphi(x)p(x)$, and as a result,
\begin{equation*}
\int_{\mathbb{R}^d} \nabla\varphi(x) \cdot (\mathbf{R}x-x)\, p(x)dx = \int_{\mathbb{R}^d} - \nabla p(x) \cdot (\mathbf{R}x-x)\,dx
\end{equation*}
is finite by the integrability conditions in \nameref{para: S1}. We use integration by parts, and write
\begin{equation*}
\begin{aligned}
\int_{\mathbb{R}^d} - \nabla p(x) \cdot (\mathbf{R}x-x)\,dx = \int_{\mathbb{R}^d}\,p(x) \nabla \cdot (\mathbf{R}-\mathbf{I})x\,dx + \lim_{\rho\to \infty} \int_{\partial \mathbb{B}_{\rho}} -p(x) \,(\mathbf{R}x-x)\cdot \nu(x) \,dS(x),
\end{aligned}
\end{equation*}
where $\mathbb{B}_{\rho}$ is the Euclidean ball of radius $\rho$ centered at the origin, $\nu(x)=x /\|x\|_{2}$ is the outward normal vector field on $\partial \mathbb{B}_{\rho}$, and $S$ is the $(d-1)$-dimensional surface area measure. We claim that the boundary term is zero; to see this, define
\begin{equation*}
Q(\rho) := \int_{\partial\mathbb{B}_{\rho}} \|x\|_{2}\,p(x)\,dS(x) = \int_{\mathbb{S}^{d-1}} p(\rho\theta)\,\rho^d\,dS(\theta)
\end{equation*}
for $\rho >0$, and note that $Q(\rho)$ controls the boundary term as
\begin{equation*}
\begin{aligned}
\left| \int_{\partial \mathbb{B}_{\rho}} -p(x) \,(\mathbf{R}x-x)\cdot \nu(x) \,dS(x) \right| &\leq \int_{\partial \mathbb{B}_{\rho}}\,p(x) \|\mathbf{R}x-x\|_{2}\,dS(x)\\
&\leq \op{\mathbf{I}-\mathbf{R}} \,Q(\rho).
\end{aligned}
\end{equation*}
It follows from the integrability conditions in \nameref{para: S1} that $Q$ and its derivative $Q'$ are integrable on $(0, \infty)$, and as a consequence, $\lim_{\rho \to \infty} Q(\rho)=0$ (see \Cref{res: boundary Q bound} for details).
% We use integration by parts, and claim that the boundary term is zero. To see this, define
% \begin{equation*}
% Q(\rho) := \int_{\partial\mathbb{B}_{\rho}} \|x\|_{2}\,p(x)\,dS(x) = \int_{\mathbb{S}^{d-1}} p(\rho\theta)\,\rho^d\,dS(\theta)
% \end{equation*}
% where $\rho >0$ and $\mathbb{B}_{\rho}$ is the Euclidean ball of radius $\rho$ centered at the origin. $Q(\rho)$ controls the boundary term, and we would like to show that $\lim_{\rho \to \infty} Q(\rho) = 0$. But this is a consequence of the facts that $Q$ is differentiable, and $Q$, $Q'$ are both integrable on $(0, \infty)$, which in turn follow from the integrability assumptions on $p$.

Hence,
\begin{align*}
    \int_{\bR^d} - \nabla p(x) \cdot (\Rpl x - x) \, dx &= \int_{\bR^d} p(x) \, \nabla \cdot (\Rpl - \mathbf{I})x \, dx \\
    &= \int_{\bR^d} p(x) \mathrm{tr}(\Rpl - \mathbf{I}) \, dx = \mathrm{tr}(\Rpl) - \mathrm{tr}(\mathbf{I}).
\end{align*}
But note that for any orthogonal $\Rpl$,
\begin{align*}
    \mathrm{tr}(\Rpl) = \sum_{i=1}^d e_i \cdot \Rpl e_i \leq \sum_{i=1}^d \norm{e_i}_2 \norm{\Rpl e_i}_2 = d = \mathrm{tr}(\mathbf{I}),
\end{align*}
where $e_1, \dots, e_d$ are the standard basis vectors in $\bR^d$. We conclude that the the first term in (\ref{eq: smt bound}) is non-positive, and put these bounds together with $\mathrm{tr}\,\VarX \leq d \op{\VarX}$ to get the desired result.
\end{proof}

\subsection{Proof of \Cref{res: linear stability of TV}}
\label{prf: linear stability of TV}
We recall the statement below.
\paragraph{Statement} Let $p$ be a zero-mean probability density on $\bR^d$, satisfying \nameref{para: S2}. Let $\mathbf{A}\in \bR^{d \times d}$  be invertible with $\op{\mathbf{A}^{-1}} \leq B$. Then,
\begin{align*}
\mathrm{TV}(p, \mathbf{A}_{\sharp}p) \equiv \frac{1}{2}
\norm{\mathbf{A}_{\sharp}p-p}_{L^1}\leq \frac{1}{2}\left[(1+B)\norm{\nabla p}_{L^1}\sqrt{d} + d\right]\op{\VarX}^{1/4}\,\op{\mathbf{I}-\mathbf{A}}^{1/2}.
\end{align*}

\begin{proof}
Let $q = \mathbf{A}_{\sharp}p$ be the transformed density. Consider the standard Gaussian density $g(x) =(2\pi)^{-d/2}\, e^{-\frac{1}{2}\norm{x}_2^2}$ on $\bR^d$, and for any $\epsilon >0$ (to be fixed later), denote the scaled version by $g_\epsilon(x) =\epsilon^{-d}g(x/\epsilon)$.  We can decompose
\begin{align}
\label{eq: smoothing decomp}
\norm{q-p}_{L^1} &\leq \norm{q-q*g_{\epsilon}}_{L^1} + \norm{q*g_{\epsilon}-p*g_{\epsilon}}_{L^1} + \norm{p*g_\epsilon-p}_{L^1}.
\end{align}
To bound the second term of (\ref{eq: smoothing decomp}), we use \Cref{res: total variation on smoothed densities} and \Cref{res: translational gaussian}, and get 
\begin{align}
\label{eq: Bound by W1}
\norm{q*g_{\epsilon}-p*g_{\epsilon}}_{L^1} \leq \frac{\sqrt{d}}{\epsilon}\,W_1(q,p).
\end{align}
The right-hand side is now a problem of stability in the Wasserstein-1 distance, and is handled by \Cref{res: stability in W1}. Noting that $\mathrm{tr}(\VarX) \leq \op{\VarX}d$, \Cref{res: stability in W1} together with (\ref{eq: Bound by W1}) give the bound
\begin{align*}
    \norm{q*g_{\epsilon}-p*g_{\epsilon}}_{L^1} \leq \frac{d}{\epsilon} \op{\VarX}^{1/2} \op{\mathbf{I}- \mathbf{A}},
\end{align*}
which concludes the analysis of the second term of (\ref{eq: smoothing decomp}). Now, we handle the first and third terms of (\ref{eq: smoothing decomp}), making use of the assumed smoothness (S2) of $p$.
\Cref{res: mollification error} directly bounds the third term of (\ref{eq: smoothing decomp}). 
To address the first term, we apply \Cref{res: mollification error} to $q(x) = |\det \mathbf{A}^{-1}| \,p(\mathbf{A}^{-1}x)$, noting that $\norm{\nabla q}_{L^1} \leq \op{\mathbf{A}^{-1}} \norm{\nabla p}_{L^1}$. This finally yields the bound
\begin{align*}
    \norm{q-p}_{L^1} \leq (1+\op{A^{-1}})\norm{\nabla p}_{L^1}\sqrt{d}\,\epsilon + \frac{d}{\epsilon}\op{\VarX}^{1/2}\,{\op{\mathbf{I}-\mathbf{A}}}.
\end{align*}
The following choice completes the proof of \Cref{res: linear stability of TV}.
\begin{equation*}
\epsilon = \op{\VarX}^{1/4}\,\op{\mathbf{I} - \mathbf{A}}^{1/2}.
\end{equation*}
\end{proof}

\subsection{Proof of \Cref{res: smooth sample complexity proposed}}
\label{prf: smooth sample complexity proposed}
We recall the statement below.
\paragraph{Statement (brief)}
Let $p$ be a zero-mean probability density on $\mathbb{R}^d$, satisfying the orthogonal independent components property, and let $\hat{p}$ be the proposed estimator. Then, for any $\epsilon>0$ and $0 < \gamma<1$, we have 
\begin{equation*}
h_{d}^2(\hat{p},p) \leq \epsilon
\end{equation*}
with probability at least $1-\gamma$ whenever
\begin{subequations}
\begin{equation*}
N \geq \sup_{s \in \mathbb{S}^{d-1}}N^*\left( \epsilon /4d,\, \gamma /2d;\, p_{s} \right),
\end{equation*}
and
\begin{equation*}
M \geq \begin{cases}
    M^*\left(\frac{1}{4} d^{-\frac{2+\alpha}{1+\alpha}}\,[\nabla \varphi]_{\alpha}^{- \frac{1}{1+\alpha}}\,\op{\mathbf{\Sigma}}^{-1/2}\,\epsilon^{\frac{1}{1+\alpha}}\,,\, \gamma/2; \,p \right) & \text{if S1 holds} \\
    M^*\left(4^{-4} \left[ \|\nabla p\|_{L^1} d^{7/4} + d^{9/4} \right]^{-2} \op{\mathbf{\Sigma}}^{-1/2} \epsilon^{2}\,,\, \gamma/2; \,p\right) & \text{if S2 holds}. 
\end{cases}
\end{equation*}
\end{subequations}

\begin{proof}
The inequality for $N$ is the same as in \Cref{res: generic sample complexity proposed}, and follows the same argument (see \ref{prf: generic sample complexity proposed}). Now consider the inequalities for $M$. Suppose \nameref{para: S1} holds. Then,
\begin{equation*}
\begin{aligned}
h_{d}^2(\hat{\mathbf{R}}_{\sharp}p, p) &\leq \frac{1}{2} [\nabla \varphi]_\alpha \, d^{(1+\alpha)/2} \, \op{\mathbf{\Sigma}}^{(1+\alpha)/2}\, \norm{\mathbf{I} - \hat{\mathbf{R}}}^{1+\alpha}_{\mathrm{op}} \\
& \leq \frac{1}{2} [\nabla \varphi]_{\alpha} \, d^{1+\alpha} \op{\mathbf{\Sigma}}^{(1+\alpha)/2}\,\left( d^{-1} \textstyle{\sum_{i=1}^d \|\hat{w}_{i}-w_{i}\|_{2}^2} \right)^{\frac{1+\alpha}{2}},
\end{aligned}
\end{equation*}
and similarly for $h_{d}^2\left( \hat{\mathbf{R}}^T_{\sharp}p, p \right)$. For any $\tilde{\epsilon}>0$, note that $M \geq M^*(\tilde{\epsilon}, \gamma /2; p)$ samples are sufficient to give
\begin{equation*}
\sqrt{d}\,h_d\left(\hat{\mathbf{R}}_{\sharp}p \, , \, p \right) + h_d\left(\hat{\mathbf{R}}^T_{\sharp}p \, , \, p \right) \leq  \frac{\sqrt{ d }+1}{2} d^{(1+\alpha)/2}\, [\nabla \varphi]_{\alpha}^{1/2}  \op{\mathbf{\Sigma}}^{(1+\alpha)/4}\,\tilde{\epsilon}^{(1+\alpha)/2}
\end{equation*}
with probability at least $1 - \gamma /2$. Setting
\begin{equation*}
\tilde{\epsilon} = \frac{1}{4} d^{-\frac{2+\alpha}{1+\alpha}}\,[\nabla \varphi]_{\alpha}^{- \frac{1}{1+\alpha}}\,\op{\mathbf{\Sigma}}^{-1/2}\,\epsilon^{\frac{1}{1+\alpha}}
\end{equation*}
then gives
\begin{equation*}
\mathbb{P}\left\{ \sqrt{d}\,h_d\left(\hat{\mathbf{R}}_{\sharp}p \, , \, p \right) + h_d\left(\hat{\mathbf{R}}^T_{\sharp}p \, , \, p \right) > \sqrt{ \epsilon } /2 \right\} \leq \gamma /2
\end{equation*}
as desired. Now, if \nameref{para: S2} holds, then
\begin{equation*}
\begin{aligned}
h_{d}^2(\hat{\mathbf{R}}_{\sharp}p, p) &\leq \left[ \|\nabla p\|_{L^1} \sqrt{ d } + d \right] \,d^{1/4}\,\op{\mathbf{\Sigma}}^{1/4}\,\left( d^{-1} \textstyle{\sum_{i=1}^d \|\hat{w}_{i}-w_{i}\|_{2}^2} \right)^{1/4}.
\end{aligned}
\end{equation*}
As before, we consider $\tilde{\epsilon}>0$ and note that $M \geq M^*(\tilde{\epsilon}, \gamma /2; p)$ samples are sufficient to give
\begin{equation*}
\sqrt{d}\,h_d\left(\hat{\mathbf{R}}_{\sharp}p \, , \, p \right) + h_d\left(\hat{\mathbf{R}}^T_{\sharp}p \, , \, p \right) \leq 2\, \left[ \|\nabla p\|_{L^1} d^{7/4} + d^{9/4} \right]^{1/2} \op{\mathbf{\Sigma}}^{1/8}\,\tilde{\epsilon}^{1/4}
\end{equation*}
with probability at least $1 - \gamma /2$. Setting
\begin{equation*}
\tilde{\epsilon} = 4^{-4} \left[ \|\nabla p\|_{L^1} d^{7/4} + d^{9/4} \right]^{-2} \op{\mathbf{\Sigma}}^{-1/2} \epsilon^{2}
\end{equation*}
concludes the proof.

\end{proof}

\subsection{Proof of \Cref{res: basic err bound misspecification}}
  \label{prf: basic err bound misspecification}
We recall the statement below.
\paragraph{Statement}
Let $P$ be a zero-mean distribution in $\mathcal{P}_{d}$, satisfying the orthogonal independent components property with independent directions $w_{1}, \dots, w_{d}$. Let $\hat{p}$ be the proposed estimator computed from samples $X^{(1)}, \dots, X^{(N)}, Y^{(1)}, \dots, Y^{(M)} \overset{\mathrm{iid}}{\sim}P$. We have the error bound
\begin{align*}
h_{d}(\hat{p}, \psi_{d}^*(P)) \leq \sqrt{ d } & \left( \max_{i \in [d]} h_{1}(\hat{p}_{\hat{w}_{i}}, \psi_{1}^*(P_{\hat{w}_{i}})) + \max_{i \in[d]} h_{1}(\psi_{1}^*(P_{\hat{w}_{i}}), \psi_{1}^*(P_{w_{i}})) \right)  \\
 & + h_{d}(\psi_{d}^*(\hat{\mathbf{R}}_{\sharp} P), \psi_{d}^*(P)),
\end{align*}
where $\hat{\mathbf{R}}= \hat{\mathbf{W}}^T \mathbf{W}$.

\begin{proof}
Since $P$ satisfies the orthogonal independent components property with independent directions $w_{1}, \dots, w_{d}$, it follows from Theorem 2 of \citet{samworth2012independent} that the density $\psi_{d}^*(P)$ factorizes as
\begin{equation*}
\psi_{d}^*(P) (x) = \prod_{i} \psi_{1}^*(P_{w_{i}})(w_{i}\cdot x).
\end{equation*}
Similar to the proof of \Cref{res: basic error bound}, we then decompose
\begin{equation*}
\begin{aligned}
h_{d}&(\hat{p}, \psi_{d}^*(P))U = h_{d}\left( \prod_{i} \hat{p}_{\hat{w}_{i}} \circ \pi_{\hat{w}_{i}}, \prod_{i}\psi_{1}^*(P_{w_{i}}) \circ \pi_{w_{i}} \right) \\
&\leq  \underbrace{ h_{d}\left( \prod_{i} \hat{p}_{\hat{w}_{i}} \circ \pi_{\hat{w}_{i}}, \prod_{i}\psi_{1}^*(P_{\hat{w}_{i}}) \circ \pi_{\hat{w}_{i}} \right) }_{ (\mathrm{I}) } + \underbrace{ h_{d}\left(\prod_{i}\psi_{1}^*(P_{\hat{w}_{i}}) \circ \pi_{\hat{w}_{i}}, \prod_{i}\psi_{1}^*(P_{{w}_{i}}) \circ \pi_{\hat{w}_{i}}  \right) }_{ (\mathrm{II}) } \\
&\phantom{=} \, + \underbrace{ h_{d}\left( \prod_{i}\psi_{1}^*(P_{{w}_{i}}) \circ \pi_{\hat{w}_{i}}, \prod_{i}\psi_{1}^*(P_{{w}_{i}}) \circ \pi_{{w}_{i}} \right) }_{ (\mathrm{III}) } .
\end{aligned}
\end{equation*}
By the tensorization argument of \Cref{lemma: h oracle}, 
\begin{equation*}
(\mathrm{I}) \leq \sqrt{ d }\,\max_{i \in [d]} h_{1}(\hat{p}_{\hat{w}_{i}}, \psi_{1}^*(P_{\hat{w}_{i}})).
\end{equation*}
An identical calculation also yields
\begin{equation*}
(\mathrm{II}) \leq \sqrt{ d }\,\max_{i \in [d]} h_{1}(\psi_{1}^*(P_{\hat{w}_{i}}), \psi_{1}^*(P_{w_{i}})).
\end{equation*}
Finally, noting that $\prod_{i} \psi_{1}^*(P_{w_{i}})\circ \pi_{\hat{w}_{i}}(x) = \prod_{i} \psi_{1}^*(P_{w_{i}})(\hat{w}_{i}\cdot x) = \psi_{d}^*(P)(\hat{\mathbf{R}}^Tx)$, we can rewrite
\begin{align*}
(\mathrm{I I I}) = h_{d}(\hat{\mathbf{R}}_{\sharp}\psi_{d}^*(P), \psi_{d}^*(P)).
\end{align*}
Since the log-concave projection $\psi_{d}^*$ commutes with affine maps \cite{dumbgen_ApproximationLogConcaveDistributions_2011}, we have that $\hat{\mathbf{R}}_{\sharp}\psi_{d}^*(P)=\psi_{d}^*(\hat{\mathbf{R}}_{\sharp}P)$ which concludes the proof.
\end{proof}

\subsection{Proof of \Cref{res: sample complexity misspecification} }
\label{prf: sample complexity misspecification}
We recall the statement below.
\paragraph{Statement (brief)}
Let $P$ be a zero-mean distribution in $\mathcal{P}_{d}$ satisfying the orthogonal independent components property. Suppose that \nameref{para: M2} is satisfied, $\eta_{P}>0$, and $L_{q} < \infty$ for some $q > 1$. Let $\hat{p}$ be the proposed estimator, invoking log-concave MLE and ICA. Then, for any $\epsilon>0$ and $6d^{-3} < \gamma<1$, we have 
\begin{equation*}
h_{d}^2(\hat{p},\psi_{d}^*(P)) \leq \epsilon
\end{equation*}
with probability at least $1-\gamma$ whenever
\begin{equation*}
\frac{N}{\log^{\frac{3q}{q-1}}N} \geq \left( K_{q} \sqrt{ \frac{L_{q}}{\eta_{P}} }\, \frac{8d^2}{\epsilon\gamma} \right)^{\frac{2q}{q-1}} ,
\end{equation*}
and
\begin{equation*}
M \geq \frac{K' C_{d}'' L_{1}^2}{\eta_{P}^2\, \epsilon^4} \log (6 /\gamma) \,\vee\, K'' d^2 \log^2(6 /\gamma) \,\vee\, (6 /\gamma)^{8/9}.
\end{equation*}

\begin{proof}
The proof follows the same pattern as that of \Cref{res: generic sample complexity proposed}. We use the error bound from \Cref{res: basic err bound misspecification}:
\begin{equation*}
\begin{aligned}
h_{d}(\hat{p}, \psi_{d}^*(P)) \leq& \underbrace{ \sqrt{ d } \max_{i \in [d]} h_{1}(\hat{p}_{\hat{w}_{i}}, \psi_{1}^*(P_{\hat{w}_{i}})) }_{ (\mathrm{I}) } + \underbrace{ \sqrt{ d }\,\max_{i \in[d]} h_{1}(\psi_{1}^*(P_{\hat{w}_{i}}), \psi_{1}^*(P_{w_{i}})) }_{ (\mathrm{II}) }  \\
&+ \underbrace{ h_{d}(\psi_{d}^*(\hat{\mathbf{R}}_{\sharp} P), \psi_{d}^*(P)) }_{ (\mathrm{III}) }.
\end{aligned}
\end{equation*}

First, we show that $\sqrt{ d }\,\max_{i \in [d]} h_{1}(\hat{p}_{\hat{w}_{i}}, \psi_{1}^*(P_{\hat{w}_{i}})) \leq \sqrt{ \epsilon } /2$ holds with probability at least $1-\gamma /2$, provided $N$ is large enough. By a conditioning argument on $\mathbf{Y}=(Y^{(1)}, \dots,Y^{(M)})$ as before, we can treat $\hat{w}_{1}, \dots, \hat{w}_{d}$ as deterministic for this part of the proof.

Define the empirical distribution $\hat{P}^X:= N^{-1} \sum_{j=1}^{N}\delta_{X^{(j)}}$, and denote it's $s$-marginal by $\hat{P}^X_{s}$ for $s \in \mathbb{S}^{d-1}$. Since we used log-concave MLE to estimate the marginals in this setting, we can express $\hat{p}_{\hat{w}_{i}}=\psi_{1}^*(\hat{P}^X_{\hat{w}_{i}})$. Noting that $\hat{P}^X_{\hat{w}_{i}}$ is itself an empirical distribution of $P_{\hat{w}_{i}}$ for each $i=1,\dots,d$, Theorem 5 of \citet{barberLocalContinuityLogconcave2021} gives the bound
\begin{equation*}
\mathbb{E}\,h_{1}^2(\hat{p}_{\hat{w}_{i}}, \psi_{1}^*(P_{\hat{w}_{i}})) \leq K_{q} \sqrt{ \frac{L_{q}}{\eta_{P}} } \, \frac{\log^{3/2} N}{N^{\frac{1}{2} - \frac{1}{2q}}},
\end{equation*}
for a constant $K_{q}>0$ depending only on $q$. (In applying the theorem, we have used the simple facts that $(\mathbb{E}\,|\hat{w}_{i}\cdot X|^{q})^{1/q} \leq L_{q}$ and $\eta_{P_{\hat{w}_{i}}} \geq \eta_{P}$ for all $i$). A direct application of Markov's inequality also gives the tail bound
\begin{equation*}
\mathbb{P}\left\{ h_{1}^2(\hat{p}_{\hat{w}_{i}}, \psi_{1}^*(P_{\hat{w}_{i}})) \geq \tilde{\epsilon}  \right\} \leq \frac{1}{\tilde{\epsilon}} \cdot K_{q} \sqrt{ \frac{L_{q}}{\eta_{P}} } \, \frac{\log^{3/2} N}{N^{\frac{1}{2} - \frac{1}{2q}}}
\end{equation*}
for any $\tilde{\epsilon}>0$. The right-hand side is bounded above by some $\tilde{\gamma} \in (0,1)$ provided 
\begin{equation*}
\frac{N}{\log^{\frac{3q}{q-1}}N} \geq \left( K_{q} \sqrt{ \frac{L_{q}}{\eta_{P}} }\, \frac{1}{\tilde{\epsilon}\tilde{\gamma}} \right)^{\frac{2q}{q-1}}.
\end{equation*}
We set $\tilde{\epsilon}=\epsilon /4d$ and $\tilde{\gamma}=\gamma /2d$, and use a union bound over $i=1,\dots,d$ to finally get
\begin{equation*}
\mathbb{P}\left\{ \sqrt{ d }\,\max_{i \in [d]} h_{1}(\hat{p},p) > \sqrt{ \epsilon } /2 \right\} \leq \gamma /2.
\end{equation*}
Now consider term (II). By Theorem 2 of \citet{barberLocalContinuityLogconcave2021}, we have that
\begin{equation*}
h_{1}(\psi_{1}^*(P_{\hat{w}_{i}}), \psi_{1}^*(P_{w_{i}})) \leq C_{1} \left[ \frac{W_{1}(P_{\hat{w}_{i}}, P_{w_{i}})}{\eta_{P}} \right]^{1/4}
\end{equation*}
for an absolute constant $C_{1}>0$, where one can further simplify
\begin{equation*}
W_{1}(P_{\hat{w}_{i}}, P_{w_{i}}) \leq \mathbb{E}_{X \sim P}\,|\hat{w}_{i}\cdot X - w_{i}\cdot X| \leq \left( \mathbb{E}_{X \sim P} \|X\|_{2} \right)\, \|\hat{w}_{i}-w_{i}\|_{2} = L_{1} \|\hat{w}_{i}-w_{i}\|_{2}.
\end{equation*}
Theorem 2 of \citet{barberLocalContinuityLogconcave2021} can also be used to bound term (III) as
\begin{equation*}
h_{d}(\psi_{d}^*(\hat{\mathbf{R}}_{\sharp}P), \psi^*_{d}(P)) \leq C_{d} \left[ \frac{W_{1}(\hat{\mathbf{R}}_{\sharp}P, P)}{\eta_{P}} \right]^{1/4},
\end{equation*}
where $C_{d}>0$ depends only on $d$, and we note that
\begin{equation*}
W_{1}(\hat{\mathbf{R}}_{\sharp}P, P) \leq \mathbb{E}_{X \sim P}\,\|\hat{\mathbf{R}}X - X\|_{2} \leq \left( \mathbb{E}_{X \sim P}\|X\|_{2} \right) \, \op{\mathbf{I} - \hat{\mathbf{R}}} = L_{1} \fro{\hat{\mathbf{W}}-\mathbf{W}}.
\end{equation*}
Putting these inequalities together gives
\begin{equation*}
	(\mathrm{I I}) + (\mathrm{I I I}) \leq C_{d}' \left( \frac{L_{1}}{\eta_{P}} \right)^{1/4} \left( d^{-1} \textstyle{\sum_{i=1}^d \|\hat{w}_{i}-w_{i}\|_{2}^2} \right)^{1/8}
\end{equation*}
for some $C_{d}'>0$ depending only on $d$. Applying \Cref{res: W estimation ICA} allows us to conclude that $(\mathrm{I I})+ (\mathrm{I I I}) \leq \sqrt{ \epsilon } /2$ with probability at least $1-\gamma /2$ provided
\begin{equation*}
M \geq \frac{C_{d}'' K' L_{1}^2}{\eta_{P}^2 \,\epsilon^4} \log (6 /\gamma) \,\vee\, K'' d^2 \log^2(6 /\gamma) \,\vee\, (6 /\gamma)^{8/9}
\end{equation*}
for a suitably chosen $C_{d}''>0$ depending only on $d$.
\end{proof}

\newpage
\section{Auxiliary lemmas}
\label{sec: aux lemmas}
Here, we prove the various lemmas invoked in \Cref{sec: app A}.

\begin{lemma}
    \label{lem: hellinger 1d bounded by d}
Let $p(x)= \prod_{i=1}^d p_{w_i}(w_i \cdot x)$, where each $p_{w_i}$ is a univariate probability density, and consider any other probability density $q:\bR^d \to [0, \infty)$. Then,
    \begin{align*}
        h_1^2 (p_{w_k}, q_{w_k}) \leq h_d^2(p,q),
    \end{align*}
    for $k\in [d]$.
\end{lemma}
The above lemma bounds the 1-dimensional Hellinger distance between marginals by the $d$-dimensional Hellinger distance between the full densities. This is a special case of the \emph{data processing inequality} for $f$-divergences, but we nevertheless provide a proof below.

\begin{proof}
Writing
    \begin{align*}
        1 - h_d^2(p,q) = \int_{\bR^d}\sqrt{p(x)q(x)} dx = \int_{\bR^d}\sqrt{\prod_{i}p_{w_i}(w_i \cdot x)\, q(x)} dx,
    \end{align*}
    and changing variables $z = \eigmat x$ gives
    \begin{align*}
        \int_{\bR^d}\sqrt{\prod_{i}p_{w_i}(z_i)\, \bar{q}(z)} dz,
    \end{align*}
    where $\bar{q}(z) := q(\eigmat^Tz)$. Splitting the components of $z$ as $z_k$ and $z_{{[d]\backslash k}}$ allows us to rewrite the above expression as an iterated integral, and use Cauchy-Schwarz on the inner integral:
    \begin{align}
        & \int_{\bR} \int_{\bR^{d-1}} \sqrt{\prod_{i\neq k}p_{w_i}(z_i)\, \bar{q}(z_k, z_{{[d]\backslash k}})} \, dz_{{[d]\backslash k}} \,\, \sqrt{p_{w_k}(z_k)}\,dz_k \nonumber \\
        & \leq \int_{\bR} \left\{ \int_{\bR^{d-1}}\prod_{i\neq k}p_{w_i}(z_i) \, dz_{{[d]\backslash k}} \right\}^{1/2} \,
        \left\{ \int_{\bR^{d-1}} \bar{q}(z_k, z_{{[d]\backslash k}})\, dz_{{[d]\backslash k}} \right\}^{1/2} \, \sqrt{p_{w_k}(z_k)}\,dz_k. \nonumber
    \end{align}
    Since each $p_i$ is a probability density,
    \begin{align*}
       \int_{\bR^{d-1}}\prod_{i\neq k}p_{w_i}(z_i) \, dz_{{[d]\backslash k}} =  \prod_{i\neq k} \int_{\bR} p_{w_i}(z_i)\, dz_i = 1.
    \end{align*}
    On the other hand, $\bar{q}$ is being marginalized and 
    \begin{align*}
        \bar{q}_{e_k}(z_k) := \int_{\bR^{d-1}} \bar{q}(z_k, z_{{[d]\backslash k}})\, dz_{{[d]\backslash k}}
    \end{align*}
    is the marginal of $\bar{q}$ along the direction $e_k$. Noting that $\bar{q}_{e_k} = q_{w_k}$ and chaining together the above inequalities gives the desired result:
    \begin{align*}
        1 - h_d^2(p,q) \leq \int_{\bR} \sqrt{q_{w_k}(z_k) p_{w_k}(z_k)}\, dz_k = 1 - h_1^2(p_{w_k}, q_{w_k}).
    \end{align*}
\end{proof}

\begin{lemma}[Jensen gap for univariate log-concave densities]
    \label{res: jensen gap univariate log-concave}
Let $f:\mathbb{R} \to[0, \infty)$ be any univariate log-concave probability density, and suppose $X \sim f$. Then,
\begin{equation*}
c\,\sqrt{ \mathrm{Var}X }  \leq \mathbb{E}|X-\mathbb{E}X| \leq \sqrt{ \mathrm{Var}X }
\end{equation*}
holds for $c = 1 /1600$.
\end{lemma}
\begin{proof}
Define the `standardized' random variable $Z := \sigma^{-1}(X-a)$ where $a$ and $\sigma$ are the mean and standard deviation of $X$ respectively. Denote the density of $Z$ by $g$, and note that $g$ is an isotropic log-concave density. By Theorem 5.14(a) of \citet{geom_logconcave_boosting} specialized to the case $d=1$, we have the lower bound $g(z) \geq 2^{-9|z|}g(0)$ for $|z| \leq 1 /9$, which can be used to estimate
\begin{equation*}
\mathbb{E}|Z| = \int_{\mathbb{R}}|z|g(z)\,dz \geq \int_{-1 /9}^{1/9} |z|g(z) \,dz \geq g(0) \int_{- 1 /9}^{1/9}|z| 2^{-9|z|}\,dz \geq \frac{g(0)}{200} .
\end{equation*}
Additionally, by Lemma 5.5(b) of \citet{geom_logconcave_boosting}, $g(0) \geq \frac{1}{8}$. We conclude that $\mathbb{E}|Z| \geq c$ for $c = 1 /1600$, which can be transformed to give
\begin{equation*}
\mathbb{E}|X - \mathbb{E}X| = \sigma\, \mathbb{E}|Z| \geq c\sigma = c \sqrt{ \mathrm{Var}X }.
\end{equation*}
The upper bound on $\mathbb{E}|X-\mathbb{E}X|$ is immediate from Jensen's inequality or the Cauchy-Schwarz inequality.
\end{proof}

\begin{lemma}[lower bound on $\eta$ for a log-concave density]
\label{res: lower bd eta}
Let $p$ be any log-concave density on $\mathbb{R}^d$, with mean $\mu$ and covariance matrix $\mathbf{\Sigma}$. Then,
\begin{equation*}
\eta_{p} \equiv \inf_{s \in \mathbb{S}^{d-1}} \mathbb{E}_{X \sim p} \,|s \cdot (X- \mu)| \geq c\, \lambda_{\min}(\mathbf{\Sigma})^{1/2}
\end{equation*}
holds for $c=1 /1600$.     
\end{lemma}

\begin{proof}
Let $s_{*} \in \arg \min_{s \in \mathbb{S}^{d-1}} \mathbb{E}_{X \sim p}\,|s \cdot (X-\mu)|$ (which exists by compactness and continuity). 
For $X \sim p$, note that $s_{*}\cdot X$ has a univariate log-concave density $p_{s_{*}}$, so that one can apply \Cref{res: jensen gap univariate log-concave} to conclude that
\begin{align*}
\eta_{p}=\mathbb{E}\,|s_{*} \cdot X - s_{*}\cdot\mu| \geq c \, \mathrm{Var}(s_{*}\cdot X)^{1/2} \geq c\,\lambda_{\min}(\mathbf{\Sigma})^{1/2}.
\end{align*}
\end{proof}

\begin{lemma}[Stability in $W_1$]
\label{res: stability in W1}
Let $p$ be a probability density on $\bR^d$ with mean zero and covariance matrix $\VarX$, and let $\mathbf{A}\in \bR^{d\times d}$.  Then, 
$$
W_1(\mathbf{A}_{\sharp}p,\, p) \leq \sqrt{\mathrm{tr}(\VarX)}\,\op{\mathbf{A}-\mathbf{I}}. 
$$
\end{lemma}
\begin{proof}
A direct calculation yields
\begin{align*}
W_1(\mathbf{A}_{\sharp}p,\, p) &\leq \inf_{\substack{T:\bR^d\to\bR^d\\ \textrm{measurable}}} \left\{\int_{\bR^d}\norm{T(x)-x}_2\, p(x)dx\,:\,T_{\sharp}p = \mathbf{A}_{\sharp}p\right\} \\
&\leq \int_{\bR^d} \norm{\mathbf{A}x - x}_2\,p(x)dx \\
&\leq \op{\mathbf{I} - \mathbf{A}}\int_{\bR^d}\norm{x}_2\,p(x)dx \\
&\leq \op{\mathbf{I} - \mathbf{A}} \,\bE_{X\sim p} \norm{X}_2\,.
\end{align*}
The result then follows because $\bE \norm{X}_2 \leq \sqrt{\bE \norm{X}_2^2} = \sqrt{\mathrm{tr}(\VarX)}$.
\end{proof}

\begin{lemma}[Controlling the boundary integral]
\label{res: boundary Q bound}
Let $p$ be a continuously differentiable probability density on $\mathbb{R}^d$, such that $\int_{\mathbb{R}^d} \|x\|_{2}p(x)\,dx < \infty$ and $\int_{\mathbb{R}^d} \|x\|_{2} \|\nabla p(x)\|_{2}\,dx < \infty$. Then, the function
\begin{equation*}
Q(\rho) := \int_{\partial\mathbb{B}_{\rho}} \|x\|_{2}\,p(x)\,dS(x) = \int_{\mathbb{S}^{d-1}} p(\rho\theta)\,\rho^d\,dS(\theta)
\end{equation*}
is differentiable on $(0,\infty)$, and moreover, $Q$ and $Q'$ are integrable on $(0,\infty)$. Consequently, we also have that $\lim_{\rho \to \infty}Q(\rho)=0$.
\end{lemma}
\begin{proof}
Using spherical integration, we immediately get that
\begin{equation*}
\int_{0}^\infty \left| Q(\rho) \right| \,d\rho = \int_{0}^\infty \int_{\mathbb{S}^{d-1}}\,p(\rho\theta) \rho^d\,dS(\theta)\,d\rho = \int_{\mathbb{R}^d}\,p(x) \|x\|_{2}\,dx < \infty.
\end{equation*}
To get $Q'$, we differentiate under the integral sign as
\begin{equation*}
Q'(\rho) = \int_{\mathbb{S}^{d-1}}\,\left\{\rho^d\, (\nabla p(\rho\theta)\cdot\theta) + (\rho^{d-1}d)\,p(\rho\theta)\right\} \,dS(\theta),
\end{equation*}
which is justified because the differentiated integrand is continuous in $(\rho,\theta)$, and as a result, can be uniformly bounded in magnitude by a constant over all $\theta \in \mathbb{S}^{d-1}$ and all $\rho$ in any compact subset of $(0,\infty)$. Now, integrating $\left| Q' \right|$ gives
\begin{equation*}
\begin{aligned}
\int_{0}^\infty \left| Q'(\rho) \right| \,d\rho &\leq \int_{0}^\infty \int_{\mathbb{S}^{d-1}}\, \left| \rho^d\, (\nabla p(\rho\theta)\cdot\theta) + (\rho^{d-1}d)\,p(\rho\theta) \right| \,dS(\theta)\,d\rho \\
&= \int_{\mathbb{R}^d} \left| \nabla p(x)\cdot x + p(x)d \right| \,dx \\
&\leq \int_{\mathbb{R}^d} \|x\|_{2}\|\nabla p(x)\|_{2}\,dx + d < \infty.
\end{aligned}
\end{equation*}

What remains, finally, is to show that $\lim_{\rho \to \infty}Q(\rho)=0$. The integrability of $Q'$ implies that the limit
\begin{equation*}
\lim_{\rho \to \infty}\left( Q(\rho)-Q(1) \right) =\int_{1}^\infty Q'(r)\,dr
\end{equation*}
exists finitely. But if $L:= \lim_{\rho\to \infty}Q(\rho)$ were a non-zero number, then $Q$ would not be integrable over $(0,\infty)$, leading to a contradiction. Hence, we may conclude that $L=0$ as required.
\end{proof}

For the next few lemmas, recall that $g(x) =(2\pi)^{-d/2}\, e^{-\frac{1}{2}\norm{x}_2^2}$ and $g_\epsilon(x) =\epsilon^{-d}g(x/\epsilon)$.

\begin{lemma}[$L^1$ bound on smoothed densities]
\label{res: total variation on smoothed densities}
Let $p$ and $q$ be probability densities on $\bR^d$, and denote by $W_1(p,q)$ the Wasserstein-1 distance between them. For any $\epsilon >0$, it holds that
$$
\norm{p*g_\epsilon - q*g_{\epsilon}}_{L^1} \leq \sup_{\substack{s,t\in \bR^d \\ s\neq t}}\frac{\norm{\tau_s g_\epsilon - \tau_tg_\epsilon}_{L^1}}{\norm{s-t}_2}\,W_1(p,\,q),
$$
where $\tau_s h(x) := h(x-s)$ denotes a translation for any function $h$ on $\bR^d$.
\end{lemma}
\begin{proof}
Denote by $\Gamma(p, q)$ the set of all couplings (or transport plans) between $p$ and $q$, and let $\gamma \in \Gamma(p,q)$. Note that one can express 
\begin{align*}
p*g_\epsilon(x) - q*g_\epsilon(x) &= \int_{\bR^d}g_\epsilon(x-s)p(s)ds - \int_{\bR^d}g_\epsilon(x-t)q(t)dt \\
&= \int_{\bR^d \times \bR^d} \left(g_\epsilon(x-s) - g_\epsilon(x-t)\right)\,d\gamma(s,t).
\end{align*}
A direct calculation involving an exchange of integrals and a Holder bound gives
\begin{align*}
\norm{p*g_\epsilon - q*g_{\epsilon}}_{L^1} &= \int_{\bR^d}\left\vert p*g_\epsilon(x) - q*g_\epsilon(x)\right\vert \,dx \\
&\leq \int_{\bR^d} \int_{\bR^d \times \bR^d}\left\vert g_\epsilon(x-s) - g_\epsilon(x-t)\right\vert \,d\gamma(s,t)\,dx \\
&= \int_{\bR^d \times \bR^d} \int_{\bR^d} \left\vert \tau_sg_\epsilon(x) - \tau_tg_\epsilon(x) \right\vert\,dx\,d\gamma(s,t) \\
&= \int_{\bR^d \times \bR^d} \norm{\tau_sg_\epsilon - \tau_tg_\epsilon}_{L^1}\,d\gamma(s,t) \\
&= \int_{\bR^d \times \bR^d} \frac{\norm{\tau_sg_\epsilon - \tau_tg_\epsilon}_{L^1}}{\norm{s-t}_2}\,\norm{s-t}_2\,d\gamma(s,t) \\
&\leq \sup_{\substack{s,t\in \bR^d \\ s\neq t}}\frac{\norm{\tau_s g_\epsilon - \tau_tg_\epsilon}_{L^1}}{\norm{s-t}_2} \int_{\bR^d \times \bR^d} \norm{s-t}_2 \, d\gamma(s,t).
\end{align*}
Taking an infimum over all $\gamma \in \Gamma(p,q)$ gives the desired result.

\end{proof}
\begin{lemma}[Translational stability of Gaussian kernels in $L^1$]
\label{res: translational gaussian}
Let $s, t \in \bR^d$, and $\epsilon>0$. Then,
$$
\norm{\tau_s g_\epsilon - \tau_tg_\epsilon}_{L^1} \leq  \frac{\sqrt{d}}{\epsilon} \norm{s-t}_2.
$$
\end{lemma}
\begin{proof}
Fix some $x\in \bR^d$ and consider the smooth function $[0,1]\ni \xi \mapsto g_\epsilon(x-s + \xi(s-t))$ where $\xi$ linearly interpolates between $x-s$ and $x-t$. This lets us write
$$
g_\epsilon(x-t) - g_\epsilon(x-s) = \int_0^1 \nabla g_\epsilon(x-s + \xi(s-t))\cdot (s-t) \, d\xi.
$$
Integrating, we get that
\begin{align*}
\int_{\bR^d} \left\vert g_\epsilon(x-t) - g_\epsilon(x-s) \right\vert \, dx &\leq \int_{\bR^d} \int_0^1 \norm{\nabla g_\epsilon(x-s + \xi(s-t))}_2\,\norm{s-t}_2 \, d\xi \, dx \\
&\leq \int_0^1 \norm{s-t}_2 \int_{\bR^d} \norm{\nabla g_\epsilon(x-s + \xi(s-t))}_2 \,dx \, d\xi \\
&=  \int_{\bR^d} \norm{\nabla g_\epsilon(x)}_2 \,dx \, \norm{s-t}_2.
\end{align*}
What remains is to bound the integral $\int_{\bR^d} \norm{\nabla g_\epsilon(x)}_2 \,dx$, and this is straightforward. Noting that 
$$
\frac{\partial g_\epsilon}{\partial x_j}(x) = -\frac{x_j}{\epsilon^2}g_\epsilon(x),
$$
we get
\begin{align*}
\int_{\bR^d}\norm{\nabla g_\epsilon(x)}_2\,dx &= \frac{1}{\epsilon^2}\int_{\bR^d}\norm{x}_2\,g_\epsilon(x)\,dx
= \frac{1}{\epsilon^2}\,\bE_{X\sim N(0,\epsilon^2)}\norm{X}_2
\leq \frac{\sqrt{d}}{\epsilon},
\end{align*}
which completes the proof of the lemma.

\end{proof}

\begin{lemma}[$L^1$ error due to mollification]
\label{res: mollification error}
Let $p$ be a probability density on $\bR^d$ satisfying \nameref{para: S2}. Then, for any $\epsilon>0$,
$$
\norm{p*g_\epsilon - p}_{L^1} \leq \norm{\nabla p}_{L^1}\, \sqrt{d}\,\epsilon.
$$
\end{lemma}
\begin{proof}
We first relate the approximation error above to the translational stability in $L^1$:

\begin{align*}
\norm{p*g_\epsilon - p}_{L^1} &= \int_{\bR^d} \left\lvert \int_{\bR^d} \left(p(x-y) - p(x)\right) g_\epsilon(y)\,dy\right\rvert \,dx \\
&\leq \int_{\bR^d} \int_{\bR^d}|\tau_yp(x) - p(x)|\,dx\,g_\epsilon(y)\,dy \\
&= \int_{\bR^d} \norm{\tau_yp - p}_{L^1}\,g_\epsilon(y)\,dy.
\end{align*}
This translational stability $\norm{\tau_yp - p}_{L^1}$ is controlled via smoothness. Suppose first that $p \in \mathcal{C}^1(\bR^d)$, so that for $x,y \in \bR^d$,
$$
p(x-y) - p(x) = \int_0^1 \nabla p(x - \xi y)\cdot (-y)\,d\xi.
$$
Integrating gives
\begin{align*}
\norm{\tau_yp - p}_{L^1} &\leq \int_{\bR^d}\int_0^1 \norm{\nabla p(x - \xi y)}_2\,\norm{y}_2\,d\xi \, dx \\
&= \int_0^1 \norm{y}_2 \int_{\bR^d} \norm{\nabla p(x - \xi y)}_2\, dx\, d\xi \\
&= \int_{\bR^d} \norm{\nabla p(x)}_2\,dx\,\norm{y}_2 = \norm{\nabla p}_{L^1}\,\norm{y}_2.
\end{align*}
If $p\notin \mathcal{C}^1(\bR^d)$, one can use the fact that $p$ lives in the Sobolev space $\mathcal{W}^{1,1}(\bR^d)$ to get an approximating sequence $p_k\in \mathcal{C}^1(\bR^d) \cap \mathcal{W}^{1,1}(\bR^d)$ satisfying $\norm{p_k - p}_{L^1}\to 0$ and $\norm{\nabla p_k - \nabla p}_{L^1}\to 0$ \cite{HWdensity}.  This allows us to extend the bound
$$
\norm{\tau_y p - p}_{L^1} \leq \norm{\nabla p}_{L^1} \norm{y}_2
$$
to all $p\in \mathcal{W}^{1,1}(\bR^d)$. Combining, we get
\begin{align*}
\norm{p*g_\epsilon - p}_{L^1} &\leq \int_{\bR^d} \norm{\tau_yp - p}_{L^1}\,g_\epsilon(y)\,dy
\leq \norm{\nabla p}_{L^1} \int_{\bR^d}\norm{y}_2 \, g_\epsilon(y)\,dy
\leq \norm{\nabla p}_{L^1}\, \sqrt{d}\,\epsilon.
\end{align*}

\end{proof}
\newpage

\section{Computational details and further numerical results}
\label{sec: app B}

\subsection{Monte Carlo integration for computing squared Hellinger distances}
\label{subsec: app monte carlo}
Recall that the squared Hellinger between densities $p$ and $q$ on $\bR^d$ is defined by the integral
\begin{align}
    h_d^2(p,q) := \frac{1}{2} \int_{\bR^d} \left(\sqrt{q(x)} - \sqrt{p(x)} \right)^2 dx.
    \label{eq: hellinger def}
\end{align}
If $d$ is larger than 4, computing the integral by discretizing on a grid can prove prohibitively expensive. Instead, one can use simple Monte Carlo integration to approximate \eqref{eq: hellinger def}. Notice that
\begin{align*}
    h_d^2(p,q) = \frac{1}{2} \int_{\bR^d} \left( \sqrt{q(x)/p(x)} - 1\right)^2 p(x) dx = \bE_{X\sim p} \, \frac{1}{2} \left( \sqrt{q(X)/p(X)} - 1\right)^2,
\end{align*}
where the expectation on the right-hand side can be approximated using an empirical average. We simulate samples $S^{(1)}, S^{(2)}, \dots, S^{(K)}\overset{\mathrm{iid}}{\sim} p$, and compute
\begin{align*}
    \frac{1}{K} \sum_{k=1}^K \frac{1}{2} \left( \sqrt{q(S^{(k)})/p(S^{(k)})} - 1\right)^2
\end{align*}
which converges almost surely to $h_d^2(p,q)$ as $K\to \infty$. In practice, we use $K=10000$. Additionally, we repeat the procedure 50 times and ensure that the resulting spread is not too large.

\subsection{The EM algorithm and the re-sampling heuristic}
\label{subsec: EM resampling}
Recall the finite mixture densities discussed in \Cref{subsec: mixtures}
\begin{align*}
    p(x) = \sum_{k=1}^K \pi_k p_k(x),
\end{align*}
and consider samples $X^{(1)}, X^{(2)}, \dots, X^{(n)}\overset{\mathrm{iid}}{\sim} p$.
We briefly describe the EM algorithm following \citet{cule2010maximum}. Given current iterates of the mixture proportions $\tilde\pi_1, \dots, \tilde\pi_K$ and component densities $\tilde p_1, \dots, \tilde p_K$, compute the posterior probabilities
\begin{align}
    \label{eq: posterior prob EM}
    \tilde\theta_{jk} \leftarrow \frac{\tilde\pi_k \tilde p_k(X^{(j)})}{\sum_{\ell=1}^K \tilde\pi_\ell \tilde p_\ell(X^{(j)})}.
\end{align}
Then, for each component $k\in \{1, \dots, K\}$, update the density $\tilde p_k$ by maximizing a weighted likelihood over a suitable class of densities $\mathcal{F}$ (such as log-concave densities):
\begin{align}
    \label{eq: weighted ML max EM}
    \tilde p_k \leftarrow \argmax_{q_k\in \mathcal{F}} \sum_{j=1}^n \tilde\theta_{jk} \log q_k(X^{(j)}).
\end{align}
Finally, update the mixture proportions as
\begin{align}
    \label{eq: mix prop EM}
    \tilde\pi_k \leftarrow \frac{1}{n} \sum_{j=1}^n \tilde\theta_{j,k}
\end{align}
for each $k$. These three step can be iterated until convergence.

When combining the EM algorithm with LC-IC, \eqref{eq: posterior prob EM}
 and \eqref{eq: mix prop EM} remain unchanged. Only the maximum likelihood step \eqref{eq: weighted ML max EM} needs to be replaced, where one needs to ``fit" a density to samples weighted by $\tilde\theta_{jk}$.
 Since weighted samples are not immediately compatible with our sample splitting scheme, we propose the following heuristic re-sampling scheme.
From the observed samples $X^{(1)}, X^{(2)}, \dots, X^{(n)}$, construct the weighted empirical distribution 
\begin{align*}
     \tilde\nu_k := \frac{\sum_{j=1}^n \tilde\theta_{jk}\, \delta_{X^{(j)}}}{\sum_{j=1}^n \tilde\theta_{jk}} 
\end{align*}
for each $k$. Notice that $\tilde\nu_k$ depends on the current iterates. Next, (computationally) generate iid samples from $\tilde\nu_k$ (hence the term \emph{re-sampled}), and feed them into \Cref{alg: main algo}, using $M$ of those re-sampled samples for the unmixing matrix estimation stage, and $N$ for the marginal estimation stage. The output of \Cref{alg: main algo} is then the updated iterate of $\tilde p_k$. Unlike the sample splitting scheme discussed before, the re-sampling heuristic involves no compromise between $M$ and $N$; one can generate as many samples from $\tilde\nu_k$ as they like. The experiments in \Cref{subsec: mixtures} used $N = M = 4n$. 

\begin{figure}[!b]
    \centering
    \includegraphics[width=0.9\textwidth]{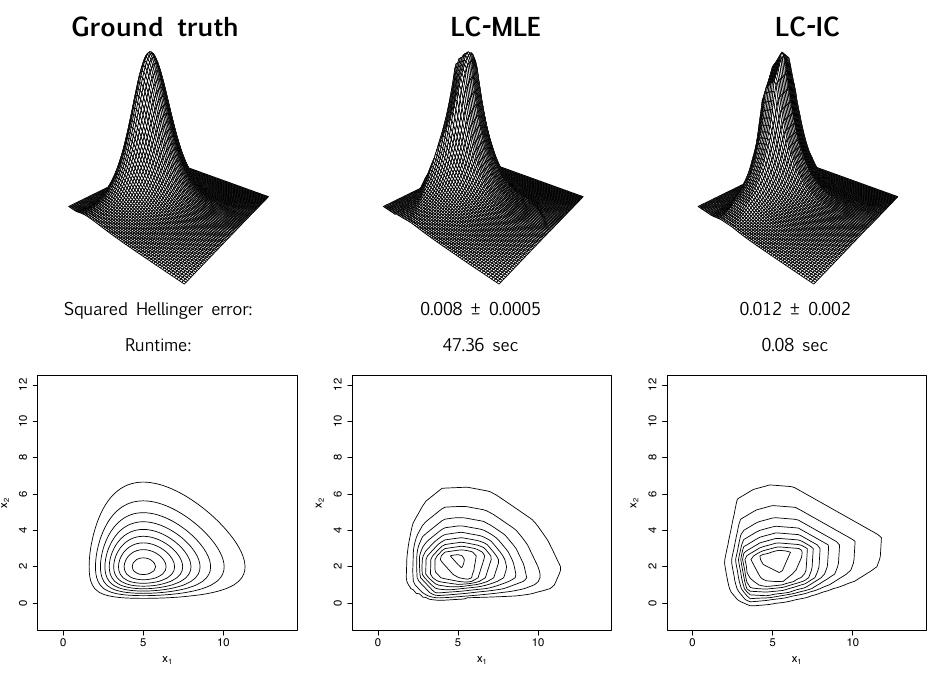}
    \caption{Comparing density estimation performance of LC-IC with LC-MLE on simulated Gamma-distributed data in $d=2$. The ground truth and estimated densities are visualized in the top row, and the corresponding level curves are plotted in the bottom row.}
    \label{fig:compare-visually-gamma}
\end{figure}

However, it is important to note that we do not have a theoretical backing for this heuristic yet. Another issue in this strategy is the generation of spuriously zeros values of $\tilde\theta_{jk}$, which can cause division-by-zero errors and likelihoods of $-\infty$. We address this by adding a small positive number such as $10^{-7}$ to the zero values of  $\tilde\theta_{jk}$.

\subsection{Comparison of performance on Gamma-distributed data}
\label{subsec: gamma comparison tests}
Here, we repeat the experiments from \Cref{subsec: gaussian comparison tests}, but for Gamma distributed data. Consider the bivariate case first. Let $X = \eigmat^T Z$ with $\eigmat = \mathbf{I}$ and $Z=(Z_1, Z_2)$, such that $Z_1 \sim \mathrm{Gamma}(6,1)$ and $Z_2 \sim \mathrm{Gamma}(3,1)$. With $n=1000$ iid samples of $X$ generated, equally split $M=N=500$, and compute the LC-IC estimate using PCA for unmixing matrix estimation and \texttt{logcondens} for marginal estimation. From the same $n=1000$ samples, also compute the LC-MLE estimate with \texttt{LogConcDEAD}. \Cref{fig:compare-visually-gamma} shows the ground truth and estimated densities together with their contour plots, and lists the estimation errors in squared Hellinger distance along with the algorithm runtimes.

Then, we conduct systematic tests similar to \Cref{subsec: gaussian comparison tests} for $d=2,3,4$ and various $n$. Set $X = \eigmat^T Z$, where $\eigmat = \mathbf{I}$ and $Z = (Z_1, \dots, Z_d)$, with $Z_i \sim \mathrm{Gamma}(6 - (i-1), 1)$. PCA is used for estimating $\eigmat$. Each experiment is repeated 5 times, and the average metrics are computed. \Cref{fig:gamma-curves} plots the estimation errors in squared Hellinger distance, as well as the algorithm runtimes, for both LC-MLE and LC-IC. The results are consistent with those from \Cref{subsec: gaussian comparison tests}.

\begin{figure}[h]
    \centering
    \includegraphics[width=\textwidth]{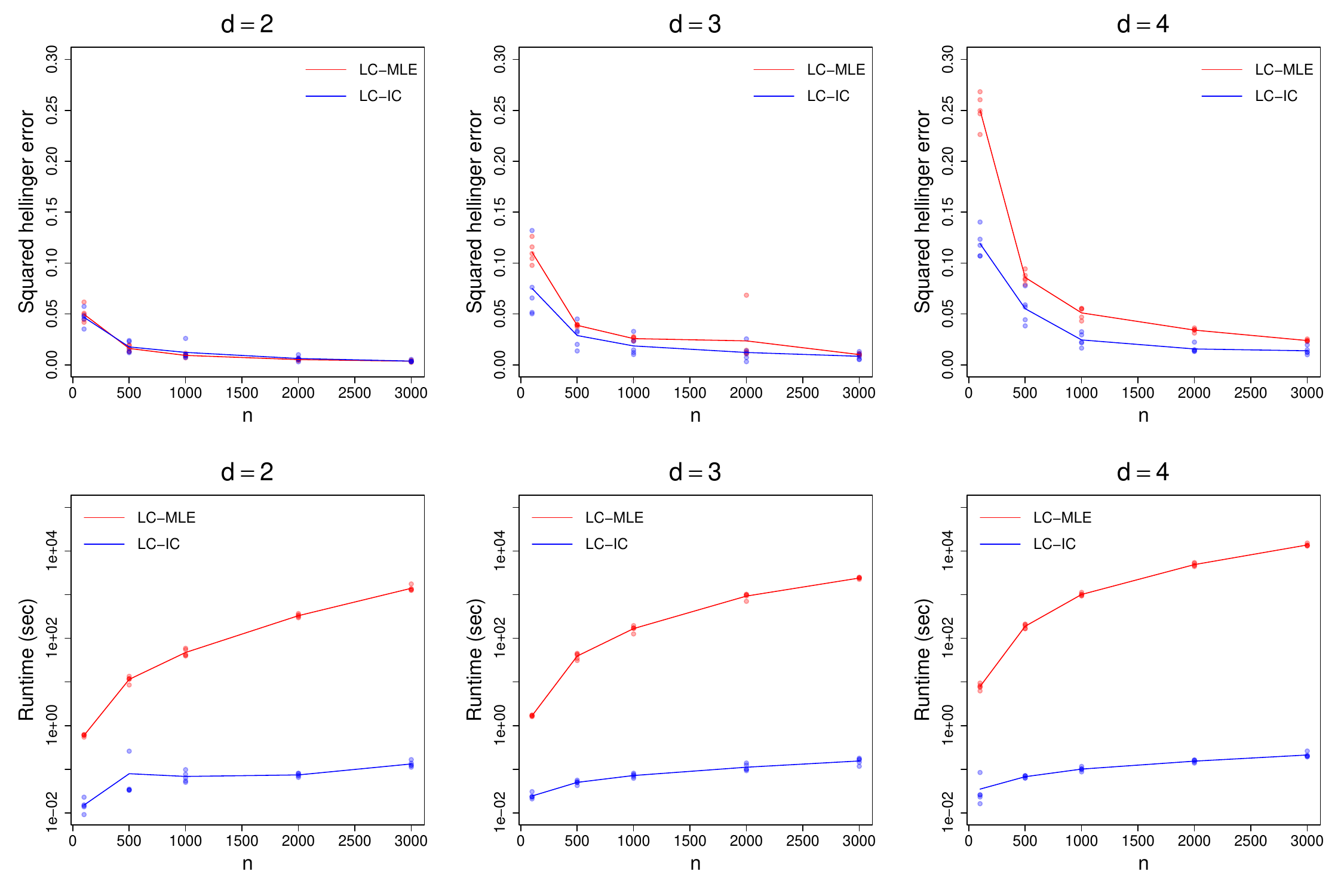}
    \caption{Comparing estimation errors in squared Hellinger distance (top row) and algorithm runtimes (bottom row) of LC-IC with LC-MLE, on simulated Gamma distributed data. The dots correspond to independent experiments, and the lines represent averages. Here, $n$ is the number of samples, and $d$ is the dimension.}
    \label{fig:gamma-curves}
\end{figure}

% Finally, we test LogConICA initialized with LC-IC on simulated $\mathrm{Unif}([-1,1]^d)$ data. The results in \Cref{fig: LogConICA init} show this scheme to have the best performance.

\end{document}